\documentclass[12pt,twoside]{article}

\usepackage{amsmath,amsfonts,amsthm,amssymb,amsopn,amstext,amscd}
\usepackage{enumerate}
\usepackage{graphicx}
\usepackage{float}
\usepackage{xcolor}
\usepackage{longtable}
\usepackage{multicol}
\usepackage{multirow}
\usepackage[font=footnotesize]{caption}
\usepackage{setspace} 
\usepackage{enumitem}
\usepackage{hyperref}
\usepackage{dsfont}
\allowdisplaybreaks

\newcommand{\mbN}{\mathbb N}
\newcommand{\mbP}{\mathbb P}

\newcommand{\mbE}{\mathbb E}

\newcommand{\N}{\mathbb{N}}
\newcommand{\R}{\mathbb{R}}

\newcommand{\V}{\text{Var}}
\newcommand{\T}{\boldsymbol{\theta}}

\newcommand{\supp}{\mathrm{supp}}
\newcommand{\eps}{\varepsilon}

\newcommand{\cuad}{\begin{flushright}\vspace{-2ex}$\Box$\vspace{-2ex}\end{flushright}}

\cuad

\newtheorem{thm}{Theorem}
\newtheorem{prop}{Proposition}

\newtheorem{lem}{Lemma}

\newtheorem{ass}{Assumption}

\newtheorem{remark}{Remark}

\newcommand{\vc}[1]{\boldsymbol{#1}}
\usepackage{dsfont}
\newcommand{\ind}[1]{\mathds{1}_{\{#1\}}}

\newcommand{\red}[1]{{\color{black}#1}}
\newcommand{\blue}[1]{{\color{black}#1}}

\newcommand*\samethanks[1][\value{footnote}]{\footnotemark[#1]}

\begin{document}
\pagenumbering{arabic}

\title{Consistent least squares estimation in population-size-dependent branching processes}

\author{Peter Braunsteins\thanks{All authors contributed equally to this work.} \thanks{School of Mathematics and Statistics, UNSW Sydney, NSW, 2052. email: \url{p.braunsteins@unsw.edu.au }, ORCID: 0000-0003-1864-0703}, \ Sophie Hautphenne\samethanks[1] \thanks{School of Mathematics and Statistics, The University of Melbourne, 3010 Parkville, VIC, Australia. email: \url{sophiemh@unimelb.edu.au }, ORCID: 0000-0002-8361-1901}, \
Carmen Minuesa\samethanks[1] \thanks{Departamento de Matem\'aticas, Instituto de Computaci\'on Cient\'ifica Avanzada de la Universidad de Extremadura (ICCAEx), Facultad de Ciencias, Universidad de Extremadura  (ROR: \url{https://ror.org/0174shg90}), 06006 Badajoz, Spain. email: \url{cminuesaa@unex.es}, ORCID: 0000-0002-8858-3145.}} 

\date{}

\maketitle

\begin{abstract}
We derive the first conditionally consistent estimators for a class of parametric Markov population models with logistic growth, which are suitable for modelling endangered populations in restricted habitats with a carrying capacity. We focus on discrete-time parametric population-size-dependent branching processes, for which we propose a new class of weighted least-squares estimators based on a single trajectory of population size counts. We establish the consistency and asymptotic normality of our estimators, conditional on non-extinction up to time $n$, as $n\to\infty$.  Since Markov population models with a carrying capacity become extinct almost surely under general conditions, our proofs rely on arguments distinct from those in the existing literature.

Our results are motivated by conservation biology, where endangered populations are often studied precisely because they are still alive, leading to an observation bias. Through simulated examples, we show that our conditionally consistent estimators generally reduce this bias
for key quantities such as a habitat's carrying capacity. We apply our methodology to estimate the carrying capacity of the Chatham Island black robin, a population reduced to a single breeding female in the 1970's, which has since recovered but has yet to reach the island's carrying capacity.
\end{abstract}

\noindent {\bf Keywords: }{branching process; population-size-dependence; almost sure extinction; inference; carrying capacity; least-squares estimation; consistency.}

\noindent {\bf MSC 2020: }{60J80.}

\vspace{1cm}


\section{Introduction}
\label{sec:intro}


Biological populations generally display \emph{logistic growth}:
after an initial period of exponential growth, competition for limited resources such as food, habitat, and breeding opportunities, causes population growth to slow down until the population eventually reaches a \emph{carrying capacity}, which is the maximum population size that a habitat can support.
Biologists often aim to estimate the carrying capacity of the habitat for the population they are studying.
One example is the \textit{Chatham Island black robin population} (New Zealand), which
was saved from the brink of extinction in the early 1980s when it was considered the world's most endangered bird species \cite{elliston1994black}.
The population has since recovered from a single breeding female to over 120 females, and appears to be nearing carrying capacity of its habitat (Rangatira Island)
\cite{massaro2018post,massaro2013nest}.
In order to estimate the carrying capacity, ecologists often rely on assessing environmental characteristics of the habitat itself.
For example, in \cite{massaro2018post} the carrying capacity of the black robin population is estimated to be 170 nesting pairs based on a \emph{habitat suitability model} which requires  precise knowledge of the habitat and is very costly.
Here we take a different statistical approach to estimate the carrying capacity based solely on annual population size counts using  stochastic population models.

In this paper we model biological populations with discrete-time \textit{population-size-dependent branching processes} (PSDBPs), which  can capture fundamental properties of many populations (including logistic growth). PSDBPs are Markov chains $\{Z_n\}_{n\geq 0}$
characterised by the recursion
\begin{equation}\label{def:process}
Z_{n+1} =
\sum^{Z_{n}}_{i=1} \xi_{n,i}(Z_{n}), \qquad n\in\N_0:=\{0,1,2,\dots\},
\end{equation}
where $Z_n$ represents
the population size at time $n$ and $\{\xi_{n,i}(\cdot)\}_{n,i}$ are independent random variables that represent each individual's offspring whose distribution depends only on the current population size $Z_n$.
In these processes, state 0 is absorbing, and represents \textit{extinction} of the population.
Let $m(z):=\mathbb E[\xi(z)]$ denote the mean offspring at population size $z$; we say that a PSDBP has a carrying capacity  $K$ if  \begin{equation}\label{defK}m(z)>1\quad \text{when $z<K$}\quad \text{and}\quad m(z)<1 \quad \text{when $z>K$}.\end{equation} When such a threshold value $K$ exists, the PSDBP becomes extinct with probability one; this is common to almost all stochastic population models with a carrying capacity \cite{jagers2020populations}. We point out that, while closely related, the mathematical definition of the carrying capacity in \eqref{defK} is not the same as the biological definition given above.

The fact that PSDBPs with a carrying capacity become extinct with probability one introduces several technical challenges. For example, if we consider a parametric PSDBP, to establish
asymptotic properties of an estimator $\widehat{\vc \theta}_n$ for a parameter $\vc\theta$ based on a single trajectory, we need to condition on non-extinction up to a finite time $n$ ---an event with vanishing probability as $n\to\infty$. Consequently, unlike branching models with positive chance of survival, many natural estimators (such as maximum likelihood estimators) do not satisfy the classical consistency property called \textit{$C$-consistency}, defined as follows: an estimator $\widehat{\vc \theta}_n:=\widehat{\vc \theta}_n(Z_0,\ldots,Z_n)$ is called $C$-consistent for a parameter $\vc\theta$ if,
\begin{equation}\label{C-cons}\textrm{for any $\varepsilon > 0$,}\qquad
\lim_{n \to \infty} \mathbb{P}( |\widehat{\vc \theta}_n - \vc\theta | > \varepsilon \, | \, Z_n >0)=0
\end{equation}
(see for instance \cite{Pakes-1975}).
 In \cite[Theorem 1]{braunsteins2022parameter} the authors show that the maximum likelihood estimators (MLEs) $\hat{m}_n(z)$ (defined in \eqref{MLEmz}) for $m(z)$ in a PSDBP  are not $C$-consistent but instead are such that,
\begin{equation}\label{Qcons}
\textrm{for any $\varepsilon > 0$,}\qquad\lim_{n \to \infty} \mathbb{P}( |\hat{m}_n(z) -m^\uparrow(z) | > \varepsilon \, | \, Z_n >0)=0,\end{equation}where $m^\uparrow(z)$ is different from $m(z)$, and can be interpreted as the mean number of offspring
born to individuals when the current population size is $z$ in the \emph{$Q$-process} associated with $\{Z_n\}$, which corresponds to $\{Z_n\}$ conditioned on non-extinction in the
distant future (see Section \ref{sec:Q-process}). In \cite{braunsteins2022parameter} the authors refer to \eqref{Qcons} as \emph{$Q$-consistency}.

A major gap remains to be filled in the context of PSDBPs with a carrying capacity: \emph{there does not exist any $C$-consistent estimator for $m(z)$ or for parameters (such as the carrying capacity) in parametric models}. We point out that $C$-consistent estimators are lacking not only for PSDBPs but for any stochastic population process with a carrying capacity and an absorbing state at 0 (for example, the diffusion models in \cite{lande2003stochastic}, the continuous-time birth-death processes in \cite{hautphenne2021birth}, and the controlled branching processes in \cite{gonzalez2021model}).

In this paper, we derive the first \emph{$C$-consistent} estimators $\widehat{\vc \theta}_n$ for any set of parameters $\vc \theta$ of \emph{parametric} PSDBPs (under some regularity conditions). In particular, this leads to the first $C$-consistent estimators for the carrying capacity $K$ \blue{(provided $K$ is a model parameter, or a function of other parameters).}
 The key idea behind our $C$-consistent estimators is to embed the $Q$-consistent MLEs $\hat{m}_n(z)$ and their conditional limit ${m}^{\uparrow}(z)$  in the objective function of a weighted least squares estimator,
\begin{equation}\label{Asss2} 
\widehat{\vc \theta}_n:=\arg\min_{\T\in\Theta} \sum_{z\geq 1} \hat{w}_n(z) \left\{\hat{m}_n(z)-{m}^{\uparrow}(z, \T)\right\}^2,
\end{equation}
with weights $ \hat{w}_n(z):= \hat{w}_n(z;Z_0,\dots, Z_n)$.
As detailed in Section \ref{sec:wlse}, for an appropriate choice of weights, this estimator
is the same as
\begin{equation}\label{cwlse_m}
\widehat{\T}_n^*:=\arg\min_{\T \in \Theta} \sum_{k =1}^{n} w_k \left\{Z_k - Z_{k-1}\,m^\uparrow(Z_{k-1}, \T)\right\}^2,
\end{equation}
which is a modification of the classical least squares estimators for stochastic processes, where the sum is taken over the successive times rather than the population sizes  (see for instance \cite{klimko1978conditional}).
%
To establish $C$-consistency and asymptotic normality of $\widehat{\vc \theta}_n$, we build on the asymptotic properties of $\hat{m}_n(z)$ and overcome several new technical challenges which arise due to the almost sure extinction of the process and the fact that $ \widehat{\vc \theta}_n$ includes an infinite sum.

Besides being of theoretical value, the $C$-consistent estimators $\widehat{\vc \theta}_n$ are also practically relevant. Indeed, many populations, such as the Chatham Island black robins, are studied precisely because they are still alive and conservation measures may need to be taken to preserve them. In these situations, the data should be viewed as being generated under the condition $Z_n>0$. This suggests that for many populations (particularly endangered populations), $C$-consistent estimators may be desirable.
We demonstrate the quality of our $C$-consistent estimators on simulated data in two contexts: \emph{(i)} on (quasi-)stable populations that are fluctuating around the carrying capacity, and \emph{(ii)} on growing populations that are yet to reach carrying capacity. In particular, we compare our  $C$-consistent estimators $\widehat{\vc \theta}_n$ with the
classical least squares estimators (which are not $C$-consistent). We find that when the data include small population sizes, where the population is at high risk of extinction (such as for the black robins), our $C$-consistent estimators $\widehat{\vc \theta}_n$ have a smaller bias than classical estimators. However when the data do not include these small population sizes, our estimators and classical estimators give similar estimates.
Finally,
we apply our methods to estimate the carrying capacity of the black robin population on Rangatira Island under several PSDBP models.

The paper is organised as follows. In the next section, we provide some background on PSDBPs and their associated $Q$-process, and we introduce a flexible class of parametric PSDBPs which will serve as an illustrative example. In Section~\ref{sec:wlse} we introduce our weighted least squares estimators $\widehat{\vc \theta}_n$ and in Section~\ref{sec:as_prop} we state their asymptotic properties. In Section~\ref{empirical}, we compare our least squares estimators $\widehat{\vc \theta}_n$ with their counterpart obtained by replacing ${m}^{\uparrow}(z, \T)$ with ${m}(z, \T)$ in \eqref{Asss2}, using simulated numerical examples. In Section~\ref{sec:br} we estimate the carrying capacity of the black robin population. 
\blue{We conclude the paper in Section~\ref{disc}, where we discuss the choice of weights, the situations in which the $C$-consistent estimator may be preferable, and potential extensions of the methodology.}
\red{The proofs are gathered in Section \ref{sec:proofs} and Appendix~\ref{appA}, \blue{and additional simulation results} are presented in Appendix~\ref{appB}.}

\section{Background}

\subsection{Parametric families of PSDBPs}

A discrete-time \emph{population-size-dependent branching process} (PSDBP) is a Markov chain $\{Z_n\}_{n\in\N_0}$ characterised by the recursion given in \eqref{def:process} with $Z_0=N$, where
$N$ is a positive integer, and where $\{\xi_{ni}(z):i=1,\ldots,z;\ z,n\in\N_0\}$ is a family of independent random variables
with distribution $\boldsymbol{p}(z)=(p_k(z))_{k\in\N_0}$ such that $p_k(z)=\mathbb{P}[\xi_{01}(z)=k]$ \emph{depends on~$z$}.
We refer to $\boldsymbol{p}(z)$ as the \emph{offspring distribution at population size $z$}.
In $\{Z_n\}$ the state $0$ is absorbing, and under minor regularity assumptions, all other states are transient \cite{jagers92}.

In this paper we
assume
that the \emph{family of offspring distributions} $(\boldsymbol{p}(z))_{z \in \mathbb{N}}$ belongs to some parametric family, that is, we assume that there exists
\[
 \T_0\in int(\Theta)\subseteq\R^d \qquad \text{such that} \qquad \boldsymbol{p}(z)\equiv\boldsymbol{p}(z,\T_0), \quad  \text{for each} \quad z \in \mathbb{N},
\]
where $int(\Theta)$ denotes the interior
of the \emph{parameter space} $\Theta$, and $d<\infty$.
As a consequence the PSDBP is parameterised by $\T_0$.
We assume that the map $\T \mapsto \boldsymbol{p}(z, \T)$ is a measurable function such that
$p_k(z,\cdot)$ is continuous on $\Theta$ for each $k \in \mathbb{N}_0$ and $z \in \N$, and is such that
the offspring mean $m(z)\equiv m(z,\T)$ and variance $\sigma^2(z)\equiv \sigma^2(z,\T)$ at population size $z$ are continuous on $\Theta$ for each $z \in \mathbb{N}$.
In this framework, our aim is to estimate the parameter $\T_0$ based on the sample of the population sizes $\mathcal{Z}_n:=\{Z_0,\ldots,Z_n\}$, \blue{where $Z_n>0$.}

\subsection{The $Q$-process associated with $\{Z_n\}$}\label{sec:Q-process}

Let $\textbf{Q}=(Q_{ij})_{i,j \in \N}$ refer to the sub-stochastic transition probability matrix of $\{Z_n\}_{n\in\N_0}$ restricted to the transient states $\{1,2,\ldots\}$.
The entry $Q_{ij}$ is the probability that the population transitions from size $i$ to size $j$ in a single time unit. The $i$th row of $\textbf{Q}$ thus corresponds to the $i$-fold convolution of the offspring distribution at population size $i$.
Because the family of offspring distributions $(\boldsymbol{p}(z))_{z \in \mathbb{N}}$ depends on $\T_0$, so does the transition probability matrix, and consequently, $\textbf{Q}\equiv \textbf{Q}(\T_0)$, where this shorthand notation extends naturally to any $\T \in \Theta$.

 For every $n\in\N_0$ fixed, the PSDBP $\{Z_\ell\}_{0\leq \ell\leq n}$ conditioned on $Z_{n}>0$ \blue{(implied by the sampling scheme)} is a \textit{time-inhomogeneous} Markov chain that we denote by $\{Z_\ell^{(n)}\}_{0\leq \ell\leq n}$ and whose one-step transition probabilities are
\begin{align}\label{condProb}
\mbP(Z_{\ell+1}^{(n)}=j\,|\,Z_{\ell}^{(n)}=i):=\mbP(Z_{\ell+1}=j\,|\,Z_{\ell}=i,\;Z_{n}>0)=Q_{ij}\,\dfrac{\vc e_j^\top \textbf{Q}^{n-\ell-1} \vc 1}{\vc e_i^\top \textbf{Q}^{n-\ell}\vc 1}\quad \text{ for }i,j\geq 1.
\end{align} \blue{A fundamental role will be played by the limit of this time-inhomogeneous Markov chain as $n\to\infty$, which we consider under the following regularity conditions (assumed to hold throughout the paper)}:
\begin{ass} \label{Asss1}
For each $\T\in\Theta$,
\begin{enumerate}[label=(C\arabic*),ref=(C\arabic*),start=1]
\item The matrix $\textbf{Q}(\T)$ is irreducible.\label{cond:irreducible-PSDBP}
\item $\limsup_{z\to\infty} m(z,\T)<1$.\label{cond:lim-sup-m-z}
\item For each $\nu\in\N$, $\sup_{z\in\N} \sum_{k=1}^\infty k^\nu p_k(z,\T)<\infty$\label{cond:bounded-moments-xi}.
\end{enumerate}
\end{ass}

Under Assumptions \ref{cond:irreducible-PSDBP}--\ref{cond:bounded-moments-xi}, the PSDBP becomes extinct almost surely for any $\T \in \Theta$ and initial state $Z_0=i$ (see \cite[Proposition 3.1]{Gosselin-2001}), that is, $\mbP_i(Z_n\to 0)=1$, where $\mbP_i(\cdot)$ denotes the probability measure given the initial state is $i$.
In addition, there exists $\rho$, $\vc u$, and $\vc v$, where $\rho:=\rho(\textbf{Q})=\lim_{n\rightarrow\infty} (\textbf{Q}^n)_{ij}^{1/n}$ \blue{is the convergence norm of $Q$}, and $\vc u$ and $\vc v$ are strictly positive column vectors such that
\begin{eqnarray}\label{uv}\vc u^\top \textbf{Q}=\rho \vc u^\top, \quad \textbf{Q}\vc v=\rho \vc v, \quad \vc u^\top\vc 1=1,\quad \textrm{and}\quad \vc u^\top\vc v=1,\end{eqnarray}
and
\begin{equation}\label{Qn}
\textbf{Q}^n\sim\rho^n \vc v\vc u^\top, \quad\textrm{as $n\rightarrow\infty$};
\end{equation}
see for instance \cite[\blue{Lemma 6.1}]{Gosselin-2001}.
\blue{The vector $\vc u$ corresponds to the \emph{quasi-stationary} (or \emph{quasi-limiting}) distribution of $\{Z_n\}$, while the vector $\vc v$ captures the relative ``strength'' of each state. Indeed, using \eqref{Qn}, and letting $\vc e_i$ denote a column vector where the $i$th entry is 1 and all other entries are 0,
$$\lim_{n\rightarrow\infty}\mbP_i(Z_n=j\,|\,Z_{n}>0)=\lim_{n\rightarrow\infty}\dfrac{\vc e_i^\top \textbf{Q}^n \vc e_j}{\vc e_i^\top\textbf{Q}^{n}\vc 1}=\lim_{n\to\infty}\dfrac{\rho^n v_i u_j}{\rho^n v_i}=u_j,\qquad j\geq 1,$$and
$$\lim_{n\rightarrow\infty}\dfrac{\mbP_j(Z_{n}>0)}{\mbP_i(Z_{n}>0)}=\lim_{n\rightarrow\infty}\dfrac{\vc e_j^\top \textbf{Q}^n \vc 1}{\vc e_i^\top \textbf{Q}^{n}\vc 1}=\dfrac{v_j}{v_i},\qquad i,j\geq 1.$$}

Taking the limit as $n\rightarrow\infty$ in \eqref{condProb} and using \eqref{Qn} leads to homogeneous transition probabilities:
\begin{align}\label{tpZh}
\mbP(Z_{\ell+1}^\uparrow=j\,|\,Z_{\ell}^\uparrow=i) := \lim_{n\rightarrow\infty}\,\mbP(Z_{\ell+1}^{(n)}=j\,|\,Z_{\ell}^{(n)}=i) = \lim_{n\rightarrow\infty}Q_{ij}\,\dfrac{\vc e_j^\top \rho^{n-\ell-1} \,\vc v }{\vc e_i^\top \rho^{n-\ell} \,\vc v } =Q_{ij}\,\dfrac{v_j}{\rho v_i}.
\end{align}
These transition probabilities define a new positive-recurrent \emph{time-homogeneous} Markov chain $\{Z^{\uparrow}_\ell\}_{ \ell\geq 0}$, which we refer to as the \emph{$Q$-process} associated with $\{Z_n\}$. {This process
can be interpreted} as the original process $\{Z_n\}$ conditioned on not being extinct in the distant future.
The $n$-step transition probabilities of the $Q$-process $\{Z^{\uparrow}_\ell\}$ are given by
$$({\textbf{Q}^{\uparrow}}^n)_{ij}:=\mbP(Z^{\uparrow}_n=j\,|\,Z^{\uparrow}_0=i)=(\textbf{Q}^n)_{ij} \dfrac{v_j}{\rho^n v_i },\quad \text{ for }n,i,j\geq 1,$$
and, by \eqref{Qn}, its stationary distribution is therefore
\begin{eqnarray}\label{stat_dis}
\lim_{n\rightarrow\infty} \mbP(Z^{\uparrow}_n=j\,|\,Z^{\uparrow}_0=i) &=& u_j v_j,\quad \text{ for }i, j\geq 1.
\end{eqnarray}
\blue{The stationary distribution of $\{Z^{\uparrow}_\ell\}$ appears in the asymptotic weights of the $C$-consistent least squares estimators (see Lemma \ref{weights}) and in their asymptotic variance (see Eq.\ \eqref{gamma}).}

For $z\in\N$ and $\T\in\Theta$, we introduce the analogue of the mean offspring at population size $z$, $m(z,\T)$, in the $Q$-process $\{Z^{\uparrow}_\ell\}$,
\begin{align}\label{mup}
m^\uparrow(z,\T)&:=\frac{1}{z}\mbE[Z^\uparrow_{n+1,\T}|Z^\uparrow_{n,\T}=z]=\frac{1}{z}\sum_{k=1}^\infty k\,Q_{zk}^\uparrow(\T),
\end{align}
\blue{which plays a key role in the definition of the new $C$-consistent estimators \eqref{Asss2}}.
Similarly, we introduce the analogue of the normalised offspring variance,
\begin{align}
\sigma^{2\uparrow}(z,\T)&=\frac{1}{z^2}\text{Var}[Z^\uparrow_{n+1,\T}|Z^\uparrow_{n,\T}=z]\label{varQ}\\\nonumber
&=\frac{1}{z^2}\left[\sum_{k=1}^\infty k^2Q_{zk}^\uparrow(\T)-\left(\sum_{k=1}^\infty k Q_{zk}^\uparrow(\T)\right)^2\right]=\left[\frac{1}{z^2}\sum_{k=1}^\infty k^2Q_{zk}^\uparrow(\T)\right]-m^\uparrow(z,\T)^2,
\end{align}
\blue{which appears in the asymptotic variance of the $C$-consistent estimators (see Eq.\ \eqref{gamma}).}

\subsection{A motivating example}\label{ex:motivation}

In this section we introduce a flexible class of PSDBPs which we will return to in our empirical analysis in Section \ref{empirical}. We consider models in which, every time unit, each individual successfully reproduces with a probability $r(z,\T)$ that depends on the current population size $z$, and if reproduction is successful, then the individual produces a random number of offspring with distribution $\vc b$ which does not depend on the current population size.
More precisely, we
consider a parametric class of PSDBPs with offspring distributions $\{\boldsymbol{p}(z)\}_{z \in \N}$ which are characterised by
\begin{align}\label{pj}
p_j(z,\T)=(1-r(z,\T))\ind{j=0}+r(z,\T)\, b_j(\mu),\qquad \qquad j\in\N_0,
\end{align}
where $z\in\N$, $\boldsymbol{b}(\mu)=\{b_j(\mu)\}_{j\geq 0}$ is a distribution on the non-negative integers with mean $\mu>1$,
$\T=(K,\mu)\in\R_+\times (1,\infty)$, and $0\leq r(z,\T)\leq 1$.
In other words, individuals give birth according to a zero-inflated distribution, where $1-r(z,\T)$ is the increase in the probability that individuals have no offspring.
We can select the function $r(\cdot)$ so that the PSDBP mimics well-established population models:
\begin{eqnarray}\label{bh}
r(z,\T)&=&\displaystyle\frac{K}{K+(\mu-1) \,z},\quad \text{ \emph{(Beverton-Holt model)}}\\
\label{ricker}
r(z,\T)&=&\displaystyle\left(\dfrac{1}{\mu}\right)^{z/K},\quad \quad \quad\;  \text{ \emph{(Ricker model).}}
\end{eqnarray}
For these choices of $r(\cdot)$,
\[
m(z, \T) > 1 \quad\text{ if } z < K, \quad \text{ and } \quad m(z, \T) < 1 \quad \text{ if } z>K,
\]
and consequently $K$ is the {carrying capacity} of the models. If the distribution $\boldsymbol{b}(\mu)$ is aperiodic and has finite moments, then Assumption \ref{Asss1} can be readily verified for the models described above.

To gain a better understanding of the Beverton-Holt and Ricker models we conduct a simulation study. We suppose the distribution $\boldsymbol{b}(\mu)$ is geometric. For each model,
 we fix the parameter $\T_0=(K_0,\mu_0)=(30,2)$, we suppose that the population starts with $N=1$ individual, and we simulate $N_0=250$ independent \emph{non-extinct} trajectories of the process for $n=5000$ time units, that is, such that $Z_{5000}>0$.
   The first 150 time units of one of these trajectories for each model is plotted in Figure \ref{fig:ex2:path}.
\begin{figure}
\centering\includegraphics[width=0.35\textwidth]{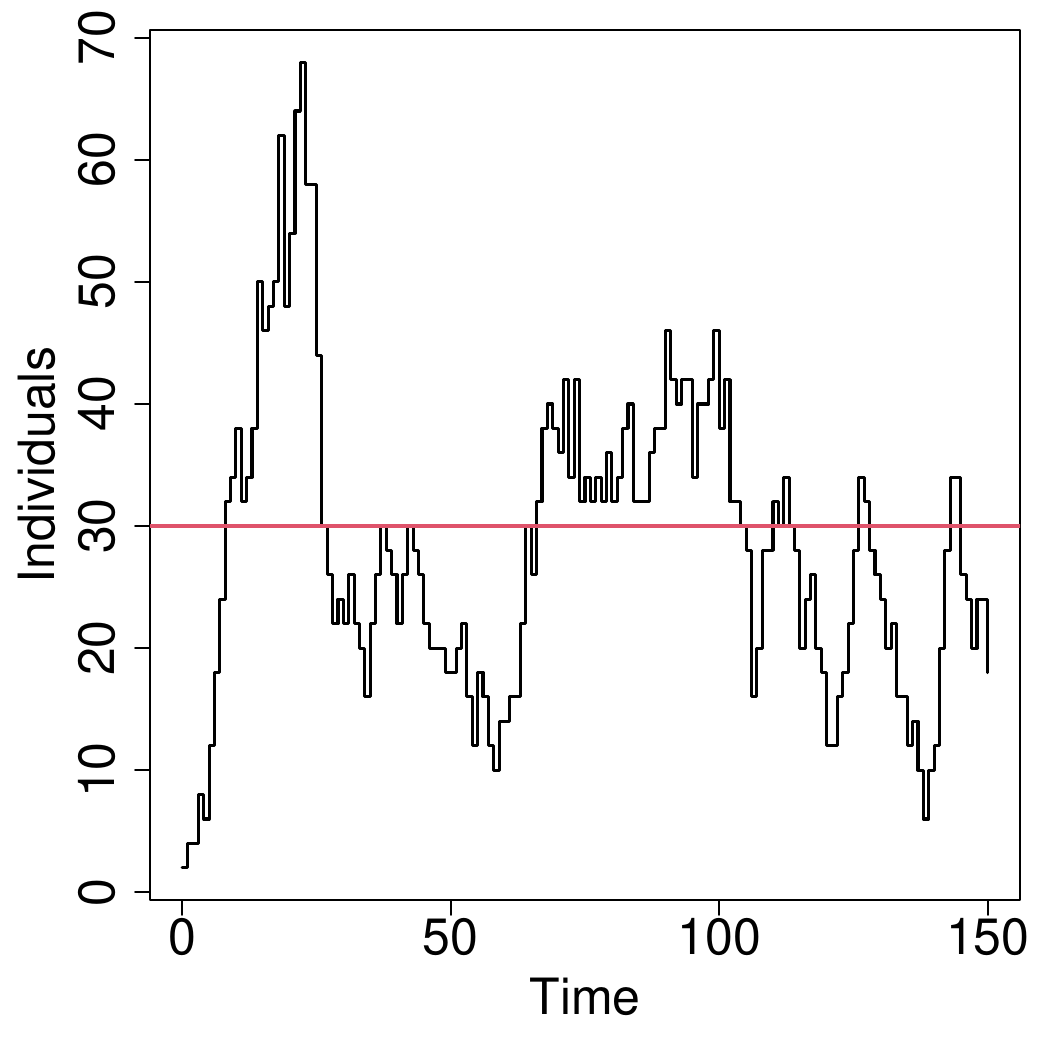}\hspace{0.5cm}\includegraphics[width=0.35\textwidth]{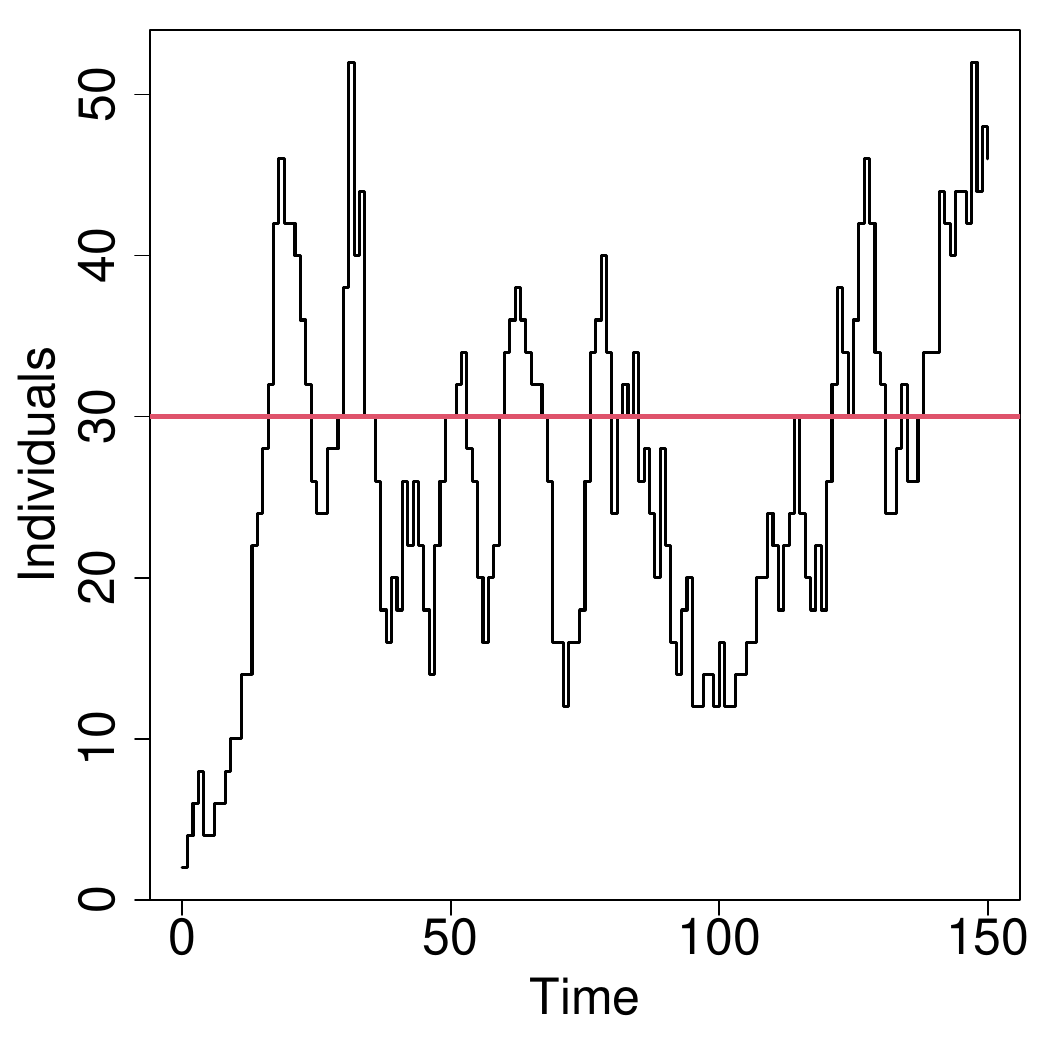}
\caption{Evolution of the population sizes of two PSDBPs over the first 150 time units. Left: Beverton-Holt model. Right: Ricker model. Red solid line represents the carrying capacity.}\label{fig:ex2:path}
\end{figure}
Suppose that, based on the simulated trajectories, we would like to estimate the parameters $\T_0$.
To this end we recall the MLE for $m(z, \T)$, the mean number of offspring at population size $z$,  derived in \cite[Proposition 1]{braunsteins2022parameter}:
\begin{equation}\label{MLEmz}
\hat{m}_{n}(z):=\frac{\sum_{i=0}^{n-1} Z_{i+1}\ind{Z_i=z}}{z\,\sum_{i=0}^{n-1}\ind{Z_i=z}},
\end{equation}where $\mathds{1}_A$ denotes the indicator function of the set $A$, and
 we take the convention that $0/0=0$.
In Figure \ref{fig:motivation-ex} we plot $m(z, \T_0)$, $m^\uparrow(z, \T_0)$ (defined in Eq.\ \eqref{mup}), and the means of $\hat{m}_n(z)$ for $n=5000$, $z\geq 1$, over the 250 trajectories, for the Beverton-Holt and the Ricker models.
Figure \ref{fig:motivation-ex} illustrates that the outcomes of $\hat m_n(z)$ cling more tightly to the function $m^\uparrow(z,\T_0)$ than to $m(z,\T_0)$, which is due to the bias caused by restricting our attention to \emph{non-extinct} trajectories.
In particular \cite[Theorem 1]{braunsteins2022parameter} states that, for any initial population size $i\in\N$ and population size $z \in \N$,
\begin{align}
\forall\varepsilon>0,\quad \lim_{n\to\infty}\mbP_i[|\hat{m}_{n}(z)-m^\uparrow(z,\T_0)|>\varepsilon|Z_n>0]&=0\qquad\qquad\textrm{($Q$-consistency)}.\label{eq:Q-consistency-m}
\end{align}
In Figure \ref{fig:motivation-ex} we see that the functions $z\mapsto m^\uparrow(z,\T_0)$ (red curve) and $z\mapsto m(z,\T_0)$ (blue curve) are relatively close to each other at moderate and high population sizes ($z\geq 10$), however, they are significantly different at low population sizes.
This motivates the weighted least squares estimators that we introduce in the next section.

\begin{figure}
\centering\includegraphics[width=0.35\textwidth]{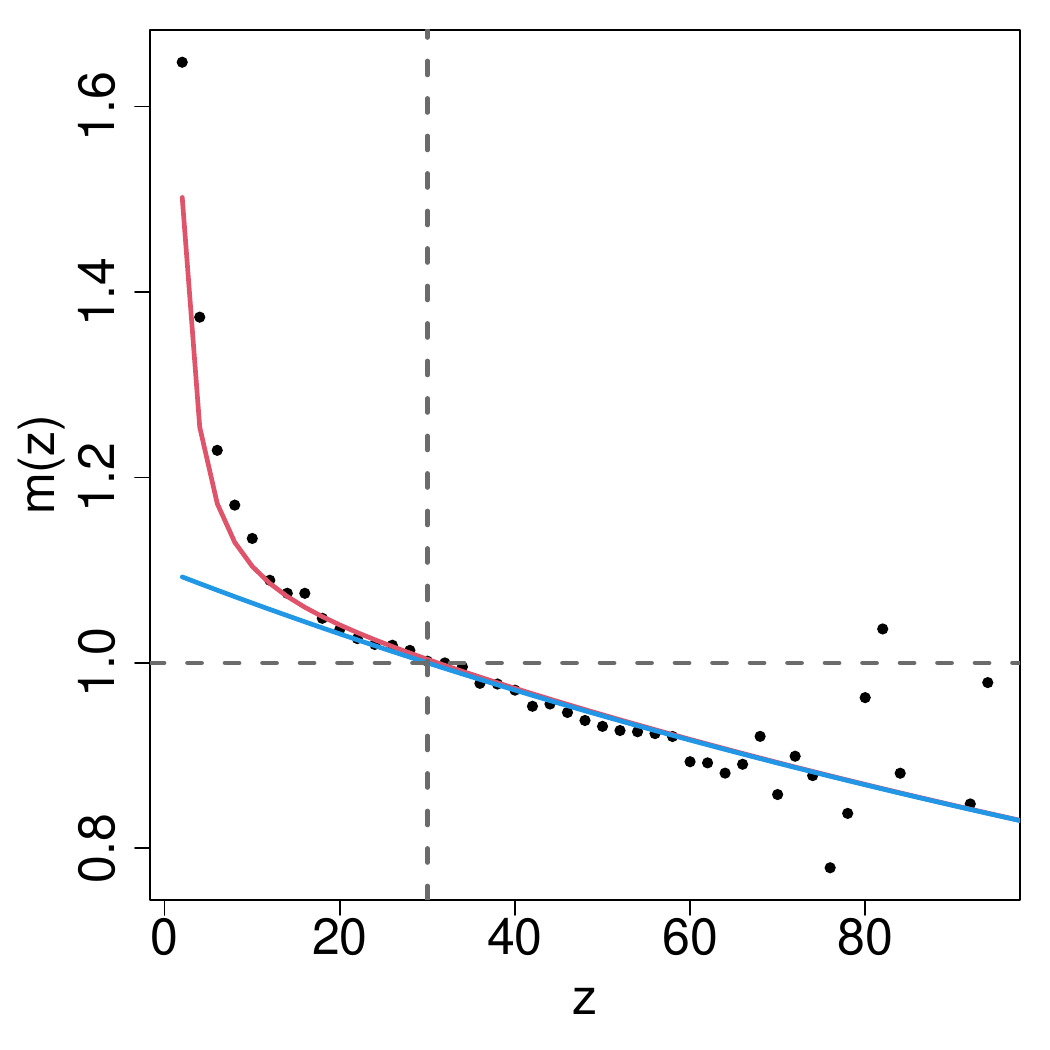}\hspace{0.5cm}\includegraphics[width=0.35\textwidth]{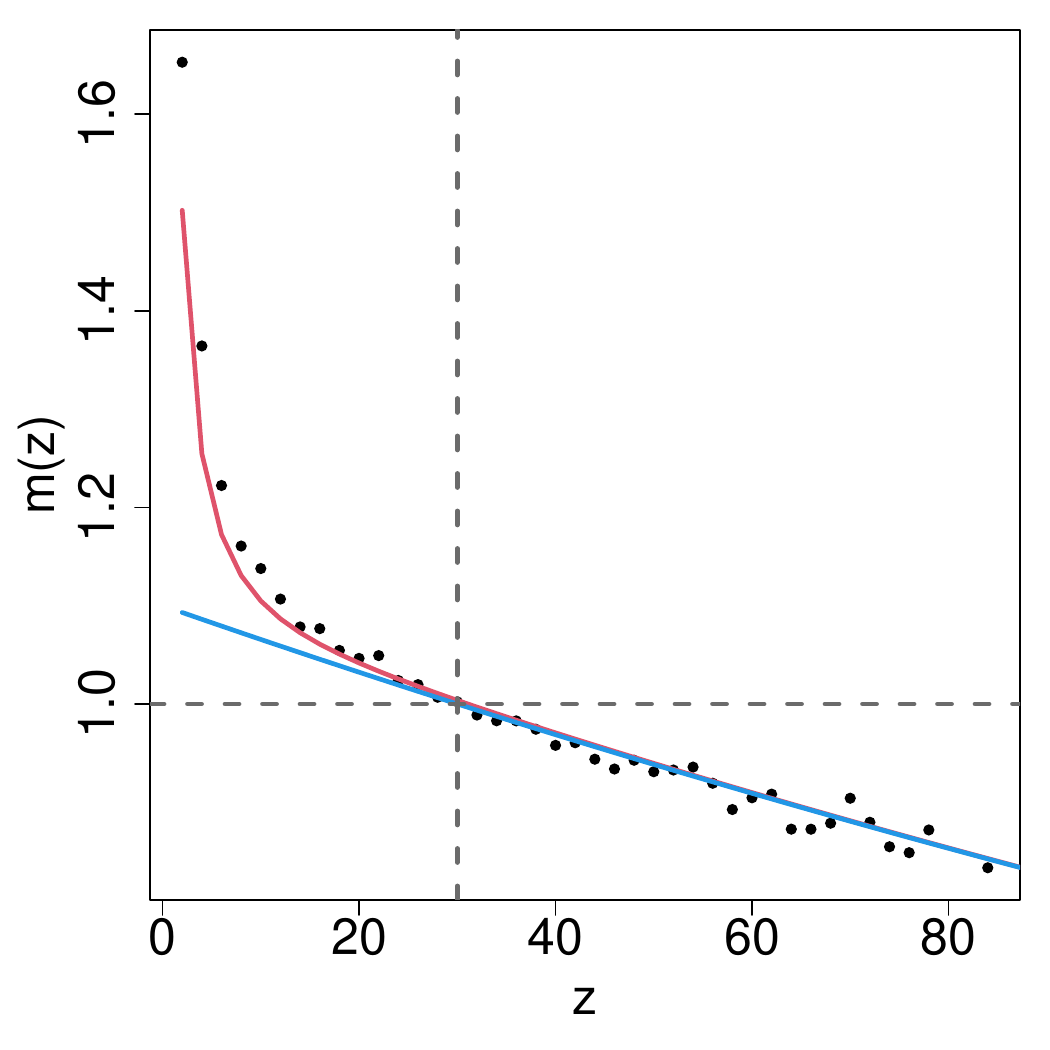}
\caption{Means of $\hat{m}_n(z)$ (calculated over \blue{$500$} simulated trajectories) for each observed population size $z$. Left: Beverton-Holt model. Right: Ricker model. The red solid curve represents the function $z\mapsto m^\uparrow(z,\T_0)$, while the blue solid curve represents the function $z\mapsto m(z,\T_0)$.}\label{fig:motivation-ex}
\end{figure}

\section{Weighted least squares estimators}\label{sec:wlse}

In the context of branching processes, classical least squares estimators
minimise the (weighted) squared differences between the population sizes $Z_k$ and their conditional mean $\mbE[Z_k\,|\, Z_{k-1}]=Z_{k-1}\,m(Z_{k-1},\T)$, for $k=1,\ldots, n$:
\begin{equation}\label{cwlse}
\widetilde{\T}_n^*:=\arg\min_{\T \in \Theta} \sum_{k =1}^{n} w_k\left\{Z_k - Z_{k-1}\,m(Z_{k-1}, \T)\right\}^2,
\end{equation}
where $\{w_k\}$ is an appropriately chosen weighting function, such as $w_k= Z_{k-1}^{-1}$ (see \cite[p.\ 40]{guttorp1991statistical}).
This approach has been applied to Galton--Watson branching processes \cite[Section~2.3]{guttorp1991statistical}, multi-type controlled branching processes \cite{art-2010}, branching processes with immigration \cite{klimko1978conditional}, and PSDBPs \cite{Lalam-Jacob-2004, Lalam-Jacob-Jagers-2004}.
Each of these papers develops asymptotic properties of $\widetilde{\T}_n^*$, such as consistency on the set of non-extinction, \emph{for processes with a positive chance of survival}.
Because we consider branching processes that \emph{become extinct almost surely}, we take a different approach.

To estimate $\T_0$, we \blue{propose a} class of least squares estimators that use the \textit{$Q$-consistent} MLEs $\{\hat m_n(z)\}_{z \in \N}$ and are \textit{$C$-consistent}, meaning they satisfy \eqref{C-cons}. The idea is that,
instead of summing over all time units $k=1,\ldots, n$ as in \eqref{cwlse}, we sum over all population sizes $z\geq 1$ and minimise the (weighted) squared differences between the MLEs $\hat{m}_n(z)$ and their conditional limits ${m}^{\uparrow}(z, \T)$:
\begin{equation*}
\widehat{\vc\theta}_n :=\arg\min_{\theta\in\Theta} \sum_{z=1}^\infty \hat{w}_n(z) \left\{\hat{m}_n(z)-{m}^{\uparrow}(z, \T)\right\}^2,
\end{equation*}
where the weights $\{\hat{w}_n(z)\}_{z\in \mbN}$ are computed from the observations $Z_0,Z_1,\ldots, Z_n$.
Estimators in this class can be viewed as hybrids between an MLE and a least squares estimator, where the MLEs $\{\hat m_n(z)\}_{z\in \mbN}$  first summarise the data, and the least squares step is then applied \blue{to the error between the MLE and its conditional limit, given $Z_n>0$ (noting that other least squares estimators can be interpreted similarly)}. In fact, as we show in Proposition \ref{wls_equiv}, for appropriate choices of weights, $\widehat{\vc\theta}_n$ corresponds to a modification of the classical estimator $\widetilde{\vc\theta}_n^*$ where $m(\cdot)$ is replaced by $m^\uparrow(\cdot)$.

We suppose that the weights in $\widehat{\vc\theta}_n$ 
satisfy the following assumption:
\begin{enumerate}[label=(A\arabic*),ref=(A\arabic*)]
\item
For each $n\in \mbN$, $\{\hat{w}_n(z)\}$ is an empirical distribution (i.e., for each $z\in\mbN$, $\hat{w}_n(z)\geq 0$, and $\sum_{z\geq 1}\hat{w}_n(z)=1$), and there exists a limiting distribution $\{w_z\equiv w_z(\vc\theta_0)\}_{z\in \mbN}$ such that, for any $i,z\geq1$ and $\varepsilon>0$,
\begin{equation*}\lim_{n\rightarrow\infty} \mathbb P_{i}[|\hat{w}_n(z)-w_z|>\varepsilon\,|\,Z_n>0]=0.
\end{equation*}\label{HI_CARMEN!}
\end{enumerate}
Two natural choices for the weight functions are
\begin{equation}\label{weights_def}
\hat{w}^{(1)}_n(z) := \frac{\sum_{i=0}^{n-1}\ind{Z_i=z}}{n} \quad \text{ and } \quad \hat{w}^{(2)}_n(z):= \frac{z \sum_{i=0}^{n-1}\ind{Z_i=z}}{\sum_{i=0}^{n-1} Z_i},
\end{equation}
where $\hat{w}^{(1)}_n(z)$ assigns weights according to the proportion of \textit{time units with population size $z$}, and $\hat{w}^{(2)}_n(z)$ assigns weights according to the proportion of \textit{parents alive when the population size is $z$}.

\begin{lem}\label{weights}
The weights $\{\hat{w}^{(1)}_n(z)\}$ and $\{\hat{w}^{(2)}_n(z)\}$ satisfy Assumption (A1) with respective limiting distributions $\{w_z^{(1)}\}$ and $\{w_z^{(2)}\}$ characterised by
 \[
{w}^{(1)}_z = u_zv_z \quad \text{ and } \quad {w}^{(2)}_z= \frac{z u_z v_z}{\sum_{k=1}^\infty k u_kv_k}, \qquad z\in \mbN.
\]
\end{lem}
Note that $\{w_z^{(1)}\}$ is the stationary distribution of the $Q$-process $\{Z^{\uparrow}_\ell\}$ (see Eq.\ \eqref{stat_dis}), and $\{w_z^{(2)}\}$ is the size-biased distribution associated with $\{w_z^{(1)}\}$. \blue{We discuss the choice of the weights in Remark \ref{rem:w2} and in Section \ref{disc}.}


Next we show that the new estimator $\widehat{\vc\theta}_n$ is equivalent to a modified version of the classical estimator $\widetilde{\T}_n^*$. This implies that all asymptotic properties for $\widehat{\vc\theta}_n$ also hold for the modified classical estimator.

\begin{prop}\label{wls_equiv}The least squares estimator $\widehat{\vc\theta}_n$ with weights $\{\hat{w}_n^{(1)}(z)\}$ or $\{\hat{w}_n^{(2)}(z)\}$ is equal to the classical least squares estimator \eqref{cwlse} modified such that $m(\cdot)$ is replaced by $m^\uparrow(\cdot)$,
\begin{equation}\label{cwlse_m}
\widehat{\T}_n^*:=\arg\min_{\T \in \Theta} \sum_{k =1}^{n} w_k^{(\cdot)} \left\{Z_k - Z_{k-1}\,m^\uparrow(Z_{k-1}, \T)\right\}^2,
\end{equation}
with respective weight $w_k^{(1)}= Z_{k-1}^{-2}$ or $w_k^{(2)}= Z_{k-1}^{-1}$.
\end{prop}

\begin{remark}\label{rem_wls_equiv}Proposition \ref{wls_equiv} is a consequence of a more general result; specifically,
$$\arg\min_{\theta\in\Theta} \sum_{z=1}^\infty \hat{w}^{(\cdot)}_n(z) \left\{\hat{m}_n(z)-f(z, \T)\right\}^2=\arg\min_{\T \in \Theta} \sum_{k =1}^{n} w_k^{(\cdot)} \left\{Z_k - Z_{k-1}\,f(Z_{k-1}, \T)\right\}^2,$$ for any function $f(\cdot)$ and weights as in Proposition \ref{wls_equiv}. Indeed, the proof of Proposition \ref{wls_equiv} does not make use of the properties of $m^\uparrow(\cdot)$.

\end{remark}


\section{Asymptotic properties}\label{sec:as_prop}

We now establish that, under regularity assumptions, the class of least squares estimators $\widehat{\vc\theta}_n$ are $C$-consistent and asymptotically normal.
We start by introducing some notation: let $\|\cdot\|$ denote any norm in $\R^d$ (recall that $\T$ is a $d$-dimensional parameter), let \blue{$\Phi_{\Sigma}(\cdot)$} denote the distribution function of a multivariate normal distribution with mean vector ${\vc 0}=(0,\ldots,0)^\top$ and covariance matrix $\vc{\Sigma}$ whose dimension will be clear from the context,
and let $\vc\nabla$ denote the gradient operator with respect to $\T$.

We make the following assumptions:

\begin{enumerate}[label=(A\arabic*),ref=(A\arabic*),start=2]
\item If $m^\uparrow(z,\T_1)=m^\uparrow(z,\T_2)$ for each $z\in \supp\left(\{ w_z(\T_0)\}\right)$, then $\T_1=\T_2$.
\label{cond:identifiab}

\item $M:=\sup_{\T\in\Theta} \sup_{z\in\N}  m^\uparrow(z,\T)<\infty$.
\label{cond:m-arrow-unif-bounded}
\item
For any initial state $i\in\N$, there exists $C>0$ such that $$\lim_{n\to\infty}\mbP_i\left[\sup_{z\in\N} |\hat{m}_n(z)-m^\uparrow(z,\T_0)|<C \,\Big|\,Z_n>0\right]=1.$$\label{cond:m-arrow-hat-unif-bounded}
\item The function $m^\uparrow(z,\T)$ is twice continuously differentiable with respect to $\T$ for each $z\in\N$. Moreover, for each $\T'\in\Theta$, there exists a compact set $\mathcal{C}\equiv\mathcal{C}(\T')$ such that $\T'\in int(\mathcal{C})$ and
\label{cond:bound-m-arrow}
\begin{enumerate}[label=(\alph*),ref=(\alph*)]
\item $M_1^*(\T'):=\sup_{\T\in \mathcal{C}}\sup_{z\in\N}\big\|\nabla m^\uparrow(z,\T)\big\|_\infty<\infty$.\label{cond:m-arrow-bound-deriv}

\item $M_2^*(\T'):=\sup_{\T\in \mathcal{C}}\sup_{z\in\N}\max_{i,j=1,\ldots,d}\left|\frac{\partial^2m^\uparrow(z,\T)}{\partial\theta_i\partial\theta_j}\right|<\infty$. \label{cond:m-arrow-bound-deriv2}

\item The matrix $\vc\eta(\T_0):= 2\sum_{z=1}^\infty w_z(\T_0)\, \vc\nabla m^\uparrow(z,\T_0)\vc\nabla m^\uparrow(z,\T_0)^\top$ is invertible.
\label{cond:eta_inv}
\end{enumerate}

\noindent
We comment on these assumptions in Remark \ref{rem:assss} below.

\end{enumerate}

\begin{thm}[$C$-consistency of $\widehat{\T}_n$]\label{thm:C-consistency-WLSE-1}
Under Assumptions \ref{HI_CARMEN!}-\ref{cond:m-arrow-hat-unif-bounded}, for any $i\in\N$,
\begin{enumerate}[label=(\roman*),ref=\emph{(\roman*)}]
\item \blue{Asymptotic existence of $\widehat{\T}_n$}:
\label{thm:C-consistency-WLSE-1-i}
\blue{$$\lim_{n\to\infty} \mbP_i\left[\exists \vc X\in \Theta: S_n(\vc X)=\min_{\T\in\Theta} S_n(\T)\,\big|\,Z_{n}>0\right]=1,$$where $S_n(\T):=\sum_{z=1}^\infty \hat{w}_n(z) \left\{\hat{m}_n(z)-{m}^{\uparrow}(z, \T)\right\}^2.$}

\item $C$-consistency: \label{thm:C-consistency-WLSE-1-ii}
$$\forall\varepsilon>0,\quad\lim_{n\to\infty} \mbP_i\left[\big\|\widehat{\T}_n-\T_0\big\|>\varepsilon\,\big|\,Z_{n}>0\right]=0.$$
\end{enumerate}
\end{thm}

\begin{thm}[$C$-normality of $\widehat{\T}_n$] \label{thm:C-normality-WLSE-1}
Assume that either \begin{itemize}\item[\emph{(i)}] there exists a value $z^*<\infty$ such that for all $n\in\mbN$ and $z\geq z^*$, $\hat{w}_n(z)=0$, and Assumption~\ref{HI_CARMEN!} holds; or \item[\emph{(ii)}] for all $n\in\mbN$, $\{\hat{w}_n(z)\}= \{\hat{w}^{(1)}_n(z)\}$; or \item[\emph{(iii)}] for all $n\in\mbN$, $\{\hat{w}_n(z)\}= \{\hat{w}^{(2)}_n(z)\}$.\end{itemize}
Then, under Assumptions \ref{cond:identifiab}-\ref{cond:bound-m-arrow},
for any  $i\in\N$ and \blue{$\vc x\in \R^d$},
\begin{align}\label{eq:C-normality-weighted-estim-1}
\blue{\lim_{n\to\infty} \mbP_i\left[\sqrt{n}\left(\widehat{\T}_{n}-\T_{0}\right)\leq \vc x \,\big|\,Z_n>0\right]=\Phi_{\beta(\T_0)}(\vc x),}
\end{align}
where the covariance matrix
$$\vc\beta(\T_0)=\vc\eta(\T_0)^{-1}\,\vc\zeta(\T_0)\,\vc\eta(\T_0)^{-1}
$$
is a positive semi-definite matrix, with the $d$-dimensional matrices $\vc\eta(\T_0)$ and $\vc\zeta(\T_0)$ given by
\begin{eqnarray*}\vc\eta(\T_0)&=& 2\sum_{z=1}^\infty w_z(\T_0)\, \vc\nabla m^\uparrow(z,\T_0)\vc\nabla m^\uparrow(z,\T_0)^\top\\
\vc\zeta(\T_0)&=&4\sum_{z=1}^\infty w_z(\T_0)^2\,\gamma(z,\T_0) \,\vc\nabla m^\uparrow(z,\T_0)\vc\nabla m^\uparrow(z,\T_0)^\top,\end{eqnarray*}
where \begin{equation}\label{gamma}\gamma(z,\T_0):=\frac{\sigma^{2\uparrow}(z,\T_0)}{u_z(\T_0)v_z(\T_0)}.\end{equation}
\end{thm}

\begin{remark}\label{rem:assss}
Assumption (A2)  holds for a wide a variety of models including those considered in Section \ref{ex:motivation}.
Assumptions (A3) and (A4) hold if, for instance, the offspring distribution has bounded support, which is a reasonable assumption in most biological models.
Assumption (A5) contains technical conditions that, in practice, would likely need to be verified numerically.
\end{remark}

\begin{remark}\label{rem:bounded_supp}
Applying Theorem \ref{thm:C-normality-WLSE-1} with weights that satisfy \emph{(i)} for a large $z^*$ should be sufficient in practice, as 
under Assumption \ref{Asss1}-\ref{cond:lim-sup-m-z}, the process rarely takes very large values.
\end{remark}

\blue{\begin{remark}\label{rem:w2}
We expect that the weighting function $\{\hat{w}^{(2)}_n(z)\}$ should lead to a more efficient estimator than $\{\hat{w}^{(1)}_n(z)\}$; this is because in weighted least squares estimation, optimal weights correspond to the inverse of the variance of the errors $\{\hat{m}_n(z)-{m}^{\uparrow}(z, \T)\}$ in the
$z$th term of the least squares estimator $\widehat{\vc \theta}_n$
\cite[Chapter 10]{wilcox2011introduction}.
By \eqref{varQ}, \eqref{gamma}, and
\eqref{eq:Q-normality-m},
the variance of the error
$\{\hat{m}_n(z)-{m}^{\uparrow}(z, \T)\}$
is given by
\begin{align}\label{ga}
\gamma(z,\T_0)
=\frac{1}{z u_z(\T_0)v_z(\T_0)}\,\cdot\frac{\textrm{Var}[Z_{n+1,\T_0}^{\uparrow}\,|\,Z_{n,\T_0}^{\uparrow}=z]}{z},
\end{align}
where, under our conditions, we expect $\textrm{Var}[Z_{n+1,\T_0}^{\uparrow}\,|\,Z_{n,\T_0}^{\uparrow}=z]/z\approx \sigma^2(z,\T_0)$ which in practice should be approximately constant in $z$. By Lemma \ref{weights}, the conditional limit of $\{\hat{w}^{(2)}_n(z)\}$ is then approximately inversely proportional to $\gamma(z,\T_0)$. For a model where $\textrm{Var}[Z_{n+1,\T_0}^{\uparrow}\,|\,Z_{n,\T_0}^{\uparrow}=z]/z^2$ is approximately constant in $z$, we expect $\{\hat{w}^{(1)}_n(z)\}$ to be \mbox{preferable.}
\end{remark}}

\blue{In the next section we compare the $C$-consistent estimator $\widehat{\vc\theta}_n$ with its natural counterpart obtained by replacing ${m}^{\uparrow}(z, \T)$ with ${m}(z, \T)$ in \eqref{Asss2},
\begin{equation}\label{theta_tilde}
\widetilde{\vc\theta}_n :=\arg\min_{\theta\in\Theta} \sum_{z=1}^\infty \hat{w}_n(z) \left\{\hat{m}_n(z)-{m}(z, \T)\right\}^2.
\end{equation}
By Remark \ref{rem_wls_equiv}, $\widetilde{\vc\theta}_n$ corresponds exactly to the classical estimator $\widetilde{\T}_n^*$ in \eqref{cwlse} with the weights specified in Proposition \ref{wls_equiv}.
In general,
 due to \eqref{Qcons}, $\widetilde{\vc\theta}_n$ is not $C$-consistent when extinction is almost sure. The next proposition establishes the conditional limit of $\widetilde{\vc\theta}_n$, which depends on the chosen weighting function $\{\hat{w}_n(z)\}$ with limit $\{w_z\}$. It requires the additional assumption:
 \begin{enumerate}[label=(A\arabic**),ref=(A\arabic**),start=3]
\item $\widetilde{M}:=\sup_{\T\in\Theta} \sup_{z\in\N}  m(z,\T)<\infty$.
\label{cond:m-unif-bounded}
\end{enumerate}
 \begin{prop}[Conditional weak convergence of $\widetilde{\vc\theta}_n$] \label{thm:Q-consistency-WLSE-1}
Suppose that Assumptions \ref{HI_CARMEN!}, \ref{cond:m-arrow-unif-bounded}, \ref{cond:m-unif-bounded}, and \ref{cond:m-arrow-hat-unif-bounded} hold. Then, for any $i\in\N$,
$$\forall\varepsilon>0,\quad\lim_{n\to\infty} \mbP_i\left[\big\|\widetilde{\T}_n-\widetilde{\T}\big\|>\varepsilon\,\big|\,Z_{n}>0\right]=0,$$
where
\begin{equation}\label{tilde_lim}
\widetilde{\T}=\arg\min_{\T \in \Theta} \sum_{z=1}^\infty w_z \left\{m^\uparrow(z;\T_0)-m(z;\T)\right\}^2
\end{equation}
whenever $\widetilde{\T}$ exists.
\end{prop}

 \blue{As we illustrate in the next section, Proposition \ref{thm:Q-consistency-WLSE-1} can be used to quantify the long-term bias in the classical estimator introduced by the condition $Z_n>0$.}

 }

\section{Simulated examples}
\label{empirical}

We separate our analysis into two cases: \emph{(i)} populations that start in the quasi-stationary regime (near the carrying capacity), and \emph{(ii)} populations which are still growing and are yet to reach the carrying capacity. We compare the accuracy of $\widehat{\vc\theta}_n$ and $\widetilde{\vc\theta}_n$ based on non-extinct trajectories.
\blue{We develop the numerical study with the statistical software and language programming environments \texttt{R} (see \cite{R}) and  \texttt{MATLAB}. 
}

\subsection{Quasi-stationary populations}\label{qsp}

\blue{We now focus on populations that linger around the carrying capacity, and return to} the illustrative example in Section \ref{ex:motivation} whose offspring distribution is given by \eqref{pj}, with $r(z,\T)$ given by \eqref{bh} (Beverton-Holt model)
and where
$\boldsymbol{b}(
\mu)$ is a binary splitting distribution
with $b_2=v$ and $b_0=1-v$ (so that $\mu=2v$).
We fix the carrying capacity at \blue{$K_0=25$} and consider two scenarios, one with small fluctuations of $Z_n$ around the carrying capacity ($v_0=0.7$), and one with larger fluctuations (\blue{$v_0=0.53$}). The amplitude of the fluctuations depends on the parameter $\mu=2 v$ as follows: assuming the support of $m(z)$ is extended to $\mathbb{R}_+$, the larger $|m'(z)|$ at $z=K$, the smaller the fluctuations, as $m'(K)$ controls the strength of the drift toward $K$. For the Beverton-Holt model,  $m'(K)=1-2v$, so taking $ v_0=0.7$ leads to smaller fluctuations than \blue{$v_0=0.53$}.

We \blue{choose the quasi-stationary distribution $\vc u$ as the probability law for the initial population size $Z_0$} and simulate $N_0=2000$ trajectories of the process for $n=50, 500, \blue{5000}$ time units such that $Z_n>0$. \blue{These long non-extinct trajectories are simulated using the method described in \cite[Section~3.3]{braunsteins2022parameter}.}
In Figure~\ref{oooooh} we plot the \blue{empirical distribution of $Z_{5000}$}
 for $ v_0=0.7$ (left panel) and \blue{$v_0=0.53$} (right panel). The spread of the two histograms confirms that the population exhibits larger fluctuations around the carrying capacity when \blue{$v_0=0.53$} than when $v_0=0.7$.

\begin{figure}[h!]
\centering\includegraphics[width=0.37\textwidth]{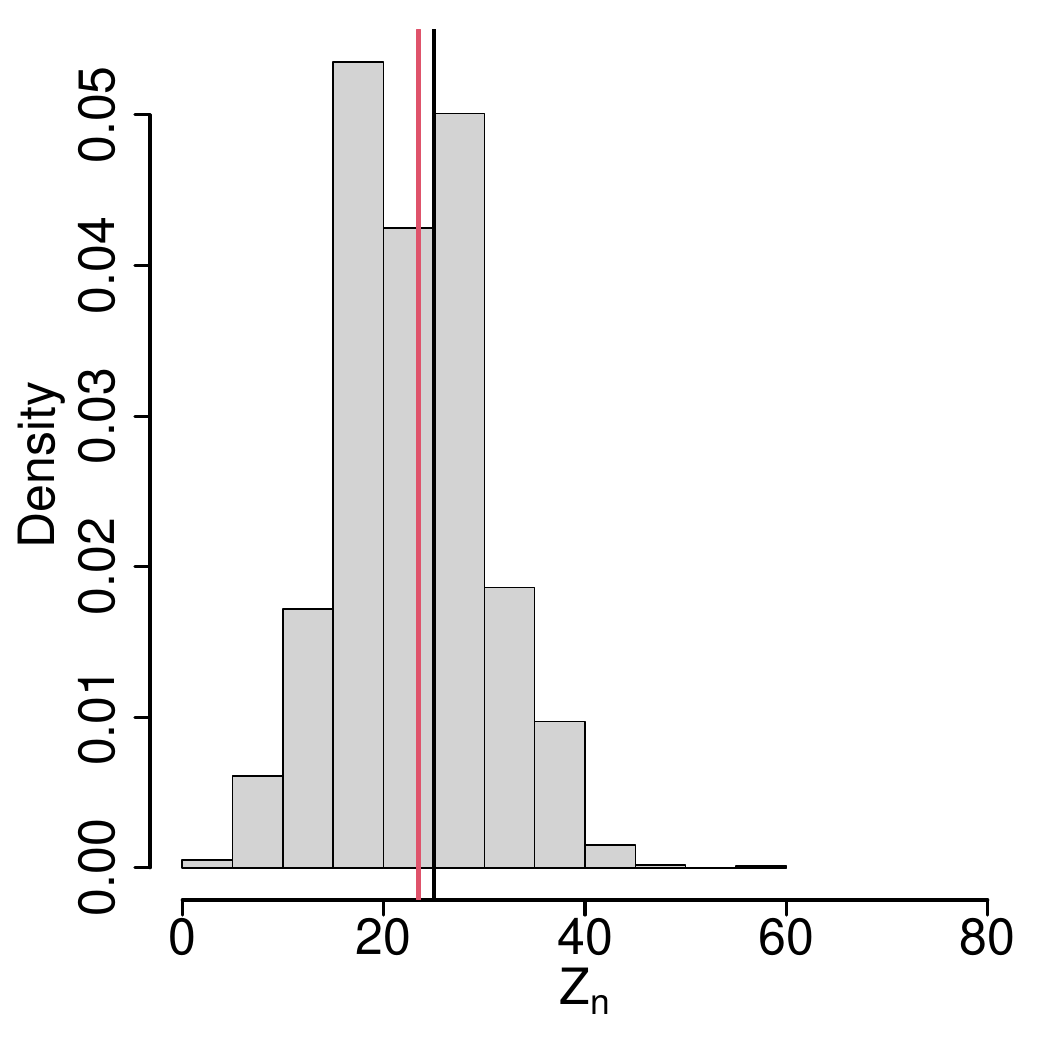}\hspace{0.5cm}
\centering\includegraphics[width=0.37\textwidth]{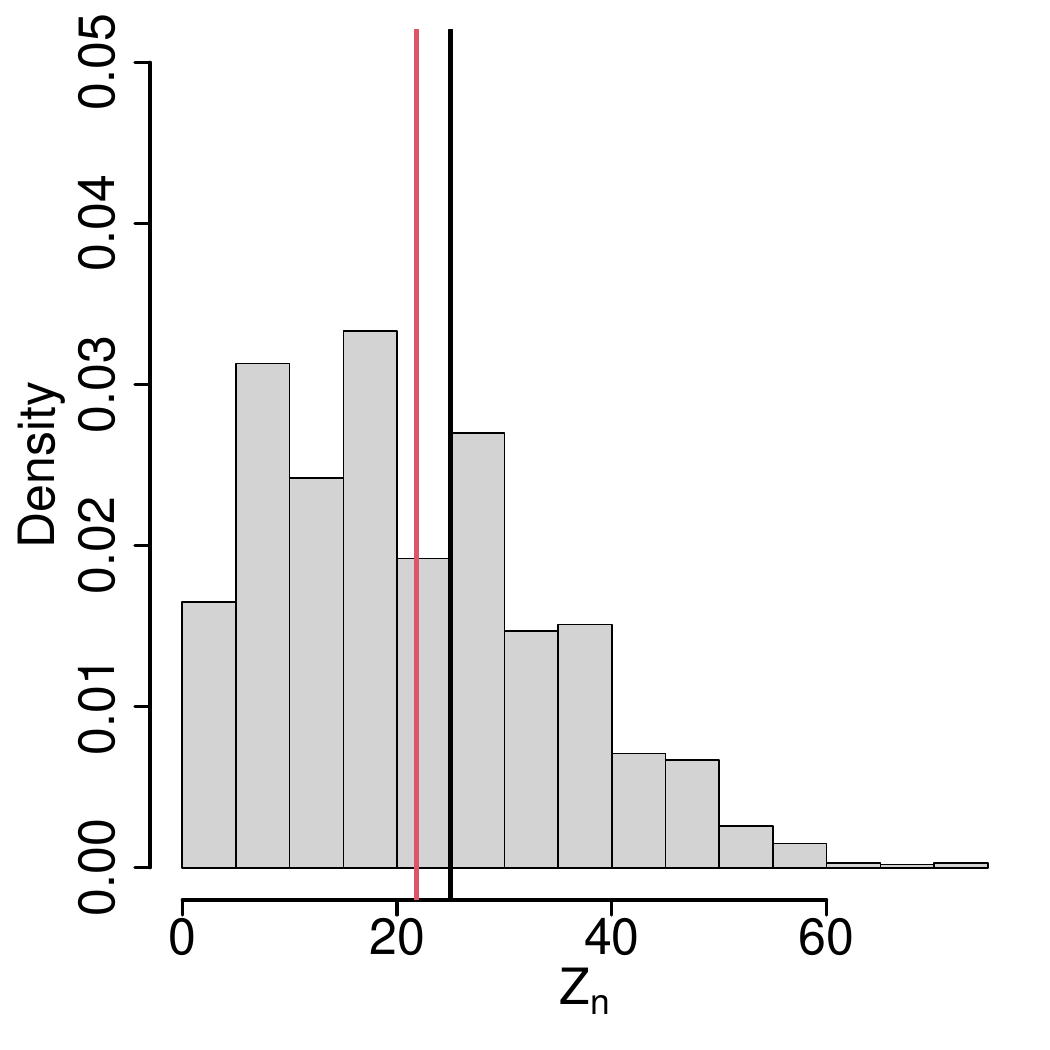}
\caption{\blue{\textsf{Beverton-Holt model}. Left: \textsf{Case with $(K_0, v_0) =(25,0.7)$}. Right: \textsf{Case with $(K_0, v_0) =(25,0.53)$}. Red solid line is the mean population size at time $n=5000$ over the $2000$ simulated paths. Black solid line is the carrying capacity $K_0$.\label{oooooh}}}
\end{figure}
\blue{Based on the simulated trajectories, we estimate $(K,v)$ using the $C$-consistent estimator $$\widehat{\vc \theta}_n=(\hat{K}_{n},\hat{v}_{n})=\arg\min_{(K,v)\in\Theta} \sum_{z\geq 1} \hat{w}_n(z) \left\{\hat{m}_n(z)-{m}^{\uparrow}(z, K,v)\right\}^2$$
and its classical counterpart $$\widetilde{\vc\theta}_n=(\tilde{K}_{n},\tilde{v}_{n})=\arg\min_{(K,v)\in\Theta} \sum_{z\geq 1} \hat{w}_n(z) \left\{\hat{m}_n(z)-{m}(z, K,v)\right\}^2.$$ 
Here we focus on
the weighting function $\boldsymbol{w}_{n}^{(2)}:=\{\hat{w}_n^{(2)}(z)\}=\{z \sum_{i=0}^{n-1}\ind{Z_i=z}/(\sum_{i=0}^{n-1} Z_i)\}$;
 the findings from $\boldsymbol{w}_{n}^{(1)}:=\{\hat{w}_n^{(1)}(z)\}=\{\sum_{i=0}^{n-1}\ind{Z_i=z}/n\}$ support similar conclusions.}
\blue{In Figures \ref{booboo} and \ref{miawiaw}
 we display the empirical marginal and joint distributions of $(\hat{K}_n,\hat{v}_n)$ and $(\tilde{K}_n,\tilde{v}_n)$, for $v_0=0.7$ and $v_0=0.53$ respectively, when $n=5000$.
The blue curve in the histograms for $\hat{K}_n$ and $\hat{v}_n$ is the density of the normal distribution obtained using Theorem~\ref{thm:C-normality-WLSE-1}, and the vertical blue lines represent the corresponding theoretical $95\%$ confidence intervals. For comparison, the vertical dashed lines represent the empirical $95\%$ confidence intervals. We refer to Section~\ref{sec:cov} in Appendix~\ref{appB} for an analysis of the coverage probabilities of the joint confidence regions implied by Theorem~\ref{thm:C-normality-WLSE-1}.}
In addition, in Table~\ref{tab:summary-stationary-Kv-K50-mu1.4} we provide the mean, median, standard deviation (SD) and relative mean square error (RMSE) of the estimates obtained with each estimator, for $v_0=0.7$ and \blue{$v_0=0.53$}.

From Figures \ref{booboo} and \ref{miawiaw} and Table \ref{tab:summary-stationary-Kv-K50-mu1.4}, we observe that when $v_0=0.7$ (small fluctuations around the carrying capacity), the distributions of $\widehat{\vc \theta}_n$ and $\widetilde{\vc\theta}_n$ are similar, whereas when \blue{$v_0=0.53$} (larger fluctuations), there is a clear difference between the distributions of $\widehat{\vc \theta}_n$ and $\widetilde{\vc\theta}_n$, and a clear bias in $\widetilde{\vc\theta}_n$. To understand why, recall that there is a larger difference between $m^\uparrow(z,\T)$ and $m(z,\T)$ for small values of $z$, and when \blue{$v_0=0.53$}, there are larger fluctuations around the carrying capacity, which means that the process reaches these lower population sizes more frequently (see Figure \ref{oooooh}). \blue{The resulting differences in the estimators $\widehat{\vc\theta}_n $ and $\widetilde{\vc\theta}_n $ can be quantified using Proposition \ref{thm:Q-consistency-WLSE-1}, from which we obtain the conditional limits $(\tilde{K},\tilde{v})=(25.018,0.7019)$ and $(\tilde{K},\tilde{v})=(30.79,0.5561)$, which should be compared to the true parameter values in the respective models, $(K_0, v_0) =(25,0.7)$ and $(K_0, v_0) =(25,0.53)$. For a more detailed discussion, we refer to Section~\ref{sec:color_map} in  Appendix~\ref{appB}. 
}

\begin{figure}
\centering\includegraphics[width=0.3\textwidth]{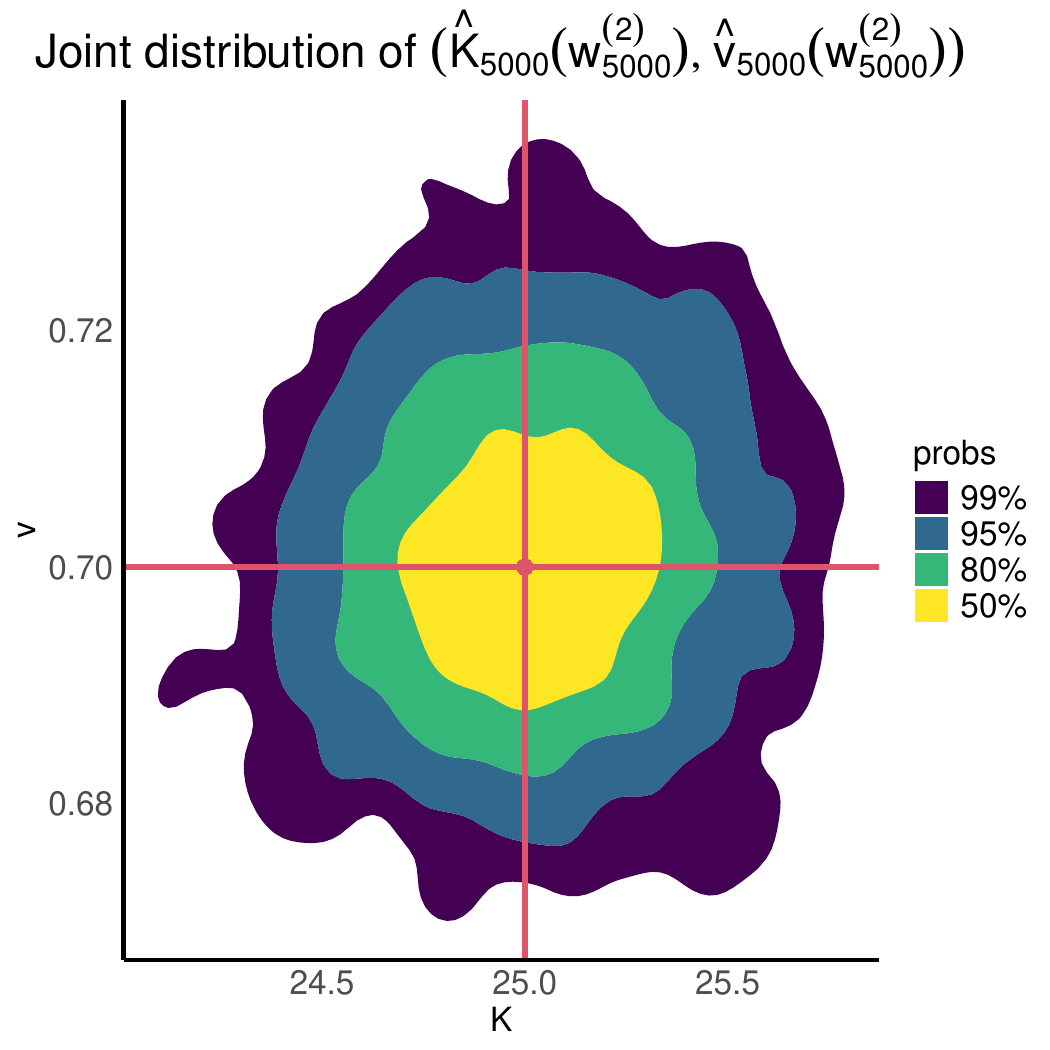}\hspace{1em}\includegraphics[width=0.3\textwidth]{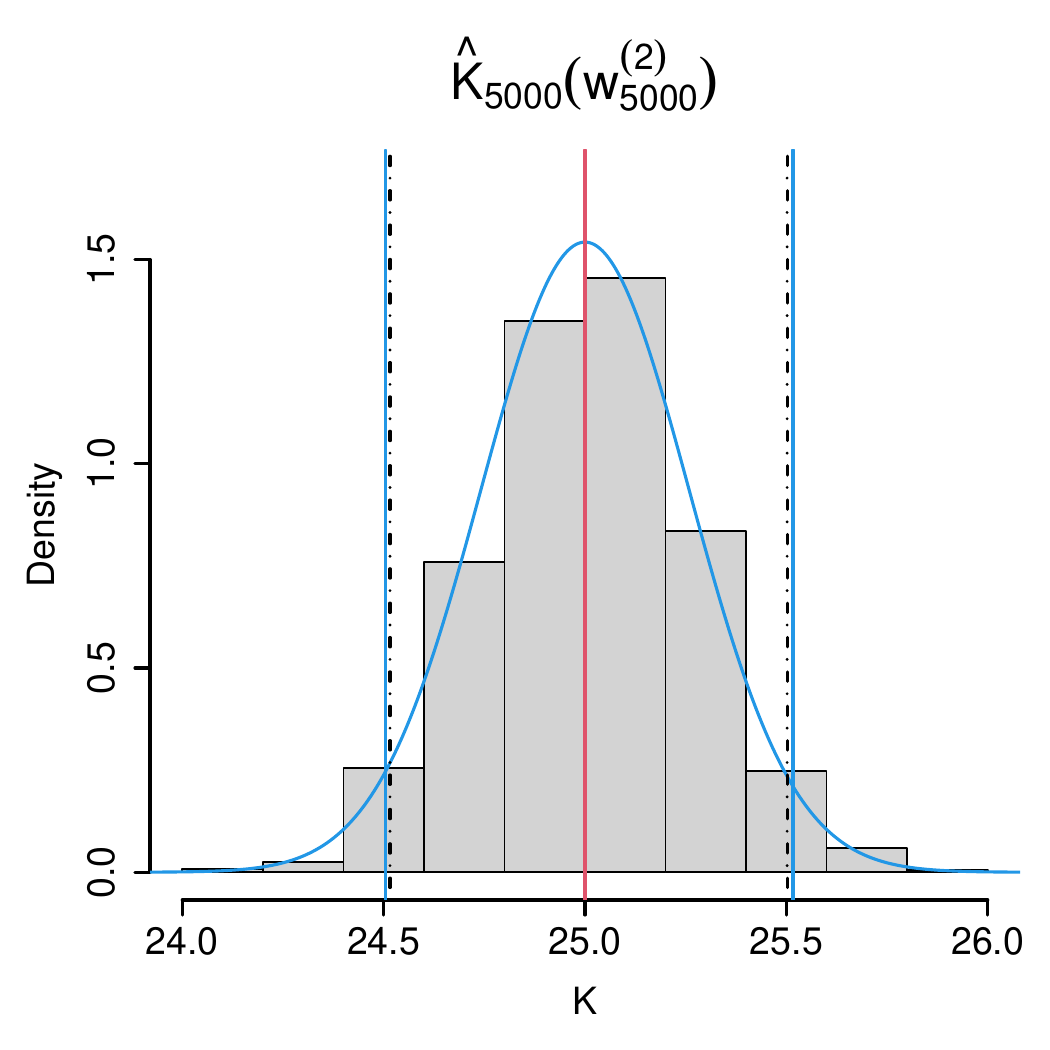}\hspace{1em}
\includegraphics[width=0.3\textwidth]{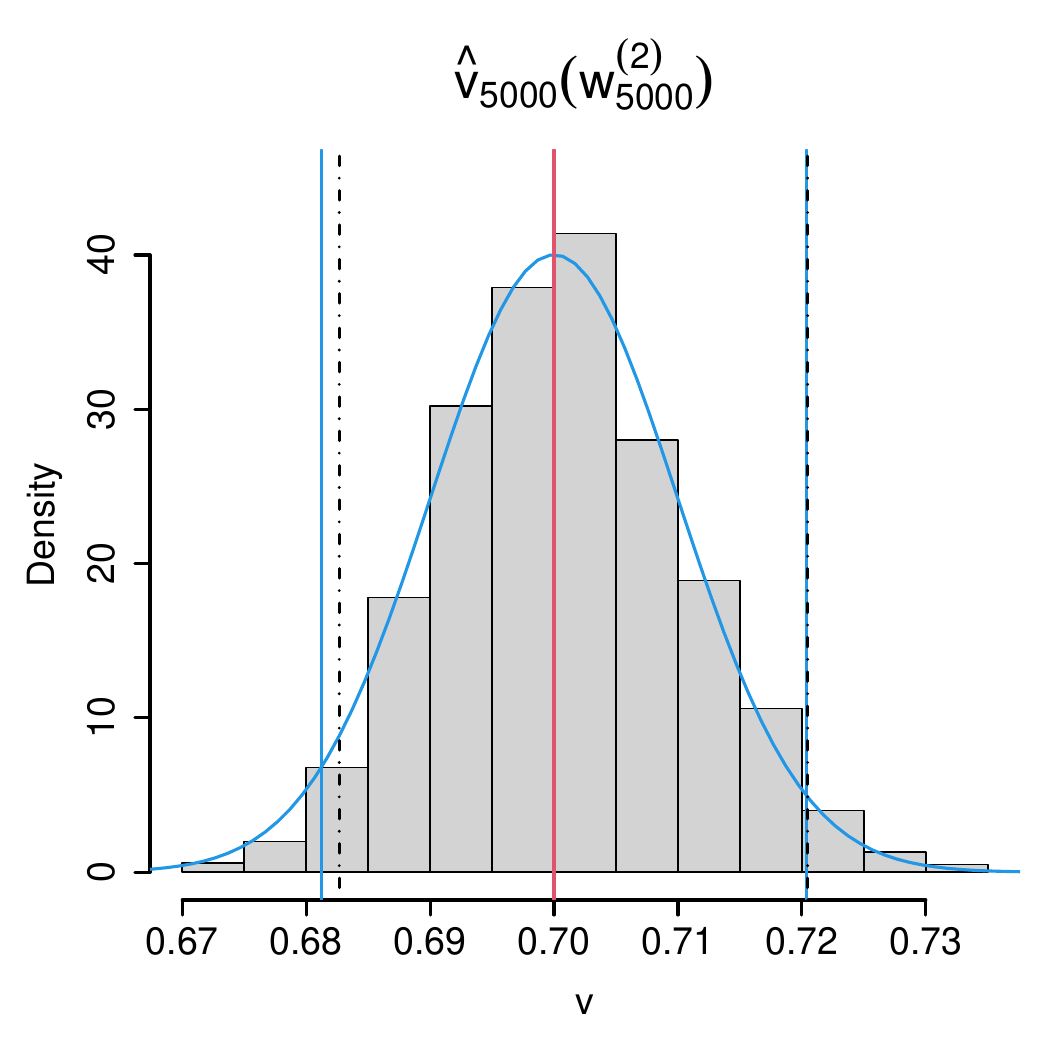}\\
\centering\includegraphics[width=0.3\textwidth]{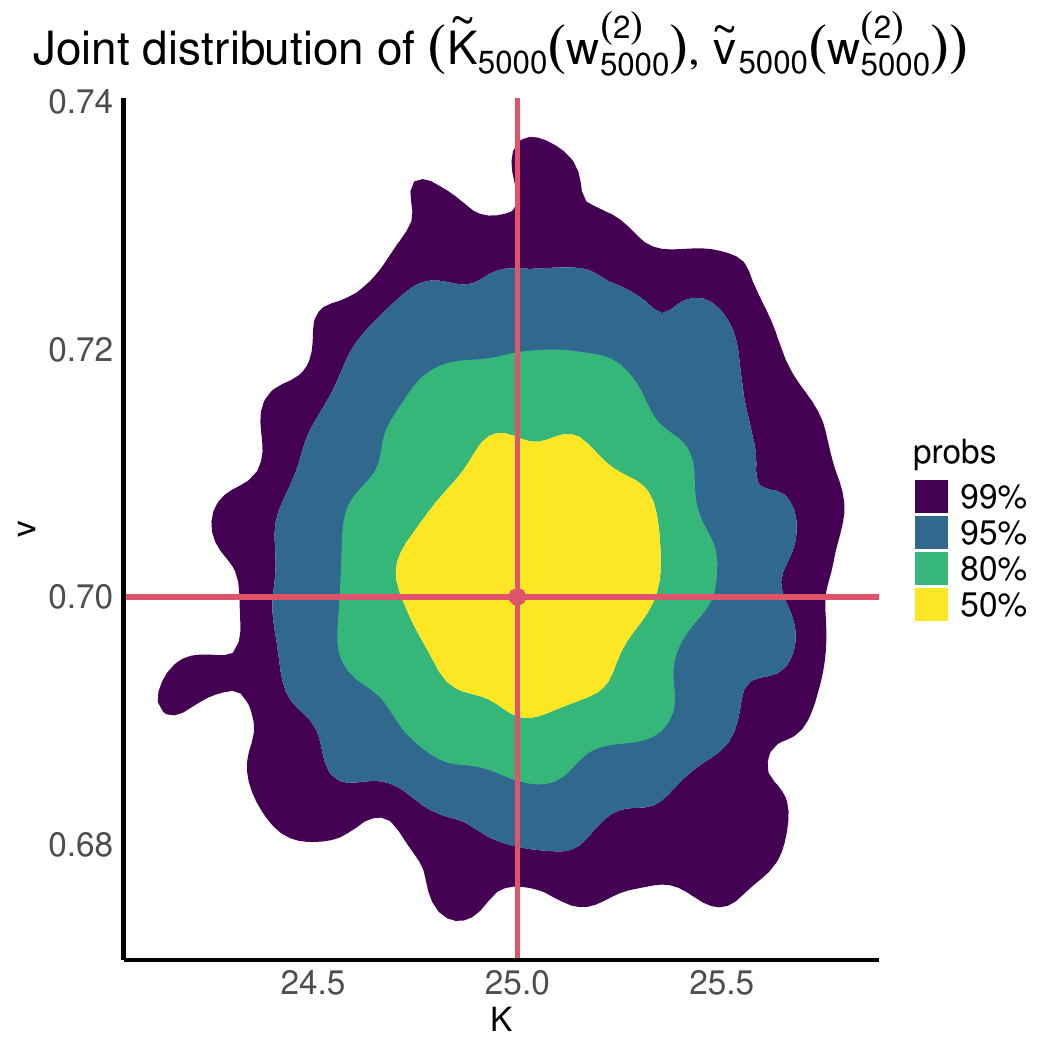}\hspace{1em}
\includegraphics[width=0.3\textwidth]{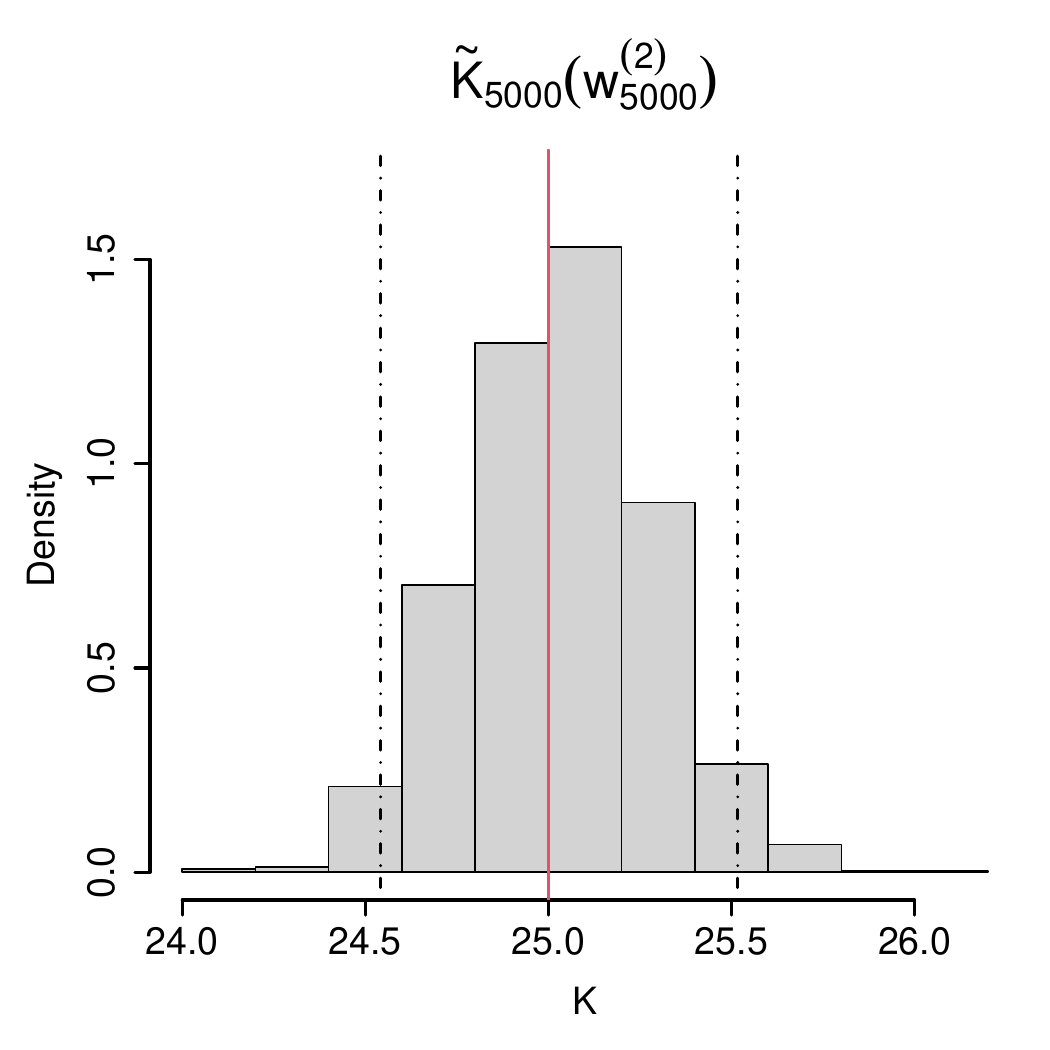}\hspace{1em}
\includegraphics[width=0.3\textwidth]{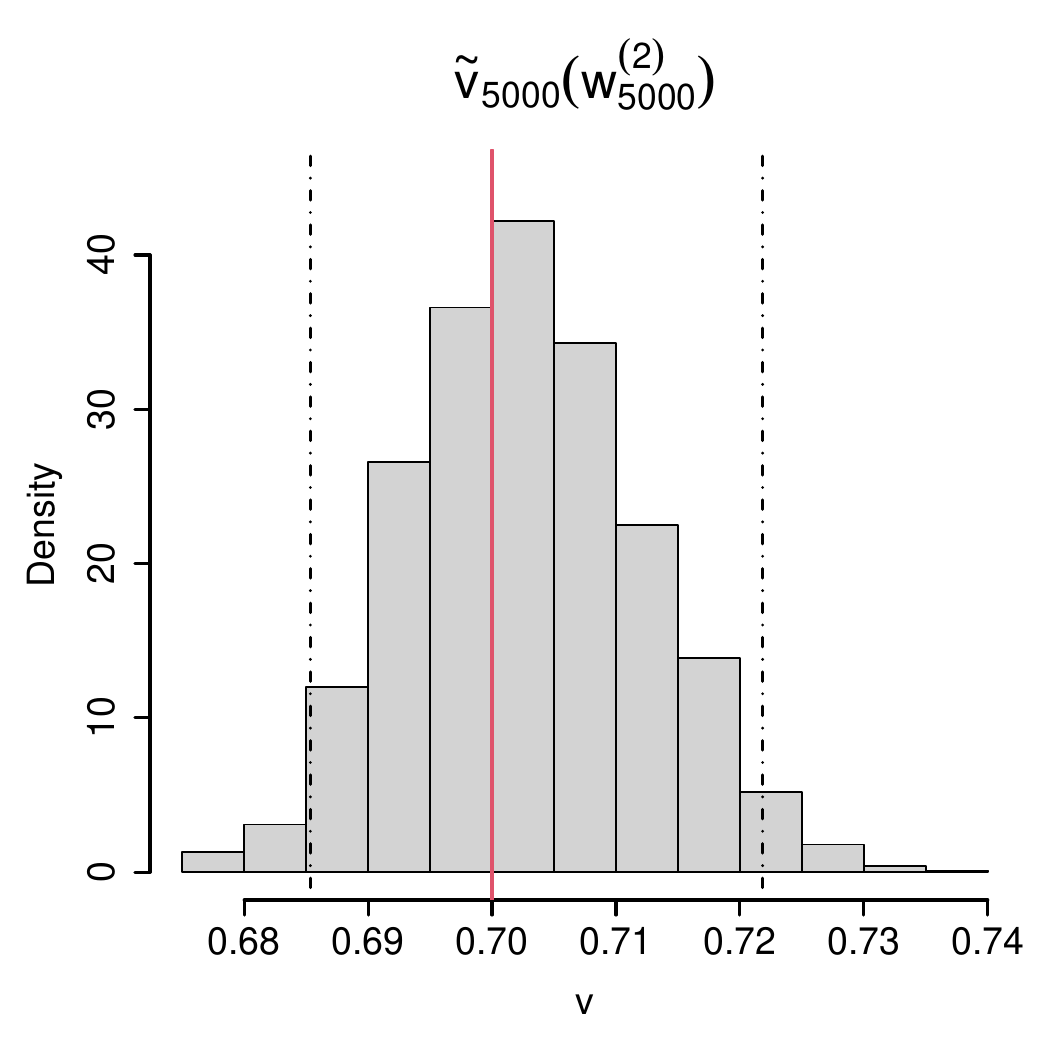}\\
\caption{\blue{\textsf{Case $(K_0,v_0)=(25,0.7)$}. Left: contour plot of the joint distribution of the estimator for $(K,v)$. Centre: marginal distribution of the estimator of $K$, with the empirical and theoretical marginal 95\% confidence intervals. Right: marginal distribution of the estimator of $v$, with the empirical and theoretical marginal 95\% confidence intervals.}\label{booboo}}
\end{figure}

\begin{figure}
\centering\includegraphics[width=0.3\textwidth]{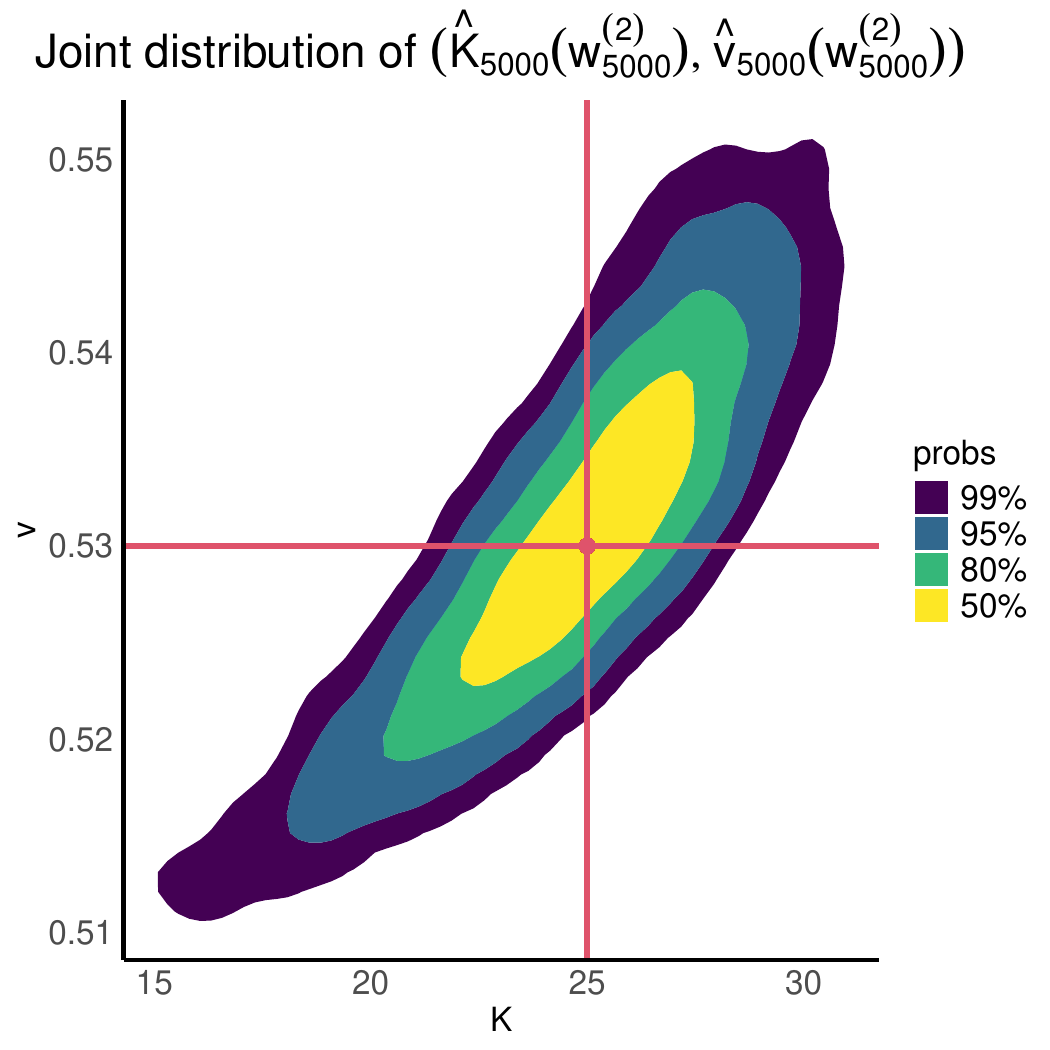}\hspace{1em}\includegraphics[width=0.3\textwidth]{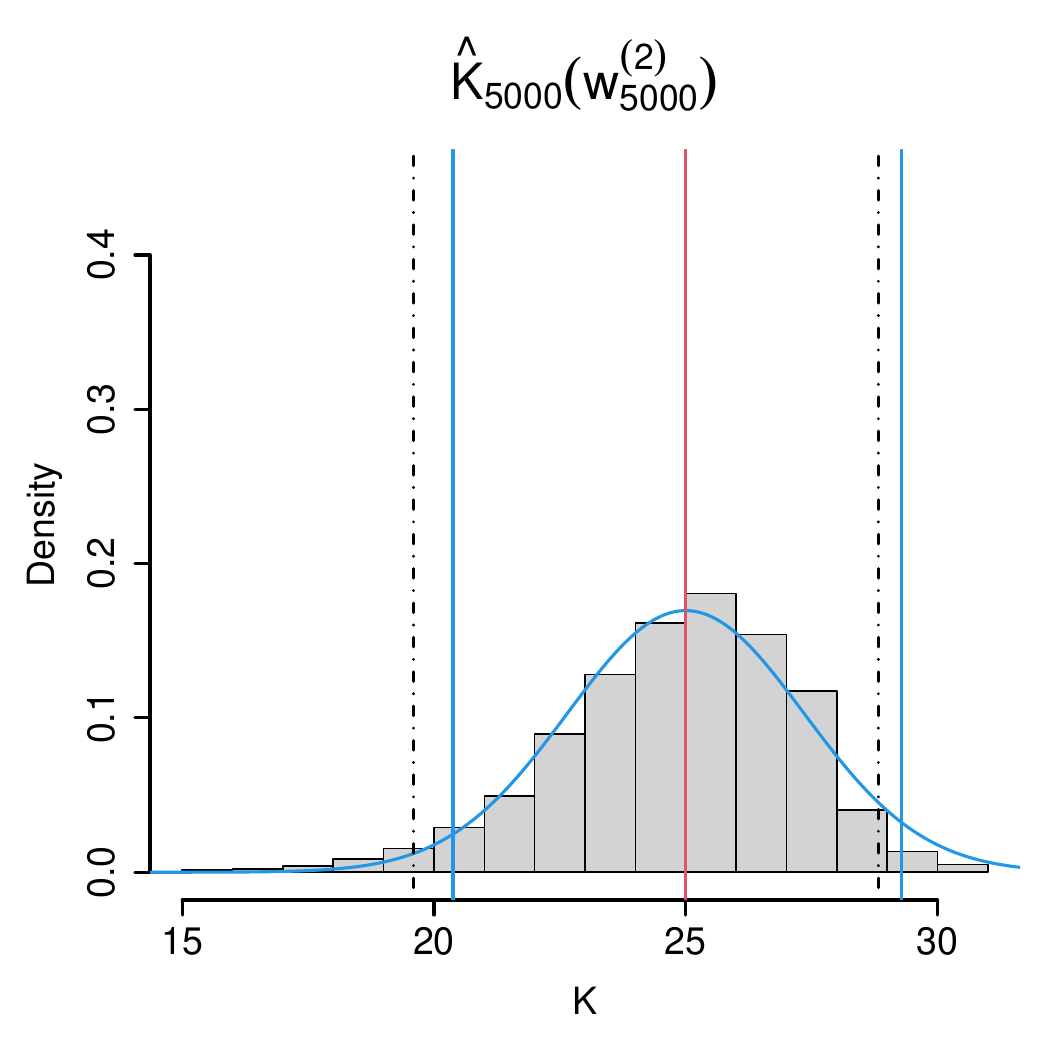}\hspace{1em}
\includegraphics[width=0.3\textwidth]{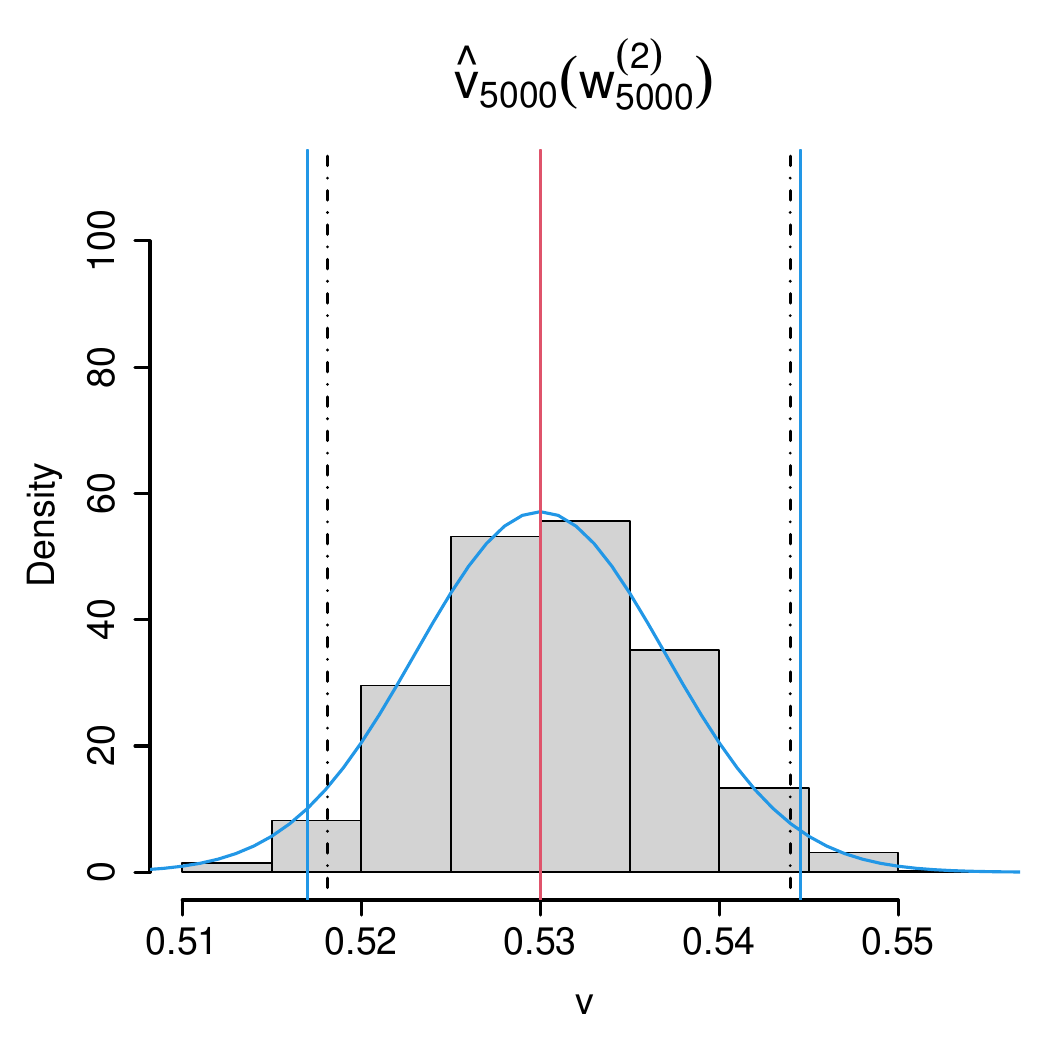}\\
\centering\includegraphics[width=0.3\textwidth]{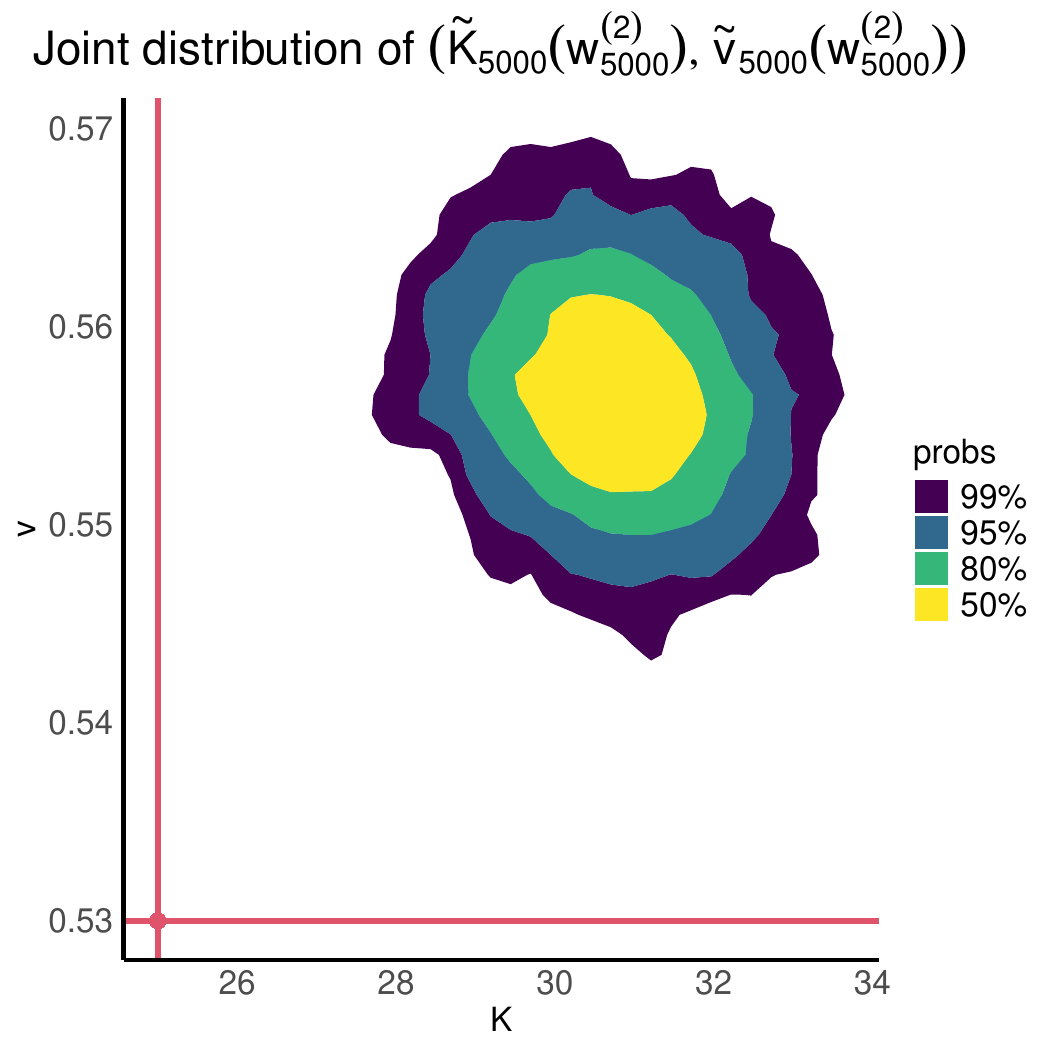}\hspace{1em}\includegraphics[width=0.3\textwidth]{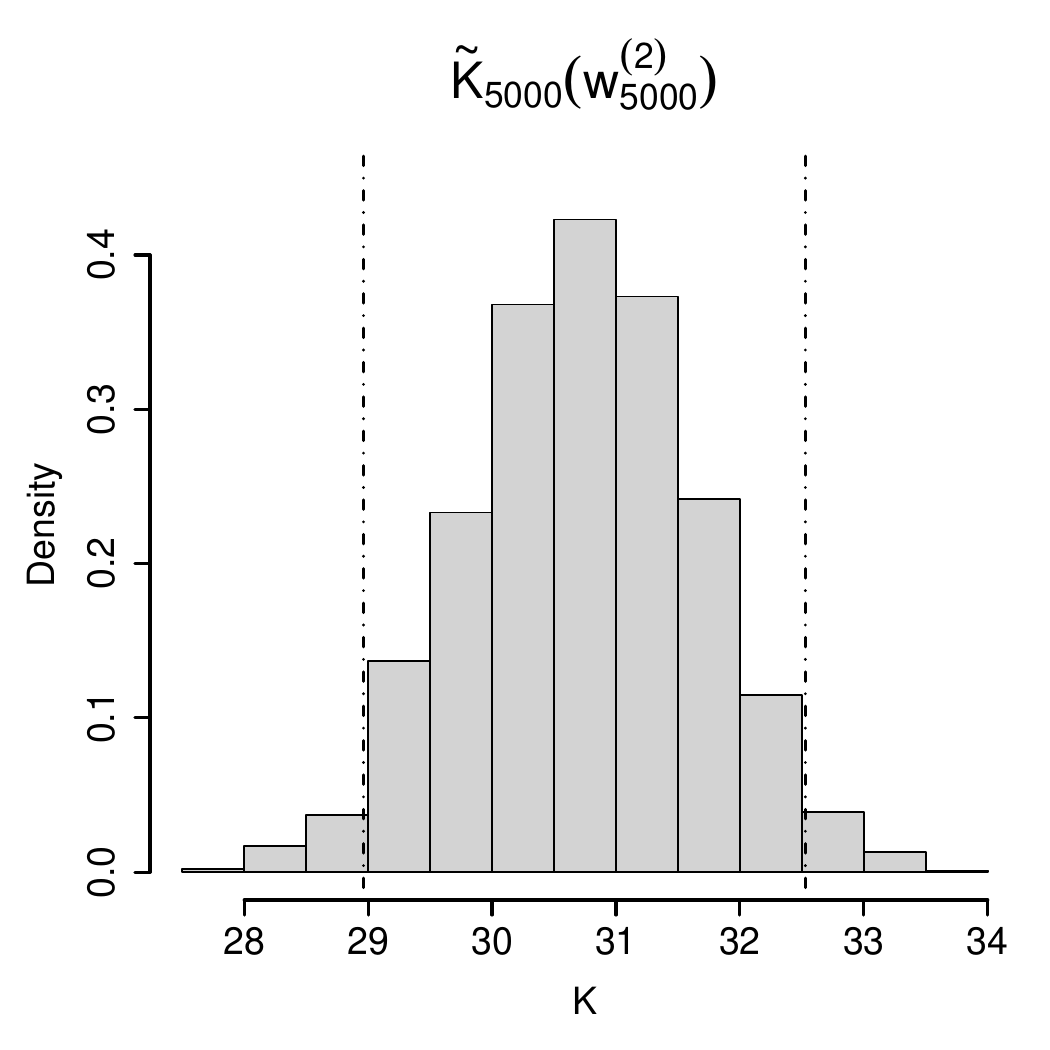}\hspace{1em}
\includegraphics[width=0.3\textwidth]{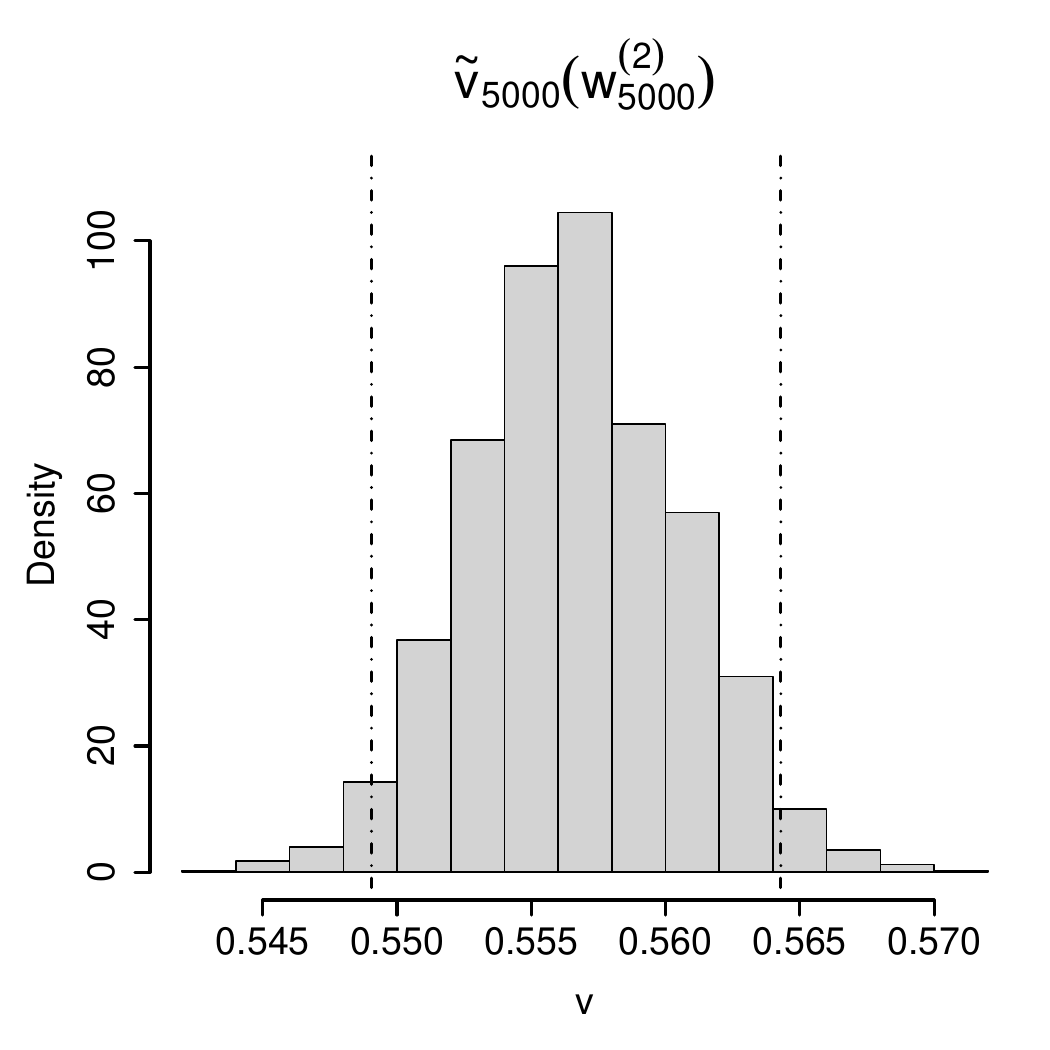}\\
\caption{\blue{\textsf{Case $(K_0,v_0)=(25,0.53)$}. Left: contour plot of the joint distribution of the estimator for $(K,v)$. Centre: marginal distribution of the estimator of $K$, with the empirical and theoretical marginal 95\% confidence intervals. Right: marginal distribution of the estimator of $v$, with the empirical and theoretical marginal 95\% confidence intervals. \label{miawiaw}}}
\end{figure}

\begin{table} \centering
\scalebox{0.75}{\blue{\begin{tabular}{|cc|cc|cc|cc|cc|}
\cline{3-10}
\multicolumn{2}{c|}{}&\multicolumn{4}{c|}{$(K_0, v_0) =(25,0.7)$}&\multicolumn{4}{c|}{$(K_0, v_0) =(25,0.53)$}\\[0.5ex]\cline{3-10}
\multicolumn{2}{c|}{} &  $\hat{K}_{n}(\boldsymbol{w}_{n}^{(2)})$&$\tilde{K}_{n}(\boldsymbol{w}_{n}^{(2)})$&$\hat{v}_{n}(\boldsymbol{w}_{n}^{(2)})$&$\tilde{v}_{n}(\boldsymbol{w}_{n}^{(2)})$&$\hat{K}_{n}(\boldsymbol{w}_{n}^{(2)})$&$\tilde{K}_{n}(\boldsymbol{w}_{n}^{(2)})$&$\hat{v}_{n}(\boldsymbol{w}_{n}^{(2)})$&$\tilde{v}_{n}(\boldsymbol{w}_{n}^{(2)})$\\[0.5ex]
\hline
\multirow{4}{*}{$n=50$} & Median & $24.749$ & $24.787$ & $0.750$ & $0.751$ & $23.486$ & $25.641$ & $0.570$ & $0.590$ \\
& Mean  & $24.573$ & $24.702$ & $0.758$ & $0.761$ & $23.639$ & $26.536$ & $0.569$ & $0.589$ \\ 
& SD   & $2.844$ & $2.665$ & $0.098$ & $0.095$ & $13.506$ & $10.645$ & $0.040$ & $0.034$ \\  
& RMSE   & $0.013$ & $0.012$ & $0.027$ & $0.026$ & $0.295$ & $0.185$ & $0.011$ & $0.016$ \\     
\hline
\multirow{4}{*}{$n=500$} & Median & $25.008$ & $25.036$ & $0.704$ & $0.706$ & $25.447$ & $30.484$ & $0.536$ & $0.560$ \\
& Mean   & $25.002$ & $25.027$ & $0.707$ & $0.709$ & $24.341$ & $30.445$ & $0.537$ & $0.561$ \\ 
& SD   & $0.802$ & $0.796$ & $0.034$ & $0.033$ & $6.189$ & $2.893$ & $0.019$ & $0.013$ \\ 
& RMSE   & $0.001$ & $0.001$ & $0.002$ & $0.002$ & $0.062$ & $0.061$ & $0.001$ & $0.004$ \\   
\hline
\multirow{4}{*}{$n=5000$} & Median & $25.015$ & $25.033$ & $0.701$ & $0.702$ & $25.062$ & $30.740$ & $0.531$ & $0.557$ \\  
& Mean   & $25.011$ & $25.029$ & $0.701$ & $0.703$ & $24.834$ & $30.734$ & $0.531$ & $0.557$ \\ 
& SD   & $0.254$ & $0.253$ & $0.010$ & $0.009$ & $2.320$ & $0.927$ & $0.007$ & $0.004$ \\    
& RSME   & $0.0001$ & $0.0001$ & $0.0002$ & $0.0002$ & $0.009$ & $0.054$ & $0.0002$ & $0.003$ \\   
\hline         
\end{tabular}}}
  \caption{\textsf{Beverton-Holt model with \blue{$(K_0, v_0) =(25,0.7)$} and \blue{$(K_0, v_0) =(25,0.53)$}.} Median, mean, standard deviation (SD) and relative mean squared error (RMSE) of the estimates obtained with each of the estimators.\label{tab:summary-stationary-Kv-K50-mu1.4}}
\end{table}

\subsection{Growing populations}\label{growing_pop}

\blue{We now compare the estimators $\widehat{\vc \theta}_n$ and
$\widetilde{\vc\theta}_n$ in the challenging ---but practically important (see Section~\ref{sec:br})--- setting where the observed population size counts are
well below the carrying capacity.
We consider the same Beverton-Holt model as in Section \ref{qsp}
with parameters $\T_0=(K_0, v_0) = (100,0.6)$.
With these parameters, Proposition~\ref{thm:Q-consistency-WLSE-1} leads to $(\tilde{K},\tilde{v})=(99.8,0.6018)$, which suggests that there would not be a large difference between  $\widehat{\vc \theta}_n$ and $\widetilde{\vc\theta}_n$ for large values of $n$. However, even though the asymptotic difference would be negligible, we will highlight that for small initial population sizes and moderate values of $n$, the $C$-consistent estimator $\widehat{\vc \theta}_n$ is noticeably less biased than $\widetilde{\vc\theta}_n$.}

We fix the initial population size at $N=2$,
and simulate $N_0=1000$ trajectories of the process for \blue{$n=0,1,\ldots,35$} time units such that $Z_n>0$.
In Figure~\ref{fig:ex3:paths} we illustrate the sample mean of the trajectories (left) and the empirical distribution of $Z_{35}$ (right). Observe that the sample mean of $Z_{35}$ is well below the carrying capacity $K_0=100$ and the sample variance  is large.

\begin{figure}
\centering\includegraphics[width=0.37\textwidth]{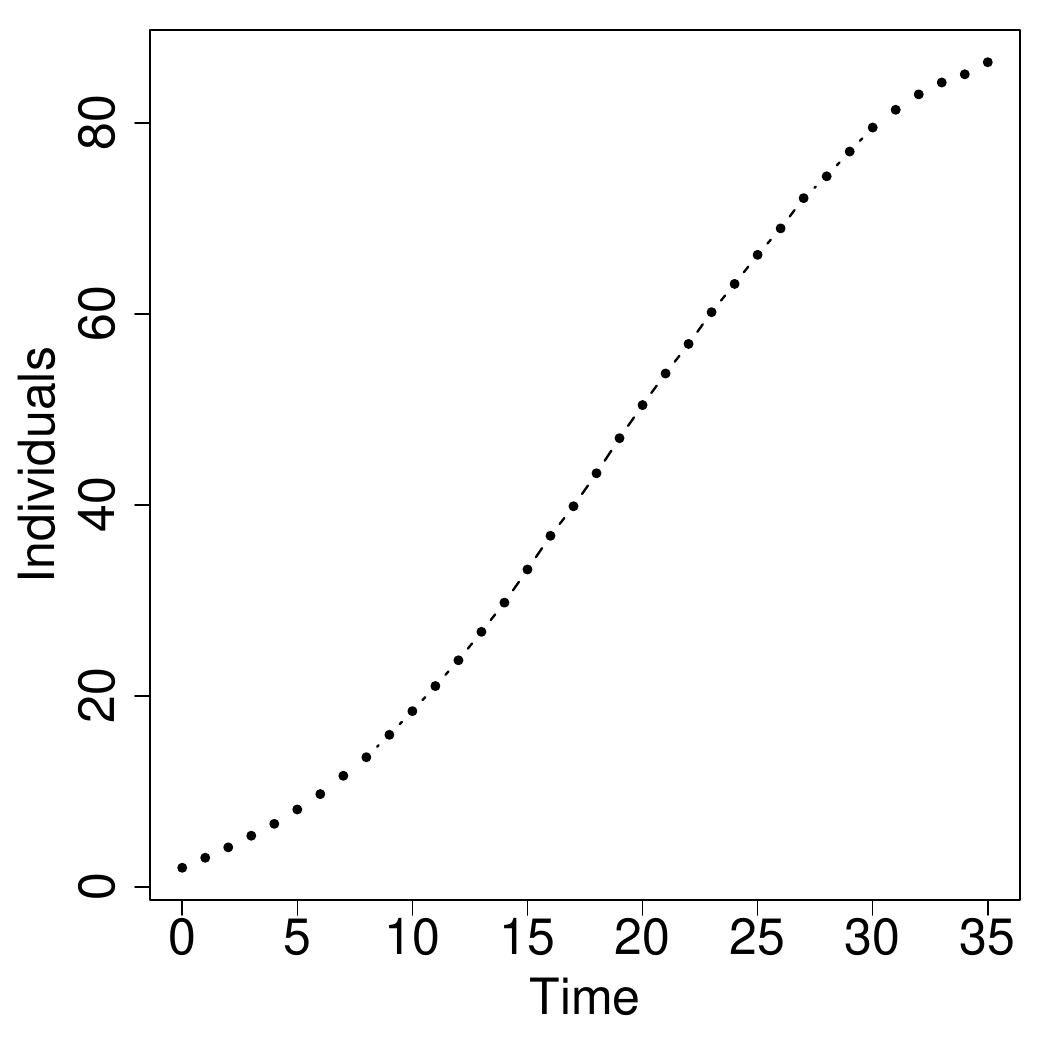}
\hspace{0.5cm}\includegraphics[width=0.37\textwidth]{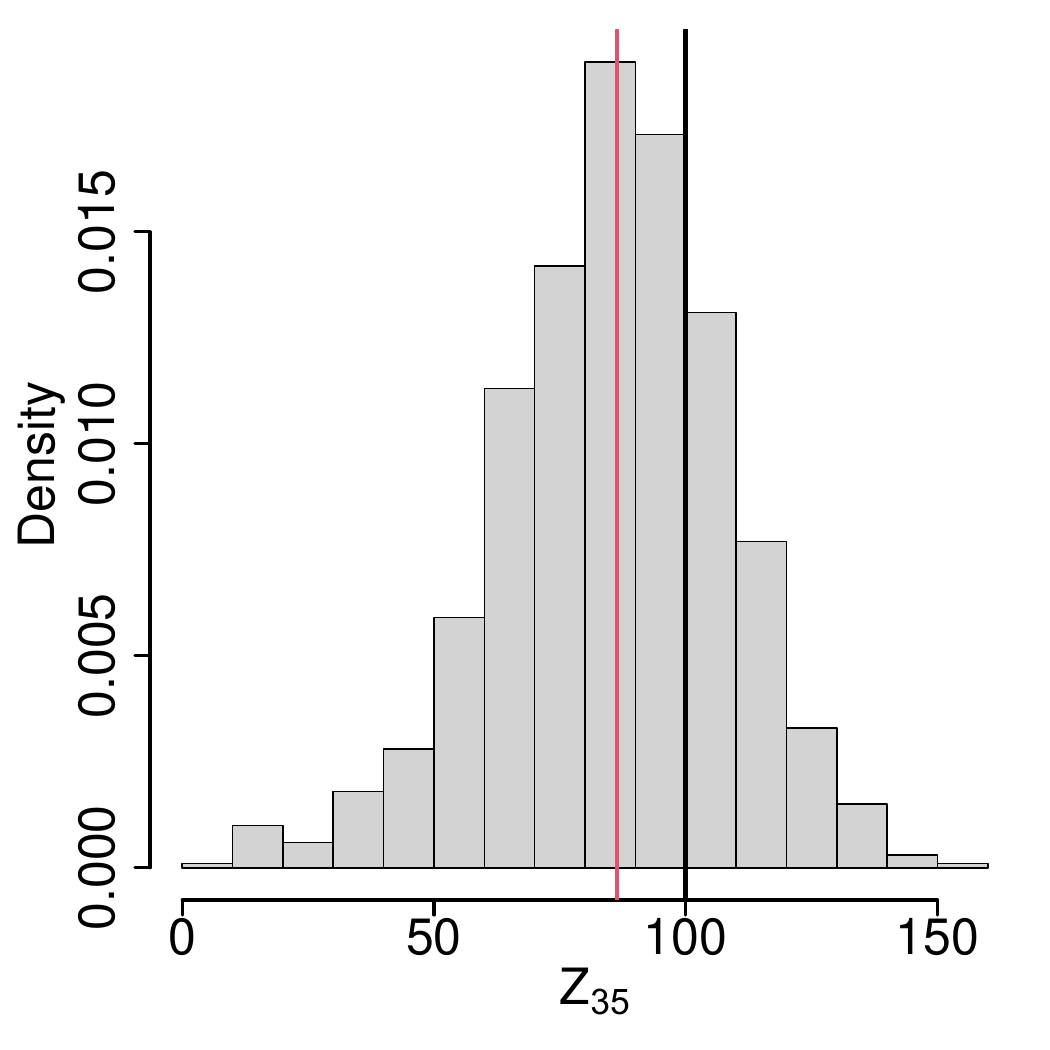}
\caption{\blue{\textsf{Beverton-Holt model with $(K_0, v_0) =(100,0.6)$}. Left: Mean path of the 1000 simulated trajectories.  Right: Histogram of $Z_{35}$; the red vertical line is the sample mean and the black vertical line is $K_0=100$.}\label{fig:ex3:paths}}
\end{figure}

\begin{table} \centering
\scalebox{0.75}{\blue{\begin{tabular}{|c|c|c|c|c|c|c|c|c|c|c|c|c|}
\cline{2-13}
\multicolumn{1}{c|}{} & \multicolumn{4}{|c|}{$n=15$}& \multicolumn{4}{|c|}{$n=25$}& \multicolumn{4}{|c|}{$n=35$}\\
\hline
Method & Mean & Median & SD & RMSE & Mean & Median & SD & RMSE & Mean & Median & SD & RMSE \\
\hline
$\hat{K}_{n}(\boldsymbol{w}_{n}^{(2)})$  & $71.840$ & $44.666$ & $79.256$ & $0.707$ & $98.100$ & $84.808$ & $60.994$ & $0.372$ & $101.010$ & $94.134$ & $41.196$ & $0.170$ \\
$\tilde{K}_{n}(\boldsymbol{w}_{n}^{(2)})$& $58.839$ & $38.670$ & $64.449$ & $0.584$ & $88.171$ & $79.864$ & $51.041$ & $0.274$ & $96.268$ & $92.885$ & $34.330$ & $0.119$ \\
$\hat{K}_{n}(\boldsymbol{w}_{n}^{(1)})$  & $67.565$ & $41.946$ & $72.407$ & $0.629$ & $97.433$ & $84.563$ & $60.504$ & $0.366$ & $103.047$ & $95.190$ & $44.320$ & $0.197$ \\
$\tilde{K}_{n}(\boldsymbol{w}_{n}^{(1)})$& $40.882$ & $32.073$ & $35.523$ & $0.476$ & $69.206$ & $68.855$ & $32.046$ & $0.197$ & $84.328$ & $84.594$ & $22.993$ & $0.077$ \\ [+0.7em]\hline
$\hat{v}_{n}(\boldsymbol{w}_{n}^{(2)})$  & $0.638$ & $0.639$ & $0.075$ & $0.020$ & $0.628$ & $0.628$ & $0.049$ & $0.009$ & $0.622$ & $0.619$ & $0.042$ & $0.006$ \\
$\tilde{v}_{n}(\boldsymbol{w}_{n}^{(2)})$& $0.670$ & $0.670$ & $0.056$ & $0.022$ & $0.643$ & $0.640$ & $0.040$ & $0.010$ & $0.631$ & $0.627$ & $0.037$ & $0.007$ \\
$\hat{v}_{n}(\boldsymbol{w}_{n}^{(1)})$  & $0.650$ & $0.652$ & $0.088$ & $0.029$ & $0.630$ & $0.625$ & $0.062$ & $0.013$ & $0.622$ & $0.617$ & $0.055$ & $0.010$ \\
$\tilde{v}_{n}(\boldsymbol{w}_{n}^{(1)})$& $0.720$ & $0.719$ & $0.060$ & $0.050$ & $0.683$ & $0.677$ & $0.047$ & $0.025$ & $0.667$ & $0.659$ & $0.045$ & $0.018$ \\
\hline
\end{tabular} }}
\caption{\blue{\textsf{Beverton-Holt model with $(K_0, v_0) =(100,0.6)$}. Mean, median, standard deviation (SD) and relative mean squared error (RSME) of the estimates for $K$ and $v$ using the $C$-consistent estimator $\widehat{\vc \theta}_n=(\hat{K}_{n},\hat{v}_{n})$ in \eqref{Asss2} and its counterpart $\widetilde{\vc\theta}_n=(\tilde{K}_{n},\tilde{v}_{n})$ in \eqref{theta_tilde} with weighting functions $\boldsymbol{w}_{n}^{(1)}$ and $\boldsymbol{w}_{n}^{(2)}$ defined in \eqref{weights_def}, for $n=15,25,35$.}\label{tab:growing-Kv-K-summary-bh}}
\end{table}

\begin{figure}[H]
\centering\includegraphics[width=0.42\textwidth]{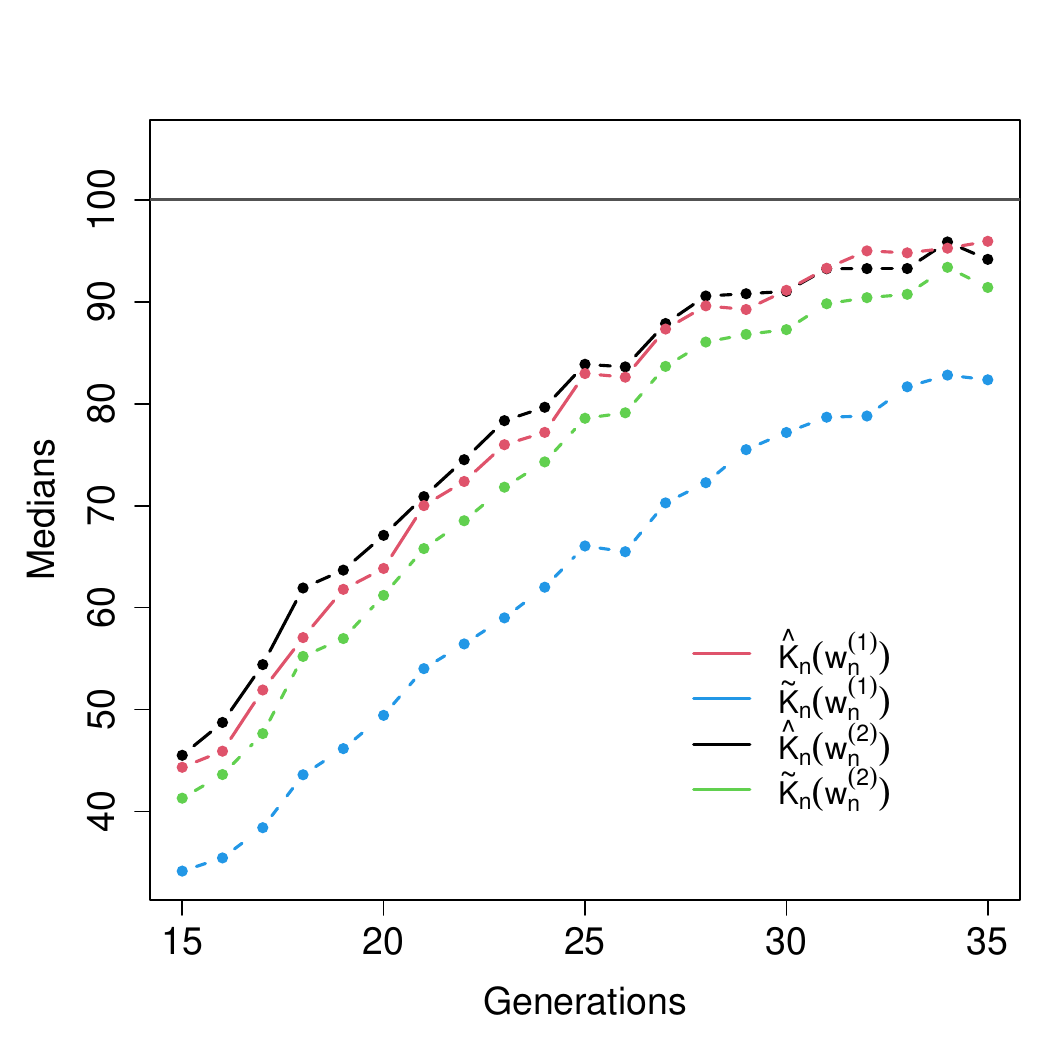}
\hspace{0.3cm}\includegraphics[width=0.42\textwidth]{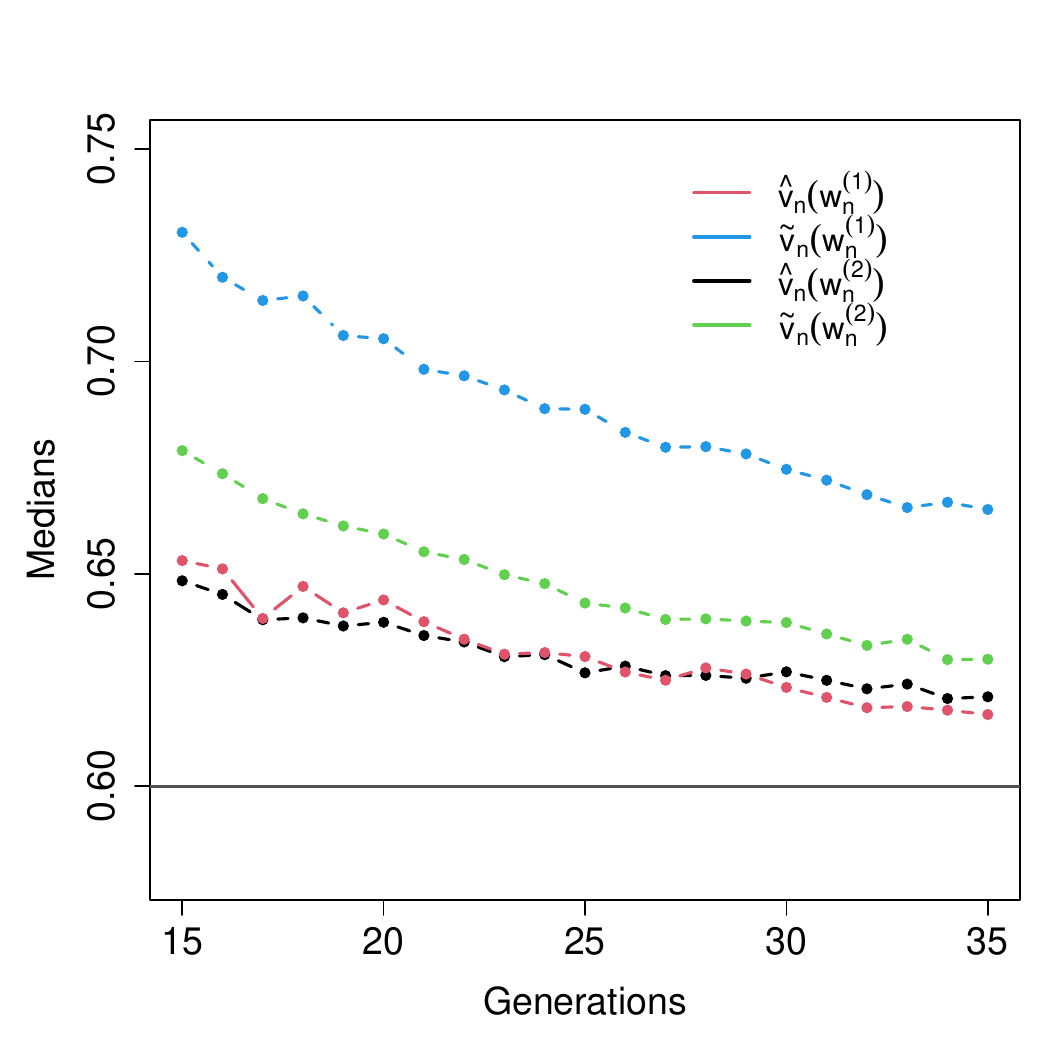}
\caption{\blue{\textsf{Beverton-Holt model with $(K_0, v_0) =(100,0.6)$}. Sample medians of $\widehat{\vc \theta}_n$ and $\widetilde{\vc\theta}_n$ with the weights $\boldsymbol{w}_{n}^{(1)}$ and $\boldsymbol{w}_{n}^{(2)}$. Horizontal lines represent the true parameter value. } \label{fig:MERIDA_RULES!!!}}
\end{figure}

In Table \ref{tab:growing-Kv-K-summary-bh}, we provide the mean, median, standard deviation (SD) and relative mean squared error (RSME) of the estimates for $K$ and $v$ using $\widehat{\vc \theta}_n=(\hat{K}_{n},\hat{v}_{n})$ and $\widetilde{\vc\theta}_n=(\tilde{K}_{n},\tilde{v}_{n})$  with weighting functions $\boldsymbol{w}_{n}^{(1)}$ and $\boldsymbol{w}_{n}^{(2)}$ defined in \eqref{weights_def}, for \blue{$n=15,25,35$}.
In line with Remark \ref{rem:w2}, we observe that the weighting function $\boldsymbol{w}_{n}^{(2)}$ generally outperforms $\boldsymbol{w}_{n}^{(1)}$.
In addition, for all $n$, and for both $K$ and $v$, the mean and median of $\widehat{\vc \theta}_n$ are closer to the true value $\T_0= (100,0.6)$ than the counterpart $\widetilde{\vc\theta}_n$ for any fixed weighting function; this illustrates that $\widehat{\vc \theta}_n$ has a smaller bias than $\widetilde{\vc\theta}_n$.
\blue{This is further illustrated in Figure~\ref{fig:MERIDA_RULES!!!} which displays the empirical medians of $\widehat{\vc \theta}_n$ and $\widetilde{\vc\theta}_n$ for $n=15, 16,\ldots,35$.
A box plot of $\widehat{\vc \theta}_n$ and $\widetilde{\vc\theta}_n$ for $n=15,25,35$ is displayed in Figure~\ref{bias_corr} of Appendix~\ref{appB}, where we observe that, for $\boldsymbol{w}_{n}^{(2)}$, the true value of $K$ and $v$ generally lies between the $1$st and $3$rd quantiles of the empirical distribution of $\widehat{\vc \theta}_n$.
Even though $\widehat{\vc \theta}_n$ is less biased than $\tilde{\vc \theta}_n$, Figure~\ref{fig:MERIDA_RULES!!!} indicates that $\widehat{\vc \theta}_n$ is still biased, especially for small values of $n$.
This is a standard feature of estimators for branching processes, particularly for growing populations; see for instance \cite[Chapter 2, Section 2]{Guttorp-2014}. In Section~\ref{subsec:bias-correction} of Appendix~\ref{appB} we apply a bias correction based on bootstrap techniques, however this did not lead to a significant bias improvement for $K$.
}

\blue{Another way to justify $\widehat{\vc \theta}_n$ for growing trajectories of moderate length is to
 consider a sample made of a large number of independent trajectories (as opposed to a single trajectory). Suppose our sample contains \blue{$10^3$}
 independent trajectories of length $n=15,16,\ldots,35$. In this case, the estimators $\widehat{\vc \theta}_n$ and $\widetilde{\vc\theta}_n$ are the same as before with the MLE for $m(z)$ given by
\begin{equation}\label{lump}\hat{m}_{n}(z)=\frac{\sum_{j=1}^{\blue{10^3}}\sum_{i=0}^{n-1} Z_{j,i+1}\ind{Z_{j,i}=z}}{z\,\sum_{j=1}^{\blue{10^3}}\sum_{i=0}^{n-1}\ind{Z_{j,i}=z}},\end{equation}
where $Z_{j,i}$ is the population size at time $i$ in the $j$th trajectory, and the weighting functions $\boldsymbol{w}_{n}^{(1)}$ and $\boldsymbol{w}_{n}^{(2)}$ are adjusted similarly.
We display the corresponding estimates in \blue{Figure~\ref{fig:accumulated-growing-Kv-K100-mu1.2-bh},}
which highlights that $\widehat{\vc \theta}_n$ significantly outperforms
$\widetilde{\vc\theta}_n$ with both weighting functions.

While the new class of estimators $\widehat{\vc \theta}_n$ relies on asymptotic properties ---because $m^\uparrow(z, \T_0)$ is the conditional limit as $n\to\infty$ of $\hat{m}_n(z)$, given $Z_n>0$--- the results of this section suggest they also perform well for moderate values of $n$.
This is because a PSDBP with a carrying capacity generally  dies rapidly or after a long time \cite{hamza2016establishment}. In our example, the values of $\mbP[Z_n>0]$  for $n=5, 15,25,35,1000$ are, respectively, $0.632,   0.548,    0.542,    0.538   , 0.523$ (approximately), highlighting the fact that $\mbP[Z_n>0]\approx \mbP[Z_{1000}>0]$ for $n\geq 15$.
 Consequently, trajectories that survive for a moderate time can be treated as trajectories that would not become extinct for a large time ---that is, trajectories of the $Q$-process.
This explains why subtracting ${m}^{\uparrow}(z, \T)$ in the least squares estimator generally reduces the bias, even for moderate values of $n$.

\begin{figure}
\centering\includegraphics[width=0.42\textwidth]{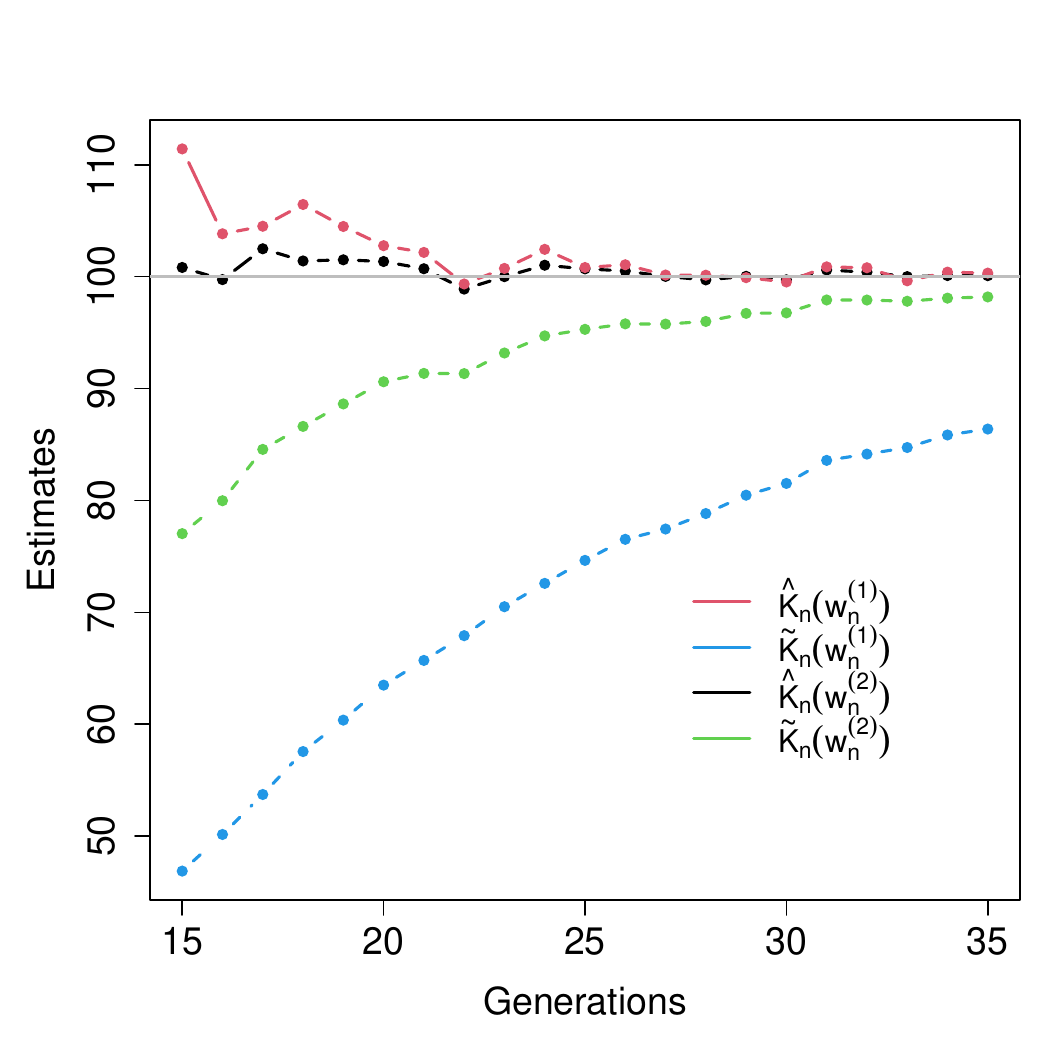}
\includegraphics[width=0.42\textwidth]{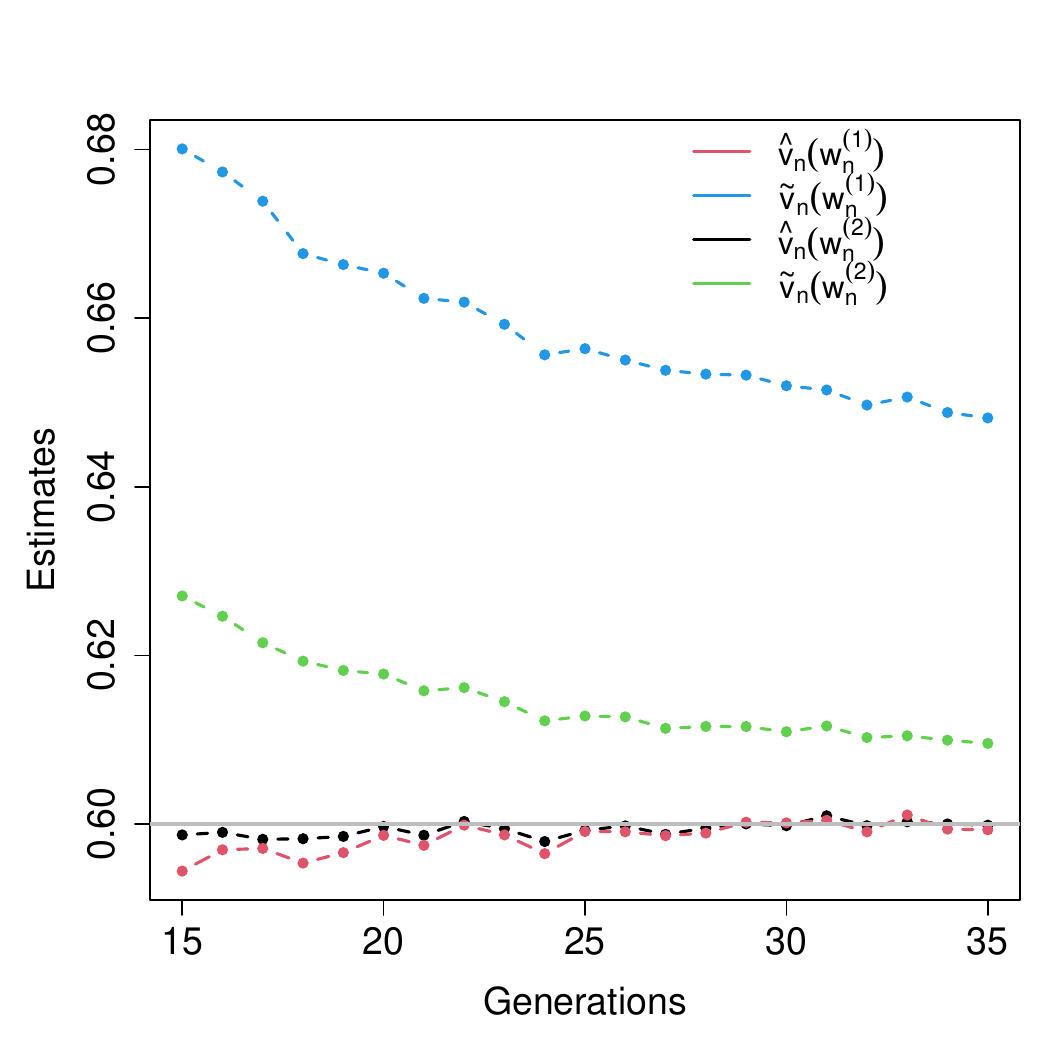}
  \caption{\textsf{Beverton-Holt model  with $(K_0, v_0) =(100,0.6)$}. Estimates obtained when \blue{$10^3$} independent trajectories are lumped.}
  \label{fig:accumulated-growing-Kv-K100-mu1.2-bh}
\end{figure}
}

\blue{
We observe that, while the $C$-consistent estimators are less biased than the classic estimators for any $n$, in our simulated examples they appear to have a larger variance. Although the smaller bias can be intuitively explained, it is more difficult to understand the larger variance and to determine whether this is a general phenomenon. A detailed comparison of the variance of the classic estimators and the $C$-consistent estimators ---both empirical and theoretical--- is a topic of future research.
}

\blue{We point out that the estimator $\widehat{\vc \theta}_n$ is designed to remove the bias caused by conditioning on $Z_n>0$; without this condition, we do not expect $\widehat{\vc \theta}_n$ to be advantageous compared to $\widetilde{\vc \theta}_n$. This is supported by the analysis in Section~\ref{subsec:conterfactual} of Appendix~\ref{appB}.

Finally, since the central limit theorem ---and thus the asymptotic normality of the estimators--- may not hold for moderate values of $n$, we may choose to use bootstrap techniques to derive confidence intervals for the parameters. This is discussed further in Section~\ref{subsec:boot_CI} of Appendix~\ref{appB}.}

\section{Application to the Chatham Island black robins}\label{sec:br}

We now consider the real-world application of the Chatham Island black robins who breed seasonally and compete for territory, and for which a discrete-time PSDBP is an appropriate model. We use a model slightly different from the one considered in Section~\ref{empirical}, 
adapted to the characteristics of the species, such as yearly nesting attempts, the number of offspring per successful nesting attempt, and bird lifetimes between 1970 and 2000 \cite{butlerblack,kennedy2013severe}.

In our model, individuals are birds who are at least one-year old and able to reproduce. For every mother bird in year $n$, the following events are assumed to happen in this specific order between year $n$ and year $n+1$:\begin{itemize}\item [(1)] the mother makes a successful attempt to reproduce with probability $r\equiv r(z,K,v)$ which depends on the current population size $z$, the carrying capacity $K$, and an `efficiency' parameter $v$;
 \item[(2)]
if a reproduction attempt is successful, the mother produces a random number of female offspring who survive to the following year, following a binomial distribution $\vc b=(b_k)$, $k\geq 0$, with parameters $n=5$ and $p=0.1988$ (fitted using the method of moments from reproduction data);
\item[(3)] the mother survives to the next year with probability $1-d=0.6861$ (estimated from lifetime data), which we assume is independent of the current population size and of whether she had a successful reproduction attempt.\end{itemize}
Taking into account that, if the mother survives, then she is counted among her progeny, the effective offspring distribution  is $\vc p=(p_k)$, $k\geq 0$, with
$$
p_0= (1-r) \,d+r\,b_0\,d,\qquad
p_1=(1-r)\,(1-d)+r\,[b_0\,(1-d)+b_1\,d]$$ $$
p_k=r\,[b_{k-1} \,(1-d)+b_k\, d], \quad k\geq 2.
$$
We consider the Beverton-Holt and Ricker models where the probability $r$ of successful reproduction attempt is given by
$$
r(z,K,v)=\dfrac{v \,K}{(\mu-1)z+K} \quad \textrm{(BH model)},\qquad
r(z,K,v)= v\,\left(\frac{1}{\mu}\right)^{\frac{z}{K}},\quad \textrm{(Ricker model),}
$$where the efficiency parameter $v$ is the probability of a successful reproduction attempt in the absence of competition, and $\mu:=5pv/d$.
We estimate the two parameters $v$ and $K$ based on the yearly number of adult females between 1972 and 1998 \cite{davison2021parameter,kennedy2013severe}.

\begin{table} \centering
\begin{tabular}{|c|c|c|c|c|}
\cline{2-5}
\multicolumn{1}{c|}{ }&\multicolumn{2}{c|}{Beverton-Holt} & \multicolumn{2}{c|}{Ricker}\\
\hline
Method & $K$ & $v$ & $K$ & $v$\\
\hline
$\hat{\T}_{n}(\boldsymbol{w}_{n}^{(1)})$    & $335.1382$ & $0.7747$ & $315.3202$ & $0.7181$ \\
$\hat{\T}_{n}(\boldsymbol{w}_{n}^{(2)})$    & $109.2578$ & $0.6989$ & $95.5603$ & $0.6799$ \\
$\tilde{\T}_{n}(\boldsymbol{w}_{n}^{(1)})$  & $162.5136$ & $0.5855$ & $128.7405$ & $0.5870$ \\
$\tilde{\T}_{n}(\boldsymbol{w}_{n}^{(2)})$ & $107.3461$ & $0.7141$ & $94.0461$ & $0.6932$
\\
\hline
\end{tabular}
  \caption{\textsf{Black robin population}. Estimates for the carrying capacity $K$ and the efficiency parameter $v$ obtained via each method. \label{br_results}}

\end{table}

\begin{figure}[h!]
\centering\includegraphics[width=0.8\textwidth]{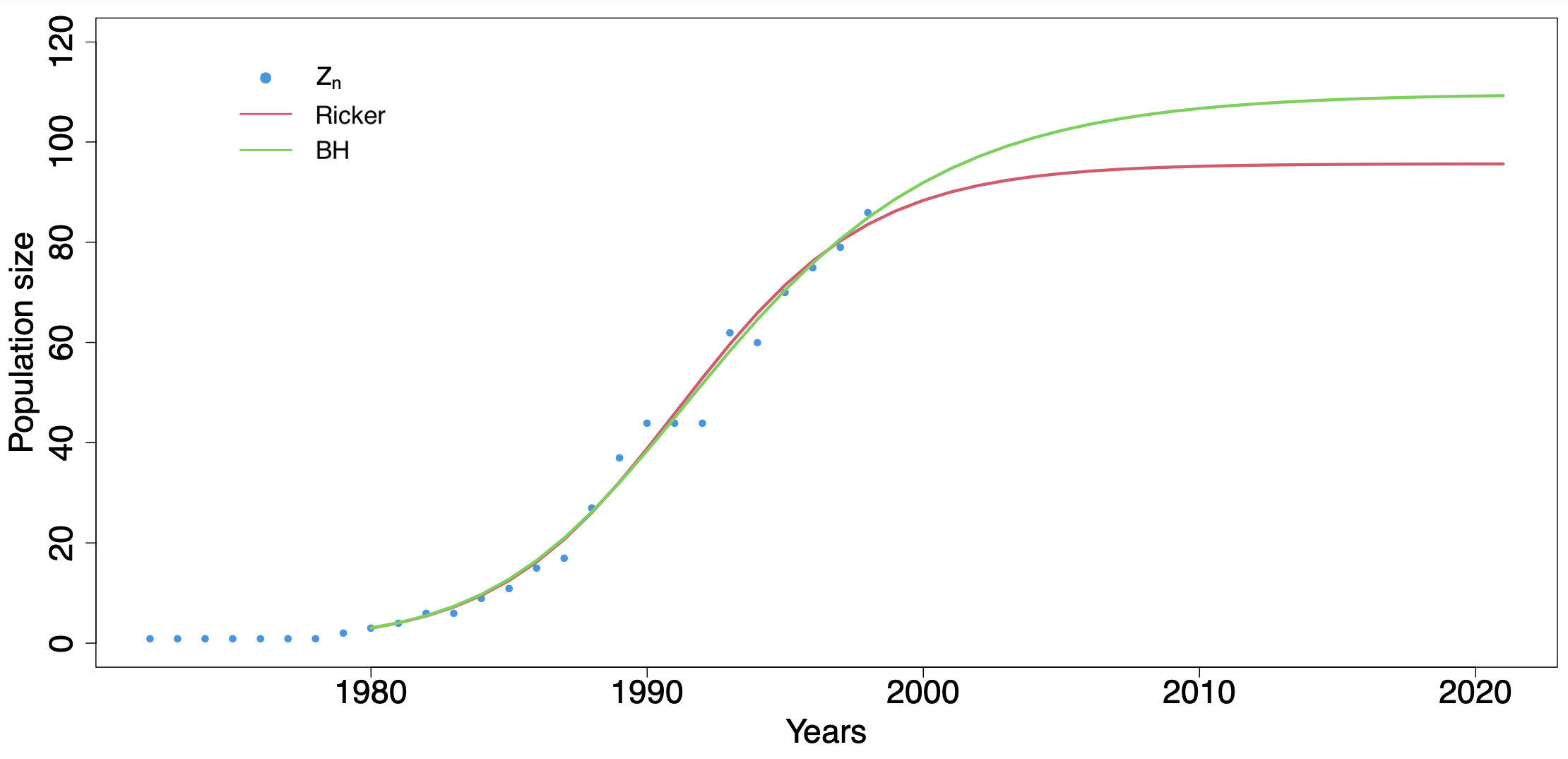}
\caption{\textsf{Black robin population}. Number of adult females between 1972 and 1998 (blue points), with the mean population size curve of the estimated Beverton-Holt model (green curve) and Ricker model (red curve)  obtained with our estimator $\widehat{\vc \theta}_n$ and weights $\boldsymbol{w}_{n}^{(2)}$.\label{br_pic}}
\end{figure}

In Table \ref{br_results} we present the estimates obtained with the estimators $\widehat{\vc \theta}_n$ and
$\widetilde{\vc\theta}_n$ and the weighting functions $\boldsymbol{w}_{n}^{(1)}$ and $\boldsymbol{w}_{n}^{(2)}$.
In Figure \ref{br_pic} we plot the yearly population counts of the number of adult females, together with the mean population size curves of the Beverton-Holt model and Ricker model estimated using the weighting function $\boldsymbol{w}_{n}^{(2)}$.

From Table \ref{br_results} we see that, for both estimators $\widehat{\vc \theta}_n$ and
$\widetilde{\vc\theta}_n$, there is a large difference in the estimates obtained using the two weighting functions. This is because the number of breeding females was one for seven years in a row (see Figure~\ref{br_pic}), and $\boldsymbol{w}_{n}^{(1)}$ places a large weight on these observations. As highlighted in Section \ref{growing_pop}, $\boldsymbol{w}_{n}^{(1)}$  does not generally provide results as accurate as $\boldsymbol{w}_{n}^{(2)}$, so here, we rely more on the results obtained with $\boldsymbol{w}_{n}^{(2)}$.

We also observe from Table \ref{br_results} that the estimates are quite sensitive to the model. This is also highlighted in Figure \ref{br_pic}, especially for the years after 1998 where the models are used for making predictions. The black robin population reached 239 adults in 2011 (i.e. $\sim 120$ females) \cite{massaro2013nest}, so it appears that the Beverton-Holt model is more appropriate here than the Ricker model. \blue{We provide bootstrap confidence intervals for the parameters of the Beverton-Holt model with weight $\boldsymbol{w}_{n}^{(2)}$ in Section~\ref{subsec:boot_CI} of Appendix~\ref{appB}.} Several other PSDBP models could be fitted to the black robin data, and a thorough model selection analysis should then be performed, in collaboration with experts in ecology. 

\blue{

\section{Discussion}\label{disc}

In this paper we described a class of estimators for parametric PSDBPs which account for a common observation bias: when we study a population for conservation purposes, its observed population sizes are sampled under the condition $Z_n>0$. We proposed a least squares method where the mean offspring at population size $z$, $m(z,\T)$, is replaced with its equivalent in the $Q$-process, $m^\uparrow (z,\T)$. We considered two weighting functions, $\{\hat{w}_{n}^{(1)}(z)\}$ and $\{\hat{w}_{n}^{(2)}(z)\}$; we generally recommend applying $\{\hat{w}_{n}^{(2)}(z)\}$ in practice, as it is approximately inversely proportional to the variance of the error $\{\hat{m}_n(z)-{m}^{\uparrow}(z, \T)\}$, which is known to be asymptotically  optimal. The bias corrected by our estimators is more pronounced when the observed population is endangered, for example, because its 
population size is low. It is in these situations that our estimators should be considered in practice. In contrast, when observed population sizes are higher, then we would generally expect $m(z, \T) \approx m^\uparrow (z, \T)$, leading to very similar estimates to those obtained with classical least squares estimators.

We supported our methodology by establishing $C$-consistency and asymptotic normality of the new class of estimators, which can be used to derive confidence regions.
While our estimators rely on asymptotic properties ---because $m^\uparrow(z, \T_0)$ is the conditional limit of $\hat{m}_n(z)$, given $Z_n>0$--- we illustrated that they also perform well for moderate values of $n$.

We expect that our approach can be extended to other stochastic population processes with
an absorbing state at zero, for which the $Q$-process and its stationary distribution exist.
Examples include certain diffusion models, continuous-time birth-death processes, controlled branching processes, multitype branching processes, or processes in a random environment.
For these processes, the data would first need to be summarized by selecting appropriate statistics (here the MLEs $\hat{m}_n(z)$ for all $z$) that converge to quantities that can be expressed in terms of the transition probabilities of the $Q$-process (here ${m}^{\uparrow}(z,\T_0)$). A least squares estimator would then be constructed by minimising the weighted sum of the squared error terms (here $\{\hat{m}_n(z)-m^\uparrow(z,\T)\}$). Asymptotic properties of that estimator would then need to be derived from those of the summary statistics. Following this approach would require overcoming similar technical challenges to those in \cite{braunsteins2022parameter} and in this paper.
}

\section{Proofs}
\label{sec:proofs}

\noindent\textbf{Proof of Lemma \ref{weights}.}
The fact that $\{\hat{w}_n^{(1)}(z)\}$ satisfies Assumption (A1) is a consequence of \cite[Theorem 3(c)]{Gosselin-2001} (see also \cite[Lemma 1]{braunsteins2022parameter}).

The fact that $\{\hat{w}_n^{(2)}(z)\}$ satisfies Assumption (A1) follows from the fact that
\begin{equation}\label{lem1p}\hat{w}^{(2)}_n(z)=\frac{z \sum_{i=0}^{n-1}\ind{Z_i=z}}{\sum_{k=1}^n Z_k}=z\dfrac{\hat{w}_n^{(1)}(z)}{\sum_{k=1}^n Z_k/n}.\end{equation} Conditional on $Z_n>0$, $\hat{w}_n^{(1)}(z)$ converges in probability to $u_zv_z$, and the denominator converges to $\sum_{k=1}^\infty k u_zv_z$ in probability by \cite[Lemma 6]{braunsteins2022parameter}. The result then follows from Slutsky's theorem. \qed

\noindent\textbf{Proof of Proposition \ref{wls_equiv}.}
We prove the result for the weights $\{\hat{w}_n^{(2)}(z)\}$; the proof is similar for the other weights. Letting $j_z:=\sum_{i=0}^{n-1}\ind{Z_i=z}$ and using \eqref{weights_def}, we thus show that the estimator
\begin{equation*}
\widehat{\vc\theta}_n  =\arg\min_{\theta\in\Theta} \sum_{z=1}^\infty \frac{z j_z}{\sum_{i=0}^{n-1} Z_i} \left\{\hat{m}_n(z)-{m}^{\uparrow}(z, \T)\right\}^2
\end{equation*}
is equal to the `modified' classic estimator
\begin{equation*}\label{cwlse}
\widehat{\T}_n^*=\arg\min_{\T \in \Theta} \sum_{k =1}^{n} Z_{k-1}^{-1}\left\{Z_k - Z_{k-1}\,m^\uparrow(Z_{k-1}, \T)\right\}^2.
\end{equation*}
First, by using \eqref{MLEmz}, we obtain
\begin{eqnarray}\nonumber
\widehat{\vc\theta}_n  &=&\arg\min_{\theta\in\Theta} \sum_{z=1}^\infty \frac{z j_z}{\sum_{i=0}^{n-1} Z_i} \left\{\frac{\sum_{i=0}^{n-1} Z_{i+1}\ind{Z_i=z}}{z\,j_z}-{m}^{\uparrow}(z, \T)\right\}^2\\ &=&\arg\min_{\theta\in\Theta} \sum_{z=1}^\infty \frac{1}{z j_z} \left\{\sum_{i=0}^{n-1} Z_{i+1}\ind{Z_i=z}-z j_z{m}^{\uparrow}(z, \T)\right\}^2.
\end{eqnarray}
Next, if we define $J_z:=\{1\leq k_z\leq n: Z_{k_z-1}=z\}$ for $z\geq 1$ (so that $j_z=|J_z|$), then
\begin{eqnarray}
\widehat{\T}_n^*&=&\arg\min_{\T \in \Theta} \sum_{z=1}^\infty \sum_{k_z\in J_z} \frac{1}{z}\left\{Z_{k_z} - z\,m^\uparrow(z, \T)\right\}^2,
\end{eqnarray}
and we are left with comparing
$$I(z,\mathcal{Z},\T):=\sum_{k_z\in J_z} \left\{Z_{k_z} - z\,m^\uparrow(z, \T)\right\}^2\quad \text{to}\quad J(z,\mathcal{Z},\T):=\frac{1}{j_z} \left\{\sum_{i=0}^{n-1} Z_{i+1}\ind{Z_i=z}-z j_z{m}^{\uparrow}(z, \T)\right\}^2.$$
We have
\begin{eqnarray*}
I(z,\mathcal{Z},\T)&=& \sum_{k_z\in J_z}Z_{k_z}^2 - 2 z\,m^\uparrow(z, \T)\sum_{k_z\in J_z}Z_{k_z}+j_z z^2\,(m^\uparrow(z, \T))^2.
\end{eqnarray*}
On the other hand,
\begin{eqnarray*}
J(z,\mathcal{Z},\T)&=& \frac{1}{j_z} \left\{\sum_{k_z\in J_z} Z_{k_z}-z j_z{m}^{\uparrow}(z, \T)\right\}^2\\&=& \frac{1}{j_z} \left(\sum_{k_z\in J_z} Z_{k_z}\right)^2-2 z {m}^{\uparrow}(z, \T)\sum_{k_z\in J_z} Z_{k_z} + z^2 j_z (m^\uparrow(z, \T))^2,
\end{eqnarray*}
so we conclude that
$$I(z,\mathcal{Z},\T)=J(z,\mathcal{Z},\T)+ C(z,\mathcal{Z}), \quad \text{where} \quad C(z,\mathcal{Z}):=
\sum_{k_z\in J_z} Z_{k_z}^2 -  \frac{1}{j_z} \left(\sum_{k_z\in J_z} Z_{k_z}\right)^2$$
is independent of the parameter $\T$. This means that
\begin{eqnarray*}
\widehat{\T}_n^*&=&\arg\min_{\T \in \Theta} \sum_{z=1}^\infty \frac{1}{z}I(z,\mathcal{Z},\T)\\&=&\arg\min_{\T \in \Theta} \sum_{z=1}^\infty  \frac{1}{z}\{J(z,\mathcal{Z},\T)+ C(z,\mathcal{Z})\}\\&=&\arg\min_{\T \in \Theta} \left\{\sum_{z=1}^\infty  \frac{1}{z}J(z,\mathcal{Z},\T)+ \sum_{z=1}^\infty  \frac{1}{z} C(z,\mathcal{Z})\right\}\\&=& \arg\min_{\T \in \Theta} \sum_{z=1}^\infty  \frac{1}{z}J(z,\mathcal{Z},\T)=\widehat{\T}_n.
\end{eqnarray*}
\qed



\medskip

To prove Theorems \ref{thm:C-consistency-WLSE-1} and \ref{thm:C-normality-WLSE-1} we require the following two lemmas. 

\begin{lem}\label{lem-sum}If the weights $\{\hat{w}_n(z)\}_{z\geq 1}$ satisfy Assumption (A1) then
for each $\epsilon>0$, \label{lem:conv-qnz-wz-i}
\begin{equation*}\label{eq:unif-conv-qnz-wz}
\lim_{n\to\infty}\mbP_i\left[\sum_{z=1}^\infty \left|\hat{w}_n(z)-w_z\right|>\epsilon|Z_n>0\right]=0.
\end{equation*}

\end{lem}

\noindent\textbf{Proof of Lemma \ref{lem-sum}.}
First recall that $\{\hat{w}_n(z)\}$ and $\{w_z\}$ are probability distributions. Therefore, for any $\varepsilon>0$, we can choose $z^*(\varepsilon)$ such that
\begin{equation}\label{lem2a}\sum_{z=z^*(\varepsilon)+1}^{\infty} w_z \leq \frac{\varepsilon}{4}.\end{equation} Observe that if
\begin{equation}\label{lem2b}
|\hat w_n(z) - w_z| < \frac{\varepsilon}{4 z^*(\varepsilon)}, \quad \text{for all } z \in \{ 1, \dots , z^*(\varepsilon)\},
\end{equation}
then
\begin{align*}
\sum_{z=1}^\infty | \hat w_n(z) - w_z | &\leq \sum^{z^*(\varepsilon)}_{z=1} |\hat w_n(z) - w_z|+ \sum^\infty_{z=z^*(\varepsilon)+1} w_z  + \sum^\infty_{z=z^*(\varepsilon)+1}  \hat w_n(z) \\
&\leq \frac{\varepsilon}{4}+ \frac{\varepsilon}{4} + 1 - \sum^{z^*(\varepsilon)}_{z=1} \hat w_n(z)  \\
&\leq \frac{\varepsilon}{2} + 1 - \sum^{z^*(\varepsilon)}_{z=1} w_z + \sum^{z^*(\varepsilon)}_{z=1} |\hat w_n(z) - w_z| \\
& \leq \varepsilon,
\end{align*}
where we have applied the triangle inequality, \eqref{lem2a}, \eqref{lem2b}, and the fact that $\{\hat{w}_n(z)\}$ is a probability distribution.
Consequently,
\begin{align*}
\lim_{n\to\infty} &\mbP_i\left[\sum_{z=1}^\infty \left|\hat{w}_n(z)-w_z\right|>\epsilon|Z_n>0\right] \\
&\leq \lim_{n\to\infty} \mbP_i\left[\left|\hat{w}_n(z)-w_z\right|>\frac{\varepsilon}{4 z^*(\varepsilon)}, \text{ for some } z = 1, \dots, z^*(\varepsilon)|Z_n>0\right] \\
&\leq \sum_{i=1}^{z^*(\varepsilon)} \lim_{n\to\infty}  \mbP_i\left[\left|\hat{w}_n(z)-w_z\right|>\frac{\varepsilon}{4 z^*(\varepsilon)}|Z_n>0\right] \\
&=0,
\end{align*}
where in the final step we apply Assumption (A1).\qed

\begin{lem}\label{lem:help-consistency}
Let $i\in\N$ be an initial state and $f:\T\in\Theta\mapsto [0,\infty)$ be a continuous function \blue{with a unique minimum at $\T_0$}. For each $n\in\N$, let $f_n:(\omega,\T)\in\Omega\times\Theta\mapsto \R_+$ be a function such that for each $\omega\in\Omega$, $f_n(\omega,\cdot)$ is a continuous function on $\Theta$, and for each $\T\in\Theta$, $f_n(\cdot,\T)$ is a measurable function of the random variables $Z_0,\ldots,Z_n$. If
\begin{align}\label{eq:hip-help-consistency}
\forall\epsilon>0,\quad\lim_{n\to\infty} \mbP_i\left[\omega\in\Omega:\sup_{\T\in\Theta}|f_n(\omega,\T)-f(\T)|>\epsilon|Z_n>0\right]=0,
\end{align}
then
\begin{enumerate}[label=(\roman*),ref=\emph{(\roman*)}]
\item \blue{$\lim_{n\to\infty} \mbP_i\left[\omega\in\Omega: \exists \vc X(\omega)\in \Theta \text{ satisfying } f_n(\omega,\vc X(\omega))=\min_{\T\in\Theta}f_n(\omega,\T)|Z_{n}>0\right]=1.$}\label{lem:C-existence}

\item $\forall\epsilon>0,\quad\lim_{n\to\infty} \mbP_i\left[\omega\in\Omega:\big\|\arg\min_{\T\in\Theta}f_n(\omega,\T)-\T_0\big\|>\epsilon|Z_{n}>0\right]=0.$\label{lem:C-consistency}
\end{enumerate}
\end{lem}

\noindent\textbf{Proof of Lemma \ref{lem:help-consistency}.}
To ease the notation, throughout this proof let us denote $f_n(\T)=f_n(\cdot,\T)$.  \blue{Without loss of generality we can assume that $f(\T_0)=0$; if this does not hold, then we consider the functions $f_n(\cdot)-f(\T_0)$ and $f(\cdot)-f(\T_0)$ instead of $f_n(\cdot)$ and $f(\cdot)$.}
To establish \emph{(i)} and \emph{(ii)} we apply the following two identities that hold under \eqref{eq:hip-help-consistency}: for any $\varepsilon >0$, and $C \subseteq \Theta$ we have
\begin{align}
\lim_{n\to\infty} \mbP_i\left[f_n(\T_0)>\epsilon|Z_n>0\right]&=0,\label{eq:help-lim1-con-Tn}\\
\lim_{n\to\infty} \mbP_i\bigg[\big|\inf_{\T\in C}f_n(\T)-\inf_{\T\in C}f(\T)\big|>\epsilon|Z_n>0\bigg]&=0.\label{eq:help-lim2-con-Tn}
\end{align}
To establish \eqref{eq:help-lim1-con-Tn} we use the fact that, since $f(\T_0)=0$, for any $n \in \mathbb{N}$ we have
\[
\{ \omega \in \Omega : f_n(\T_0) > \varepsilon\} = \{ \omega \in \Omega: |f_n(\T_0) -f(\T_0)| > \varepsilon \} \subseteq \{\omega \in \Omega: \sup_{\T \in \Theta} | f_n(\T) - f(\T) | > \varepsilon \},
\]
which, by \eqref{eq:hip-help-consistency}, implies
\begin{align*}
\lim_{n\to\infty} \mbP_i\bigg[f_n(\T_0)>\epsilon|Z_n>0\bigg]\leq \lim_{n\to\infty} \mbP_i\bigg[\sup_{\T\in\Theta} \big|f_n(\T)-f(\T)\big|>\epsilon|Z_n>0\bigg]=0.
\end{align*}
To establish \eqref{eq:help-lim2-con-Tn}, we denote $\T_C=\arg\inf_{\T\in C}f(\T)$ and $\T_{n,C}=\arg\inf_{\T\in C}f_n(\T)$. For any $\T' \in C$ we have (surely)
\begin{align*}
\inf_{\T\in C}f_n(\T)-\inf_{\T\in C}f(\T)\leq f_n(\T')-\inf_{\T\in C}f(\T),
\end{align*}
and consequently,
\begin{align*}
\inf_{\T\in C}f_n(\T)-\inf_{\T\in C}f(\T)\leq f_n(\T_C)-\inf_{\T\in C}f(\T)= f_n(\T_C)-f(\T_C)\leq \sup_{\T\in \Theta}\big|f_n(\T)-f(\T)\big|.
\end{align*}
Similarly,
\begin{align*}
\inf_{\T\in C}f(\T)-\inf_{\T\in C}f_n(\T)\leq f(\T_{n,C})-\inf_{\T\in C}f_n(\T)= f(\T_{n,C})-f_n(\T_{n,C})\leq \sup_{\T\in \Theta}\big|f_n(\T)-f(\T)\big|.
\end{align*}
Combining both the above inequalities we see that for any $n \in \mathbb{N}$ and $\epsilon>0$,
\begin{align*}
\bigg\{\big|\inf_{\T\in C}f_n(\T)-\inf_{\T\in C}f(\T)\big|>\epsilon\bigg\}\subseteq\bigg\{\sup_{\T\in \Theta}\big|f_n(\T)-f(\T)\big|>\epsilon\bigg\},
\end{align*}
which then allows us to apply \eqref{eq:hip-help-consistency} to obtain \eqref{eq:help-lim2-con-Tn}.

To establish \emph{(i)} we observe that it is sufficient to show that, for any compact set $C \subseteq \Theta$ with $\T_0 \in C$, we have
\begin{align}\label{eq:help-lim3-con-Tn}
\lim_{n\to\infty} \mbP_i\bigg[\inf_{\T\in \Theta\backslash C}f_n(\T)>\min_{\T\in C}f_n(\T)|Z_n>0\bigg]=1.
\end{align}
Note that, because $C$ is a compact set, and because for any $n \in \mathbb{N}$ and $\omega \in \Omega$, $\T \mapsto f_n(\omega,\T)$ is a continuous function on $\Theta$, $\min_{\T\in C}f_n(\T)$ exists.
The claim in \emph{(i)} then follows from \eqref{eq:help-lim3-con-Tn} by letting $C$ be an arbitrarily small closed ball centered at $\T_0$.

To establish \eqref{eq:help-lim3-con-Tn} we let $0<\varepsilon=a/4$, with $a=\inf_{\T\in \Theta\backslash C}f(\T)>0$, and observe that
\begin{align*}
\bigg\{f_n(\T_0)\leq\epsilon\bigg\}\cap\bigg\{\Big|\inf_{\T\in \Theta\backslash C}f_n(\T)-a\Big|\leq\epsilon\bigg\}\subseteq\bigg\{\min_{\T\in C}f_n(\T)\leq f_n(\T_0)\leq\frac{a}{4}<\frac{3a}{4}\leq\inf_{\T\in \Theta\backslash C}f_n(\T)\bigg\}.
\end{align*}
We then apply \eqref{eq:help-lim1-con-Tn} and \eqref{eq:help-lim2-con-Tn} to obtain \eqref{eq:help-lim3-con-Tn}.

\vspace{2ex}

To establish \emph{(ii)} we let $\boldsymbol{\vartheta}_n=\arg\min_{\T\in\Theta}f_n(\T)$, and use the fact that
\begin{align}
\forall\epsilon>0,\quad \lim_{n\to\infty} \mbP_i\left[f(\boldsymbol{\vartheta}_n)>\epsilon|Z_n>0\right]=0.\label{eq:help-C-consistency-weighted-fn}
\end{align}
Indeed, \eqref{eq:help-C-consistency-weighted-fn} follows by applying both \eqref{eq:help-lim2-con-Tn} with $C=\Theta$, to obtain
\begin{align*}
\forall\epsilon>0,\quad \lim_{n\to\infty} \mbP_i\left[f_n(\boldsymbol{\vartheta}_n)>\epsilon|Z_n>0\right]=0,
\end{align*}
and \eqref{eq:hip-help-consistency} to obtain
\begin{align*}
\forall\epsilon>0,\quad \lim_{n\to\infty} \mbP_i\Big[\big|f_n(\boldsymbol{\vartheta}_n)-f(\boldsymbol{\vartheta}_n)\big|>\epsilon|Z_n>0\Big]=0.
\end{align*}
To prove \ref{lem:C-consistency}, we then note that, for each $\eps>0$, given that $f(\cdot)$ has a unique minimum at $\T_0$, we can take $C=\{\T\in\Theta:\|\T-\T_0\|\leq \eps\}$, $b=\inf_{\Theta\backslash C} f(\T)$, and $\beta=b/2$; then $0=\min_{\T\in C} f(\T)<\beta<b$. Moreover, we have $b=\inf_{\Theta\backslash C} f(\T)\leq f(\T')$ for every $\T'\in\Theta\backslash C$, and consequently
$$\{\|\boldsymbol{\vartheta}_n-\T_0\|>\eps\}\subseteq \{f(\boldsymbol{\vartheta}_n)>\beta\}.$$
The result in \ref{lem:C-consistency} now follows from \eqref{eq:help-C-consistency-weighted-fn}.
\qed

\medskip

We are now ready to establish the limiting properties of the estimator $\widehat{\T}_n$, namely, $C$-consistency (Theorem \ref{thm:C-consistency-WLSE-1}) and $C$-normality (Theorem \ref{thm:C-normality-WLSE-1}).
For each $n \in \N$, recall that
$$\widehat{\T}_n=\arg\min_{\T\in\Theta} S_n(\T),\qquad \textrm{where}\qquad
 S_{n}(\T):=\sum_{z=1}^\infty \hat{w}_n(z) (m^\uparrow(z,\T)-\hat{m}_n(z))^2.$$

\noindent\textbf{Proof of Theorem \ref{thm:C-consistency-WLSE-1}.}
Let us fix $i\in\N$. By Lemma \ref{lem:help-consistency}, it is enough to prove that
\begin{align*}
\forall\epsilon>0,\quad\lim_{n\to\infty} \mbP_i\left[\sup_{\T\in\Theta}\left|S_{n}(\T)-S(\T)\right|>\epsilon|Z_{n}>0\right]=0,
\end{align*}where
$$S(\T):=\sum_{z=1}^\infty w_z (\T_0) (m^\uparrow(z,\T)-m^\uparrow(z,\T_0))^2.$$
To that end, we
introduce the following function
$$\hat{S}_{n}(\T)=\sum_{z=1}^\infty w_z(\T_0)\ind{j_n(z)>0} (m^\uparrow(z,\T)-\hat{m}_n(z))^2,\quad n\in\N,\ \T\in\Theta,$$ where $j_n(z):=\sum_{i=0}^{n-1} \ind{Z_i=z}$.
It then suffices to prove that, for each $\epsilon>0$,
\begin{enumerate}[label=\emph{(\alph*)},ref=\emph{(\alph*)}]
\item $\lim_{n\to\infty} \mbP_i\left[\sup_{\T\in\Theta}\left|S_{n}(\T)-\hat{S}_{n}(\T)\right|>\epsilon|Z_{n}>0\right]=0$,\label{eq:lim-sup-C-consist-weighted1-help1}
\item $\lim_{n\to\infty} \mbP_i\left[\sup_{\T\in\Theta}\left|\hat{S}_{n}(\T)-S(\T)\right|>\epsilon|Z_{n}>0\right]=0$.\label{eq:lim-sup-C-consist-weighted1-help2}
\end{enumerate}

We start with the proof of \ref{eq:lim-sup-C-consist-weighted1-help1}. Recall the definition of $M$ and $C$ from Assumptions \ref{cond:m-arrow-unif-bounded} and \ref{cond:m-arrow-hat-unif-bounded}, respectively. Let $M'=2(M+C)^2$ and $A_n=\sup_{z\in\N} |m^\uparrow(z,\T_0)-\hat{m}_n(z)|$. On the event $\{A_n\leq C\}$ we have
\begin{align*}
\sup_{\T\in\Theta} |S_n(\T)&-\hat{S}_n(\T)|\leq\sup_{\T\in\Theta} \sum_{z=1}^\infty |\hat{w}_n(z)-w_z(\T_0)\ind{j_n(z)>0}| (m^\uparrow(z,\T)-\hat{m}_n(z))^2\\
&=\sup_{\T\in\Theta} \sum_{z=1}^\infty |\hat{w}_n(z)-w_z(\T_0)\ind{j_n(z)>0}| (m^\uparrow(z,\T)-m^\uparrow(z,\T_0)+m^\uparrow(z,\T_0)-\hat{m}_n(z))^2\\
&\leq \sup_{\T\in\Theta} \sum_{z=1}^\infty |\hat{w}_n(z)-w_z(\T_0)\ind{j_n(z)>0}| (m^\uparrow(z,\T)-m^\uparrow(z,\T_0))^2\\
&\phantom{=}+ \sum_{z=1}^\infty |\hat{w}_n(z)-w_z(\T_0)\ind{j_n(z)>0}| (m^\uparrow(z,\T_0)-\hat{m}_n(z))^2\\
&\phantom{=}+2\sup_{\T\in\Theta} \sum_{z=1}^\infty |\hat{w}_n(z)-w_z(\T_0)\ind{j_n(z)>0}|\cdot |m^\uparrow(z,\T)-m^\uparrow(z,\T_0)|\cdot |m^\uparrow(z,\T_0)-\hat{m}_n(z)|\\
&\leq M' \sum_{z=1}^\infty |\hat{w}_n(z)-w_z(\T_0)\ind{j_n(z)>0}|\\
&\leq M' \sum_{z=1}^\infty |\hat{w}_n(z)-w_z(\T_0)+w_z(\T_0)-w_z(\T_0)\ind{j_n(z)>0}|\\
&\leq M'  \sum_{z=1}^\infty |\hat{w}_n(z)-w_z(\T_0)|+M'  \sum_{z=1}^\infty w_z(\T_0)\ind{j_n(z)=0}.
\end{align*}
Consequently,
\begin{align*}
\left\{\sup_{\T\in\Theta} |S_n(\T)-\hat{S}_n(\T)|>\epsilon\right\}&=\left\{\sup_{\T\in\Theta} |S_n(\T)-\hat{S}_n(\T)|>\epsilon,A_n\leq C\right\}\cup\\
&\phantom{=}\cup\left\{\sup_{\T\in\Theta} |S_n(\T)-\hat{S}_n(\T)|>\epsilon,A_n> C\right\}\\
&\subseteq\left\{\sum_{z=1}^\infty |\hat{w}_n(z)-w_z(\T_0)|>\frac{\epsilon}{2M'},A_n\leq C\right\}\cup\\
&\phantom{=}\cup\left\{\sum_{z=1}^\infty w_z(\T_0)\ind{j_n(z)=0}>\frac{\epsilon}{2M'},A_n\leq C\right\} \cup\left\{A_n> C\right\}\\
&\subseteq\left\{\sum_{z=1}^\infty |\hat{w}_n(z)-w_z(\T_0)|>\frac{\epsilon}{2M'}\right\}\cup\left\{\sum_{z=1}^\infty w_z(\T_0)\ind{j_n(z)=0}>\frac{\epsilon}{2M'}\right\}\\
&\phantom{=}\cup\left\{A_n> C\right\}.
\end{align*}
Conditioning on $Z_n>0$, we then obtain \ref{eq:lim-sup-C-consist-weighted1-help1} by applying Lemma \ref{lem-sum} to the first event, the fact that $\lim_{n\to\infty} \mbP[\ind{j_n(z)=0}\leq\epsilon|Z_n>0]=1$ for any $\epsilon>0,z\in\N$ and that $\sum_{z=1}^\infty w_z(\T_0)=1$ to the second event, and Assumption \ref{cond:m-arrow-hat-unif-bounded} to the third event.

\vspace{2ex}

In order to prove \ref{eq:lim-sup-C-consist-weighted1-help2}, we observe that
\begin{align*}
\hat{S}_{n}(\T)-S(\T)&= \sum_{z=1}^\infty w_z(\T_0) \ind{j_n(z)>0}\left[(m^\uparrow(z,\T)-\hat{m}_n(z))^2-(m^\uparrow(z,\T)-m^\uparrow(z,\T_0))^2\right]\\
&\phantom{=}-\sum_{z=1}^\infty w_z(\T_0) \ind{j_n(z)=0}(m^\uparrow(z,\T)-m^\uparrow(z,\T_0))^2\\
&= \sum_{z=1}^\infty w_z(\T_0) \ind{j_n(z)>0}\left[\hat{m}_n(z)^2 -2\hat{m}_n(z)m^\uparrow(z,\T)-m^\uparrow(z,\T_0)^2+2m^\uparrow(z,\T)m^\uparrow(z,\T_0)\right]\\
&\phantom{=}-\sum_{z=1}^\infty w_z(\T_0) \ind{j_n(z)=0}(m^\uparrow(z,\T)-m^\uparrow(z,\T_0))^2,
\end{align*}
then
\begin{align*}
\sup_{\T\in\Theta}|\hat{S}_{n}(\T)-S(\T)|&\leq  \sum_{z=1}^\infty w_z(\T_0) \ind{j_n(z)>0}\left|\hat{m}_n(z)^2 -m^\uparrow(z,\T_0)^2\right|\\
&\phantom{=}+2\sup_{\T\in\Theta}\sum_{z=1}^\infty w_z(\T_0) \ind{j_n(z)>0}m^\uparrow(z,\T)\left|\hat{m}_n(z)-m^\uparrow(z,\T_0)\right|\\
&\phantom{=}+\sup_{\T\in\Theta}\sum_{z=1}^\infty w_z(\T_0) \ind{j_n(z)=0}(m^\uparrow(z,\T)-m^\uparrow(z,\T_0))^2,
\end{align*}
and consequently,
\begin{align*}
\left\{\sup_{\T\in\Theta}|\hat{S}_{n}(\T)-S(\T)|>\epsilon\right\}&\subseteq\left\{\sum_{z=1}^\infty w_z(\T_0) \ind{j_n(z)>0}\left|\hat{m}_n(z)^2 -m^\uparrow(z,\T_0)^2\right|>\frac{\epsilon}{3}\right\}\cup\\
&\phantom{\subseteq}\cup\left\{\sup_{\T\in\Theta}\sum_{z=1}^\infty w_z(\T_0) \ind{j_n(z)>0}m^\uparrow(z,\T)\left|\hat{m}_n(z)-m^\uparrow(z,\T_0)\right|>\frac{\epsilon}{6}\right\}\\
&\phantom{\subseteq}\cup\left\{\sup_{\T\in\Theta}\sum_{z=1}^\infty w_z(\T_0) \ind{j_n(z)=0}(m^\uparrow(z,\T)-m^\uparrow(z,\T_0))^2>\frac{\epsilon}{3}\right\}
\end{align*}

Because $\mathbb{P}_i(A_n>C | Z_n >0) \to 0$ by \ref{cond:m-arrow-hat-unif-bounded}, we can restrict our attention to the event $\{A_n \leq C\}$.
 On the event $\{A_n\leq C\}$,
\begin{align*}
\sup_{z\in\N}\left|\hat{m}_n(z)^2 -m^\uparrow(z,\T_0)^2\right|&\leq \sup_{z\in\N}\left|\hat{m}_n(z) -m^\uparrow(z,\T_0)\right|\cdot\sup_{z\in\N}\left|\hat{m}_n(z)+m^\uparrow(z,\T_0)\right| \\
&\leq C(C+2M)\leq M'.
\end{align*}
In addition, since ${\vc w}(\T_0)$ is a probability distribution, given $\epsilon>0$, there exists $K_0=K_0(\epsilon)\in\N$ such that $\sum_{z=K_0+1}^\infty w_z(\T_0)<\frac{\epsilon}{6M'}$. Consequently, on $\{A_n\leq C\}$
\begin{align*}
\sum_{z=1}^\infty w_z(\T_0) \ind{j_n(z)>0}\left|\hat{m}_n(z)^2 -m^\uparrow(z,\T_0)^2\right|&\leq \sum_{z=1}^{K_0} w_z(\T_0) \ind{j_n(z)>0}\left|\hat{m}_n(z)^2 -m^\uparrow(z,\T_0)^2\right|\\
&\phantom{=} + M'\sum_{z=K_0+1}^\infty w_z(\T_0),
\end{align*}
and hence
\begin{align*}
\bigcap_{z=1}^{K_0}\bigg\{w_z(\T_0) \ind{j_n(z)>0}&\left|\hat{m}_n(z)^2 -m^\uparrow(z,\T_0)^2\right|\leq \frac{\epsilon}{6K_0}\bigg\}\bigcap \{A_n\leq C\}\subseteq\\
&\subseteq \Big\{\sum_{z=1}^\infty w_z(\T_0) \ind{j_n(z)>0}\left|\hat{m}_n(z)^2 -m^\uparrow(z,\T_0)^2\right|\leq \frac{\epsilon}{3}\Big\}\bigcap \{A_n\leq C\}.
\end{align*}
Thus, by \eqref{eq:Q-consistency-m}, we get that
\begin{align*}
\lim_{n\to\infty}\mbP_i\bigg[\Big\{\sum_{z=1}^\infty w_z(\T_0) \ind{j_n(z)>0}\left|\hat{m}_n(z)^2 -m^\uparrow(z,\T_0)^2\right|\leq \frac{\epsilon}{3}\Big\}\cap \{A_n\leq C\}|Z_n>0\bigg]=1,
\end{align*}
and then
\begin{align*}
\lim_{n\to\infty}\mbP_i\bigg[\sum_{z=1}^\infty w_z(\T_0) \ind{j_n(z)>0}\left|\hat{m}_n(z)^2 -m^\uparrow(z,\T_0)^2\right|\leq \frac{\epsilon}{3}|Z_n>0\bigg]=1.
\end{align*}
With similar arguments using $\sup_{\T\in\Theta}\sup_{z\in\N}\left(m^\uparrow(z,\T)-m^\uparrow(z,\T_0)\right)^2\leq 4M^2$, and we obtain
\begin{align*}
\lim_{n\to\infty}\mbP_i\bigg[\sup_{\T\in\Theta}\sum_{z=1}^\infty w_z(\T_0) \ind{j_n(z)>0}m^\uparrow(z,\T)\left|\hat{m}_n(z)-m^\uparrow(z,\T_0)\right|\leq \frac{\epsilon}{6}|Z_n>0\bigg]=1,\\
\lim_{n\to\infty}\mbP_i\bigg[\sup_{\T\in\Theta}\sum_{z=1}^\infty w_z(\T_0) \ind{j_n(z)=0}(m^\uparrow(z,\T)-m^\uparrow(z,\T_0))^2\leq \frac{\epsilon}{3}|Z_n>0\bigg]=1.
\end{align*}
Consequently we have proved \ref{eq:lim-sup-C-consist-weighted1-help2} and the result now follows.\qed

\medskip

Before we prove Theorem \ref{thm:C-normality-WLSE-1} we recall the asymptotic normality of the MLEs $\{\hat{m}_n(z)\}_{z \in \N}$ that is established in
\cite[Theorem 1]{braunsteins2022parameter}: for any  $i,z,z_1,z_2\in\N$ and $x,x_1,x_2\in\R$,
\begin{align}
\lim_{n\to\infty}\mbP_i[\sqrt{n}(\hat{m}_{n}(z)-m^\uparrow(z,\T_0))\leq x|Z_n>0]&=\blue{\Phi}_{\gamma(z,\T_0)}(x),\label{eq:Q-normality-m}\\
\lim_{n\to\infty}\mbP_i[\sqrt{n}(\hat{m}_{n}(z_i)-m^\uparrow(z_i,\T_0))\leq x_i,\ i=1,2|Z_n>0]&=\blue{\Phi}_{\Gamma(z_1,z_2,\T_0)}(x_1,x_2),\label{eq:Q-indep-m}
\end{align}
where $\gamma(z,\T_0)$ is defined in \eqref{gamma} and
\begin{align}\label{eq:def-gamma}
\Gamma(z_1,z_2,\T_0):=\left(\begin{array}{cc}
   \gamma(z_1,\T_0) & 0 \\
   0 & \gamma(z_2,\T_0) \\
  \end{array}\right).
\end{align}

\noindent\textbf{Proof of Theorem \ref{thm:C-normality-WLSE-1}.}
Let us fix $i\in\N$. First, we use a Taylor expansion for the function $\vc\nabla S_{n}(\cdot)$ around $\T_0$:
$${\vc 0}=\vc\nabla S_{n}(\widehat{\T}_n)=\vc\nabla S_{n}(\T_0)+\vc\nabla^2 S_{n}(\T_n)^\top(\widehat{\T}_n-\T_0),$$surely,
where $\T_n$ is a point between $\widehat{\T}_n$ and $\T_0$, and $\vc\nabla^2 S_{n}(\T_n)$ is the Jacobian matrix of $S_{n}(\cdot)$ at $\T_n$. Then,
\begin{equation}\label{eq:pikachu}\sqrt{n}(\widehat{\T}_n-\T_0)=-\left(\vc\nabla^2 S_{n}(\T_n)\right)^{-1}\sqrt{n}\vc\nabla S_{n}(\T_0),\end{equation}
where we used the fact that $\vc\nabla^2 S_{n}(\T_n)$ is a symmetric matrix.
By Assumption \ref{cond:bound-m-arrow}, we can interchange derivatives and infinite sum in the function $S_n(\cdot)$ and invert $\vc\nabla^2 S_{n}(\T_n)$ for $n$ large enough (given \ref{eq:C-normality-wlse-help2} below). Given Equation \eqref{eq:pikachu}, by Slutsky's theorem, it is then sufficient to establish \emph{(a)} and \emph{(b)} below:
\begin{enumerate}[label=\emph{(\alph*)},ref=\emph{(\alph*)}]
\item If $\blue{\Phi_{\zeta(\T_0)}(\cdot)}$ is the distribution function of a $d$-dimensional normal distribution with mean vector ${\vc 0}=(0,\ldots,0)^\top$ and covariance matrix $\vc\zeta(\T_0)$ and $x_1,\ldots,x_d\in\R$, then \label{eq:C-normality-wlse-help1}
$$\lim_{n\to\infty} \mbP_i\left[-\sqrt{n}\frac{\partial S_{n}(\T_0)}{\partial\theta_j} \leq x_j,\ j=1,\ldots,d\ \bigg|Z_n>0\right]=\blue{\Phi}_{\zeta(\T_0)}(x_1,\ldots,x_d).$$

\item For each $\epsilon>0$, \label{eq:C-normality-wlse-help2}
$$\lim_{n\to\infty} \mbP_i\left[\bigg|\frac{\partial^2 S_{n}(\T_n)}{\partial\theta_j\partial\theta_l}-\eta_{jl}(\T_0)\bigg|>\epsilon|Z_n>0\right]=0,\quad j,l=1,\ldots,d.$$
\end{enumerate}

\vspace{2ex}

To prove \ref{eq:C-normality-wlse-help1}, we denote by $\nabla\vc\mu(\T_0)$ the infinite matrix with $d$ rows whose $z$-th column is the column vector $\nabla m^\uparrow(z,\T_0)$, and we let $\nabla\vc\mu_k(\T_0)$ be the matrix of order $d\times k$ whose $j$-th column is $\nabla m^\uparrow(j,\T_0)$. Moreover, for $n,z,k\in\N$, we write
\begin{align*}
\hat{I}_n(z)&=\hat{w}_n(z)\sqrt{n}(m^\uparrow(z,\T_0)-\hat{m}_n(z)),\\
U_{nk}&=\nabla\vc\mu_k(\T_0)(\hat{I}_n(1),\ldots,\hat{I}_n(k))^\top,\\
\rho(k,\T_0)&=\nabla\vc\mu_k(\T_0)\text{diag}\{w_1(\T_0)^2\gamma(1,\T_0),\ldots,w_k(\T_0)^2\gamma(k,\T_0)\}\nabla\vc\mu_k(\T_0)^\top,
\end{align*}
where we recall the definition of $\gamma(z,\T_0)$ in \eqref{gamma} 
 \blue{and let \linebreak $\text{diag}\{w_1(\T_0)^2\gamma(1,\T_0),\ldots,w_k(\T_0)^2\gamma(k,\T_0)\}$ be the diagonal matrix of order $k$ whose $(i,i)$-th element is $w_i(\T_0)^2\gamma(i,\T_0)$. Because $-\sqrt{n}\nabla S_n(\T_0)=2\lim_{k\to\infty} U_{nk}$, we then apply a counterpart of \cite[Theorem 25.5]{Bi-79} for random vectors, which states that, for random vectors $X_n^{(k)}$, $X^{(k)}$, and $Y_n$ satisfying
\begin{enumerate} 
\item[\emph{(i)}] for each $k$, $X_n^{(k)}\stackrel{d}{\to} X^{(k)}$ as $n\to\infty$, 
\item[\emph{(ii)}] $X^{(k)}\stackrel{d}{\to} X$ as $k\to\infty$, and 
\item[\emph{(iii)}] $\lim_{k\to\infty} \limsup_{n\to\infty} \mbP[\|X_n^{(k)}-Y_n\|\geq \varepsilon]=0 $ for all $\varepsilon>0$,
\end{enumerate}
we have $Y_n\stackrel{d}{\to} X$, where in our case $X_n^{(k)}=U_{nk}$ and $Y_n=\blue{\frac{\sqrt{n}}{2}}\cdot\nabla S_n(\T_0)$}. 
We now establish  \emph{(i)--(iii)}.
\begin{enumerate}[label=\emph{(a-\roman*)},ref=\emph{(a-\roman*)}]
\item For each $k\in\N$ fixed we have
\blue{\begin{equation*}
\lim_{n\to\infty}\mbP_i[U_{nk}\leq \vc x|Z_n>0]=\Phi_{\rho(k,\T_0)}(\vc x),\quad \vc x\in\R^d,
\end{equation*}}
where $\Phi_{\rho(k,\T_0)}(\cdot)$ is the distribution function of a \blue{$d$-dimensional} normal distribution with mean equal \blue{to the vector $\vc 0$} and \blue{covariance matrix} $\rho(k,\T_0)$. Indeed, this follows by applying \eqref{eq:Q-normality-m} and \eqref{eq:Q-indep-m} in combination with Assumption (A1) and Slutsky's theorem.

\item Let \blue{$U_k$ be a $d$-dimensional normal distribution with mean equal to the vector $\vc 0$ and covariance matrix $\rho(k,\T_0)$. Then, as $k\to\infty$, $U_k$ converges in distribution to a normal distribution with mean $\vc 0$ and covariance matrix whose $(i,j)$-th element is
$\lim_{k \to \infty} (\rho(k,\T_0))_{i,j}=\zeta_{ij}(\T_0)/4$.}

\item It remains to show that, for each $\epsilon>0$, \blue{and $j=1,\ldots,d$, if $U_{nk}(j)$ is the $j$-th coordinate of the vector $U_{nk}$ then}
\begin{align*}
\lim_{k\to\infty}\bigg(\limsup_{n\to\infty}\mbP_i\bigg[\bigg|\sum_{z=1}^\infty\frac{\partial m^\uparrow(z,\T_0)}{\partial \theta_j}  \hat{w}_n(z) \sqrt{n}(m^\uparrow(z,\T_0)-\hat{m}_n(z))-\blue{U_{nk}(j)}\bigg|>\epsilon|Z_n>0\bigg]\bigg)=0,
\end{align*}
or equivalently
\begin{align*}
\lim_{k\to\infty}\bigg(\limsup_{n\to\infty}\mbP_i\bigg[\sum_{z=k+1}^\infty\frac{\partial m^\uparrow(z,\T_0)}{\partial \theta_j}  \hat{w}_n(z) \sqrt{n}(m^\uparrow(z,\T_0)-\hat{m}_n(z))>\epsilon|Z_n>0\bigg]\bigg)=0.
\end{align*}
In the case where there exists a value $z^*<\infty$ such that for all $n\in\mbN$ and $z\geq z^*$, $\hat{w}_n(z)=0$, \emph{(a-iii)} is immediate. When $\hat{w}_n(z)= \hat{w}^{(1)}_n(z)$ or $\hat{w}_n(z)= \hat{w}^{(2)}_n(z)$, \emph{(a-iii)} is more  technically demanding, and consequently we postpone the proof of \emph{(a-iii)} in that case until Lemma~\ref{aditional_lemma} in Appendix \ref{appA}.
\end{enumerate}

Finally, in order to prove \ref{eq:C-normality-wlse-help2}, let us fix $j,l=1,\ldots,d$, and write
\begin{align*}
T_n^{(1)}&=\sum_{z=1}^\infty \frac{\partial^2 m^\uparrow(z,\T_n)}{\partial \theta_j\partial \theta_l}  \hat{w}_n(z) (m^\uparrow(z,\T_n)-\hat{m}_n(z)),\quad n\in\N\\
T_n^{(2)}&=\sum_{z=1}^\infty \frac{\partial m^\uparrow(z,\T_n)}{\partial \theta_j} \cdot\frac{\partial m^\uparrow(z,\T_n)}{\partial \theta_l}  \hat{w}_n(z)-\frac{\eta_{jl}(\T_0)}{2},\quad n\in\N.
\end{align*}
Because $\frac{\partial^2 S_{n}(\T_n)}{\partial\theta_j\partial\theta_l}-\eta_{jl}(\T_0)=2(T_n^{(1)}+T_n^{(2)})$, by the triangle inequality, it suffices to prove that, for each $\epsilon>0$,
\begin{enumerate}[label=\emph{(b-\arabic*)},ref=\emph{(b-\arabic*)}]
\item $\lim_{n\to\infty} \mbP_i\big[\big|T_n^{(1)}\big|>\epsilon|Z_n>0\big]=0$,\label{eq:C-normality-wlse-help2-a}
\item $\lim_{n\to\infty} \mbP_i\big[\big|T_n^{(2)}\big|>\epsilon|Z_n>0\big]=0$.\label{eq:C-normality-wlse-help2-b}
\end{enumerate}

Let us start with \ref{eq:C-normality-wlse-help2-a}. We write
\begin{align*}
T_{1n}^{(1)}&=\sum_{z=1}^\infty \bigg|\frac{\partial^2 m^\uparrow(z,\T_n)}{\partial \theta_i\partial \theta_j} \bigg|  \cdot |m^\uparrow(z,\T_n)-\hat{m}_n(z)|\cdot |\hat{w}_n(z)-w_z(\T_0)|,\\
T_{2n}^{(1)}&=\sum_{z=1}^\infty \bigg|\frac{\partial^2 m^\uparrow(z,\T_n)}{\partial \theta_i\partial \theta_j} \bigg| \cdot |m^\uparrow(z,\T_n)-\hat{m}_n(z)|w_z(\T_0).
\end{align*}
Since $\big|T_n^{(1)}\big|\leq T_{1n}^{(1)}+T_{2n}^{(1)}$, then
$$\Big\{T_{1n}^{(1)}\leq\frac{\epsilon}{2}\Big\}\cap \Big\{T_{2n}^{(1)}\leq\frac{\epsilon}{2}\Big\}\subseteq\Big\{\big|T_n^{(1)}\big|\leq\epsilon\Big\},$$
and thus, it suffices to prove that $\lim_{n\to\infty}\mbP_i\Big[T_{\ell n}^{(1)}\leq\epsilon/2\big|Z_n>0\Big]=1$, for $\ell=1,2$.

Recall the definition of $C$ from Assumption \ref{cond:m-arrow-hat-unif-bounded} and $\mathcal C\equiv\mathcal C(\T_0)$ from Assumption \ref{cond:bound-m-arrow}.  As in the proof of Theorem \ref{thm:C-consistency-WLSE-1}, let us write $A_n=\sup_{z\in\N}|m^\uparrow(z,\T_0)-\hat{m}_n(z)|$, and take $\delta_1>0$ sufficiently small so that $\{\T\in\Theta:\|\T-\T_0\|\leq \delta_1\}\subseteq \mathcal{C}$. On the one hand,
if we define the event $B_{1n}:=\{\|\T_n-\T_0\|\leq \delta_1\}\cap\{A_n\leq C\}$, then
\begin{align}\label{eq:prob-0-Bn}
\lim_{n\to\infty}\mbP_i[B_{1n}|Z_n>0]=1,
\end{align}
where we have used Assumption \ref{cond:m-arrow-hat-unif-bounded}, and Theorem \ref{thm:C-consistency-WLSE-1} in combination with the fact
 that $\T_n=(1-\lambda)\widehat{\T}_n'+\lambda\T_0$, for some $0<\lambda<1$, thus,
$$\|\T_n-\T_0\|\leq (1-\lambda)\|\widehat{\T}_n'-\T_0\|\leq \|\widehat{\T}_n'-\T_0\|.$$

On the other hand, by Assumption \ref{cond:bound-m-arrow}-(b) and the triangle inequality, on $B_{1n}$ we have
\begin{align*}
T_{1n}^{(1)}&\leq M_2^* \sum_{z=1}^\infty \left(|m^\uparrow(z,\T_n)-m^\uparrow(z,\T_0)|+|m^\uparrow(z,\T_0)-\hat{m}_n(z)|\right)|\hat{w}_n(z)-w_z(\T_0)|\\
&\leq M_2^*(M+C)\sum_{z=1}^\infty |\hat{w}_n(z)-w_z(\T_0)|,
\end{align*}
where in the second inequality we used Assumptions \ref{cond:m-arrow-unif-bounded} and \ref{cond:m-arrow-hat-unif-bounded}.
Therefore, by Assumption~(A1) and Lemma \ref{lem-sum},
{\small\begin{align*}
\lim_{n\to\infty} \mbP_i\left[\Big\{T_{1n}^{(1)}>\frac{\epsilon}{2}\Big\}\cap B_{1n}|Z_n>0\right]\leq \lim_{n\to\infty} \mbP_i\left[ \sum_{z=1}^\infty |\hat{w}_n(z)-w_z(\T_0)|>\frac{\epsilon}{2M_2^*(M+C)}\bigg|Z_n>0\right]=0.
\end{align*}}
Now, using \eqref{eq:prob-0-Bn} we get $\lim_{n\to\infty}\mbP_i\Big[T_{1n}^{(1)}\leq\epsilon/2\big|Z_n>0\Big]=1$.

On the other hand, for $T^{(1)}_{2n}$, we note that since ${\vc w(\T_0)}$ is a probability distribution
we can select $K_1=K_1(\epsilon)\in\N$ such that
$$\sum_{z=K_1+1}^\infty w_z(\T_0)\leq \frac{\epsilon}{6 M_2^*(M+C)}.$$
Moreover, by Assumption \ref{cond:bound-m-arrow}-(a) we can choose $0<\delta_2=\delta_2(\epsilon)<\delta_1$ such that on the event $B_{2n}=\{\|\T_n-\T_0\|\leq \delta_2\}$ we have
$$\max_{z=1,\ldots,K_1}|m^\uparrow(z,\T_n)-m^\uparrow(z,\T_0)|\leq \frac{\epsilon}{6M_2^* K_1},$$
where we observe that by the same arguments as for \eqref{eq:prob-0-Bn},
\begin{equation}\label{B2n}\lim_{n\to\infty}\mbP_i[B_{2n}|Z_n>0]=1.\end{equation}
In addition, on $\{A_n\leq C\}\cap B_{2n}$, we have
\begin{eqnarray*}
T_{2n}^{(1)}&\leq& M_2^* \sum_{z=1}^{K_1}|m^\uparrow(z,\T_0)-\hat{m}_n(z)|+M_2^* \sum_{z=1}^{K_1}|m^\uparrow(z,\T_n)-m^\uparrow(z,\T_0)|\\
&&\phantom{\leq}+M_2^*(M+C)\sum_{z=K_1+1}^\infty w_z(\T_0)\\&\leq& M_2^* \sum_{z=1}^{K_1}|m^\uparrow(z,\T_0)-\hat{m}_n(z)|+ \frac{\epsilon}{6}+ \frac{\epsilon}{6}
\end{eqnarray*}
As a consequence, we obtain
\begin{eqnarray*}\lefteqn{
\bigg\{T_{2n}^{(1)}>\frac{\epsilon}{2}\bigg\}\cap B_{2n}\cap \{A_n\leq C\}}\\
&\subseteq& \left\{ \exists z\in\{1,\ldots,K_1\}: |m^\uparrow(z,\T_0)-\hat{m}_n(z)|> \frac{\epsilon}{6 M_2^* K_1}\right\}\cap B_{2n}\cap \{A_n\leq C\}.
\end{eqnarray*}
Then, by Assumption \ref{cond:m-arrow-hat-unif-bounded} and Theorem \ref{thm:C-consistency-WLSE-1}, together with  \eqref{eq:Q-consistency-m}, we obtain
$$\lim_{n\to\infty} \mbP_i\left[\Big\{T_{2n}^{(1)}>\frac{\epsilon}{2}\Big\}\cap B_{2n}\cap \{A_n\leq C\}|Z_n>0\right]=0.$$
This yields $\lim_{n\to\infty}\mbP_i\Big[T_{2n}^{(1)}\leq\epsilon/2\big|Z_n>0\Big]=1$.

\vspace{2ex}

Now, to prove \ref{eq:C-normality-wlse-help2-b} we observe that $\big|T_n^{(2)}\big|\leq T_{1n}^{(2)}+T_{2n}^{(2)}$, with
\begin{align*}
T_{1n}^{(2)}&=\sum_{z=1}^\infty \bigg|\frac{\partial m^\uparrow(z,\T_n)}{\partial \theta_j}\cdot\frac{\partial m^\uparrow(z,\T_n)}{\partial \theta_l}\bigg| |\hat{w}_n(z)-w_z(\T_0)|,\\
T_{2n}^{(2)}&=\sum_{z=1}^\infty \bigg|\frac{\partial m^\uparrow(z,\T_n)}{\partial \theta_j} \cdot\frac{\partial m^\uparrow(z,\T_n)}{\partial \theta_l}-\frac{\partial m^\uparrow(z,\T_0)}{\partial \theta_j}\cdot \frac{\partial m^\uparrow(z,\T_0)}{\partial \theta_l}\bigg| w_z(\T_0),
\end{align*}
hence, since $\Big\{T_{1n}^{(2)}\leq\frac{\epsilon}{2}\Big\}\cap \Big\{T_{2n}^{(2)}\leq\frac{\epsilon}{2}\Big\}\subseteq\Big\{\big|T_n^{(2)}\big|\leq\epsilon\Big\}$, it is enough to prove that
\begin{equation}\label{eq:gl22CL}
\lim_{n\to\infty}\mbP_i\left[T_{\ell n}^{(2)}\leq\frac{\epsilon}{2}\big|Z_n>0\right]=1,\quad \text{ for }\ell=1,2.
\end{equation}

For $T_{1n}^{(2)}$, by Assumption \ref{cond:bound-m-arrow}-\ref{cond:m-arrow-bound-deriv}, and the fact that on $B_{2n}$ we have $\T_n\in\mathcal{C}$,
\begin{align*}
B_{2n}\cap\left\{\sum_{z=1}^\infty |\hat{w}_n(z)-w_z(\T_0)|\leq \frac{\epsilon}{2M_1^{*2}}\right\}\subseteq \Big\{T_{1n}^{(2)}\leq \frac{\epsilon}{2}\Big\},
\end{align*}
which, combined with \eqref{B2n} and Lemma \ref{lem-sum}, leads to
\eqref{eq:gl22CL} for $\ell=1$.

For $T_{2n}^{(2)}$, we take $K_2=K_2(\epsilon)\in\N$ such that
$$\sum_{z=K_2+1}^\infty w_z(\T_0)\leq \frac{\epsilon}{8M_1^{*2}}.$$
Since, for each $j=1,\ldots,d$ and $z\in\N$ fixed, the function $\T\in\Theta\mapsto \frac{\partial m^\uparrow(z,\T)}{\partial \theta_j}$ is continuous at $\T_0$ by Assumption \ref{cond:bound-m-arrow}-\ref{cond:m-arrow-bound-deriv}, we can  choose $\delta_3=\delta_3(\epsilon)>0$ such that on $B_{3n}=\{\|\T_n-\T_0\|\leq\delta_3\}$,
\begin{align*}
\max_{z=1,\ldots,K_2}\bigg|\frac{\partial m^\uparrow(z,\T_n)}{\partial \theta_j}\cdot \frac{\partial m^\uparrow(z,\T_n)}{\partial \theta_l}-\frac{\partial m^\uparrow(z,\T_0)}{\partial \theta_j}\cdot \frac{\partial m^\uparrow(z,\T_0)}{\partial \theta_l}\bigg|\leq\frac{\epsilon}{4 K_2}.
\end{align*}
On $B_{3n}$, we then have
\begin{eqnarray*}T_{2n}^{(2)}&=& \sum_{z=1}^{K_2} \bigg|\frac{\partial m^\uparrow(z,\T_n)}{\partial \theta_j} \cdot\frac{\partial m^\uparrow(z,\T_n)}{\partial \theta_l}-\frac{\partial m^\uparrow(z,\T_0)}{\partial \theta_j}\cdot \frac{\partial m^\uparrow(z,\T_0)}{\partial \theta_l}\bigg| w_z(\T_0)\\&&+\sum_{z=K_2+1}^\infty \bigg|\frac{\partial m^\uparrow(z,\T_n)}{\partial \theta_j} \cdot\frac{\partial m^\uparrow(z,\T_n)}{\partial \theta_l}-\frac{\partial m^\uparrow(z,\T_0)}{\partial \theta_j}\cdot \frac{\partial m^\uparrow(z,\T_0)}{\partial \theta_l}\bigg| w_z(\T_0)\\&\leq& \frac{\epsilon}{4}+\frac{\epsilon}{4},
\end{eqnarray*}
where we have bounded $w_z(\T_0)$ by 1 in the first sum and we have used Assumption \ref{cond:bound-m-arrow}-\ref{cond:m-arrow-bound-deriv} in the second sum.
Because
$\lim_{n\to\infty}\mbP_i[B_{3n}|Z_n>0]=1$ by the same argument as for \eqref{eq:prob-0-Bn},
we obtain \eqref{eq:gl22CL} for $\ell=2$.
\qed

\medskip

\blue{
\noindent\textbf{Proof of Proposition \ref{thm:Q-consistency-WLSE-1}.}

Let
$$\tilde{S}(\T)=\sum_{z=1}^\infty w_z(\T_0) (m(z,\T)-m^\uparrow(z,\T_0))^2,$$
so that $\tilde{\T}=\arg\min_{\T\in\theta} \tilde{S}(\T)$, and let
\begin{align*}
\tilde{S}_n^*(\T)=\sum_{z=1}^\infty w_z(\T_0)\ind{j_n(z)>0} (m(z,\T)-\hat{m}_n(z))^2,
\end{align*}
where $j_n(z):=\sum_{i=0}^{n-1} \ind{Z_i=z}$. Let us fix the initial state $i\in\N$. Similar to Theorem~\ref{thm:C-consistency-WLSE-1}, by Lemma \ref{lem:help-consistency}, it suffices to prove that, for each $\epsilon>0$,
\begin{enumerate}[label=\emph{(\alph*)},ref=\emph{(\alph*)}]
\item $\lim_{n\to\infty} \mbP_i\left[\sup_{\T\in\Theta}\left|\tilde{S}_n(\T)-\tilde{S}_n^*(\T)\right|>\epsilon|Z_{n}>0\right]=0$,\label{eq:lim-sup-lim-classic-weighted1-help1}
\item $\lim_{n\to\infty} \mbP_i\left[\sup_{\T\in\Theta}\left|\tilde{S}_n^*(\T)-\tilde{S}(\T)\right|>\epsilon|Z_{n}>0\right]=0$.\label{eq:lim-sup-lim-classic-weighted1-help2}
\end{enumerate}
We would then get
\begin{align*}
\forall\epsilon>0,\quad\lim_{n\to\infty} \mbP_i\left[\sup_{\T\in\Theta}\left|\tilde{S}_{n}(\T)-\tilde{S}(\T)\right|>\epsilon|Z_{n}>0\right]=0.
\end{align*}

We start with the proof of \ref{eq:lim-sup-lim-classic-weighted1-help1}. Recall the definition of $M$, $\widetilde{M},$ and $C$ from Assumptions \ref{cond:m-arrow-unif-bounded}, \ref{cond:m-unif-bounded}, and \ref{cond:m-arrow-hat-unif-bounded}, respectively. Let $M'=2(M+\widetilde{M}+C)^2$ and $A_n=\sup_{z\in\N} |m^\uparrow(z,\T_0)-\hat{m}_n(z)|$. On the set $\{A_n\leq C\}$ we have
\begin{align*}
\sup_{\T\in\Theta} |\tilde{S}_n(\T)&-\tilde{S}_n^*(\T)|\leq\sup_{\T\in\Theta} \sum_{z=1}^\infty |\hat{w}_n(z)-w_z(\T_0)\ind{j_n(z)>0}| (m(z,\T)-\hat{m}_n(z))^2
\\
&\leq \sup_{\T\in\Theta} \sum_{z=1}^\infty |\hat{w}_n(z)-w_z(\T_0)\ind{j_n(z)>0}| (m(z,\T)-m^\uparrow(z,\T_0))^2\\
&\phantom{=}+ \sum_{z=1}^\infty |\hat{w}_n(z)-w_z(\T_0)\ind{j_n(z)>0}| (m^\uparrow(z,\T_0)-\hat{m}_n(z))^2\\
&\phantom{=}+2\sup_{\T\in\Theta} \sum_{z=1}^\infty |\hat{w}_n(z)-w_z(\T_0)\ind{j_n(z)>0}|\cdot |m(z,\T)-m^\uparrow(z,\T_0)|\cdot |m^\uparrow(z,\T_0)-\hat{m}_n(z)|\\
&\leq M' \sum_{z=1}^\infty |\hat{w}_n(z)-w_z(\T_0)\ind{j_n(z)>0}|\\
&\leq M'  \sum_{z=1}^\infty |\hat{w}_n(z)-w_z(\T_0)|+M'  \sum_{z=1}^\infty w_z(\T_0)\ind{j_n(z)=0}.
\end{align*}
Consequently,
\begin{align*}
\left\{\sup_{\T\in\Theta} |\tilde{S}_n(\T)-\tilde{S}_n^*(\T)|>\epsilon\right\}&=\left\{\sup_{\T\in\Theta} |\tilde{S}_n(\T)-\tilde{S}_n^*(\T)|>\epsilon,A_n\leq C\right\}\cup\\
&\phantom{=}\cup\left\{\sup_{\T\in\Theta} |\tilde{S}_n(\T)-\tilde{S}_n^*(\T)|>\epsilon,A_n> C\right\}\\
&\subseteq\left\{\sum_{z=1}^\infty |\hat{w}_n(z)-w_z(\T_0)|>\frac{\epsilon}{2M'}\right\}\cup\left\{\sum_{z=1}^\infty w_z(\T_0)\ind{j_n(z)=0}>\frac{\epsilon}{2M'}\right\}\\
&\phantom{=}\cup\left\{A_n> C\right\}.
\end{align*}
Conditioning on $Z_n>0$, we then obtain \ref{eq:lim-sup-C-consist-weighted1-help1} by applying Lemma \ref{lem-sum} to the first event, the fact that $\sum_{z=1}^\infty w_z(\T_0)=1$, $\lim_{n\to\infty} \mbP[\ind{j_n(z)=0}\leq\epsilon|Z_n>0]=1$ for any $\epsilon>0,z\in\N$, and dominated convergence  to the second event, and Assumption \ref{cond:m-arrow-hat-unif-bounded} to the third event.


In order to prove \ref{eq:lim-sup-lim-classic-weighted1-help2}, we observe that
{\small\begin{align*}
\tilde{S}_{n}^*(\T)-\tilde{S}(\T)
&= \sum_{z=1}^\infty w_z(\T_0) \ind{j_n(z)>0}\left[\hat{m}_n(z)^2 -2\hat{m}_n(z)m(z,\T)-m^\uparrow(z,\T_0)^2+2m(z,\T)m^\uparrow(z,\T_0)\right]\\
&\phantom{=}-\sum_{z=1}^\infty w_z(\T_0) \ind{j_n(z)=0}(m(z,\T)-m^\uparrow(z,\T_0))^2,
\end{align*}}
then
\begin{align*}
\sup_{\T\in\Theta}|\tilde{S}_{n}^*(\T)-\tilde{S}(\T)|&\leq  \sum_{z=1}^\infty w_z(\T_0) \ind{j_n(z)>0}\left|\hat{m}_n(z)^2 -m^\uparrow(z,\T_0)^2\right|\\
&\phantom{=}+2\sup_{\T\in\Theta}\sum_{z=1}^\infty w_z(\T_0) \ind{j_n(z)>0}m(z,\T)\left|\hat{m}_n(z)-m^\uparrow(z,\T_0)\right|\\
&\phantom{=}+\sup_{\T\in\Theta}\sum_{z=1}^\infty w_z(\T_0) \ind{j_n(z)=0}(m(z,\T)-m^\uparrow(z,\T_0))^2,
\end{align*}
and consequently,
\begin{align*}
\left\{\sup_{\T\in\Theta}|\tilde{S}_{n}^*(\T)-\tilde{S}(\T)|>\epsilon\right\}&\subseteq\left\{\sum_{z=1}^\infty w_z(\T_0) \ind{j_n(z)>0}\left|\hat{m}_n(z)^2 -m^\uparrow(z,\T_0)^2\right|>\frac{\epsilon}{3}\right\}\cup\\
&\phantom{\subseteq}\cup\left\{\sup_{\T\in\Theta}\sum_{z=1}^\infty w_z(\T_0) \ind{j_n(z)>0}m(z,\T)\left|\hat{m}_n(z)-m^\uparrow(z,\T_0)\right|>\frac{\epsilon}{6}\right\}\\
&\phantom{\subseteq}\cup\left\{\sup_{\T\in\Theta}\sum_{z=1}^\infty w_z(\T_0) \ind{j_n(z)=0}(m(z,\T)-m^\uparrow(z,\T_0))^2>\frac{\epsilon}{3}\right\}.
\end{align*}

Regarding the first event, similar to the proof of Theorem~\ref{thm:C-consistency-WLSE-1}, we can focus on the event $\{A_n \leq C\}$. In that proof, we showed that
\begin{align*}
\lim_{n\to\infty}\mbP_i\bigg[\sum_{z=1}^\infty w_z(\T_0) \ind{j_n(z)>0}\left|\hat{m}_n(z)^2 -m^\uparrow(z,\T_0)^2\right|\leq \frac{\epsilon}{3}|Z_n>0\bigg]=1.
\end{align*}
For the second event,
we observe that
\begin{align*}
\frac{\epsilon}{6}&<\sup_{\T\in\Theta}\sum_{z=1}^\infty w_z(\T_0) \ind{j_n(z)>0}m(z,\T)\left|\hat{m}_n(z)-m^\uparrow(z,\T_0)\right|\\
&\leq \widetilde{M}\sum_{z=1}^\infty w_z(\T_0) \ind{j_n(z)>0}\left|\hat{m}_n(z)-m^\uparrow(z,\T_0)\right|,
\end{align*}
and we proceed using similar arguments as for the first event to show that
\begin{align*}
\lim_{n\to\infty}\mbP_i\bigg[\sum_{z=1}^\infty w_z(\T_0) \ind{j_n(z)>0}\left|\hat{m}_n(z) -m^\uparrow(z,\T_0)\right|\leq \frac{\epsilon}{6}|Z_n>0\bigg]=1.
\end{align*}

Finally, for the third event, we use the fact that $\sup_{\T\in\Theta}\sup_{z\in\N}\left(m(z,\T)-m^\uparrow(z,\T_0)\right)^2\leq (M+\widetilde{M})^2$ and the same argument as for the second event in \emph{(a)} to show that
\begin{align*}
\lim_{n\to\infty}\mbP_i\bigg[\sup_{\T\in\Theta}\sum_{z=1}^\infty w_z(\T_0) \ind{j_n(z)=0}(m(z,\T)-m^\uparrow(z,\T_0))^2\leq \frac{\epsilon}{3}|Z_n>0\bigg]=1.
\end{align*}
Consequently we have proved \ref{eq:lim-sup-lim-classic-weighted1-help2}, and by Lemma~\ref{lem:help-consistency} the result follows. \qed

}

\section*{Acknowledgements}
Peter Braunsteins and Sophie Hautphenne would like to thank the Australian Research Council (ARC) for support through the Discovery Project DP200101281.

This manuscript was started while Carmen Minuesa was a visiting postdoctoral researcher in the School of Mathematics and Statistics at The University of Melbourne, and she is grateful for the hospitality and collaboration. She also acknowledges the Australian Research Council Centre of Excellence for Mathematical and Statistical Frontiers for partially supporting her research visit at The University of Melbourne. Carmen Minuesa's research is part of the R\&D\&I project PID2019-108211GB-I00, funded by MCIN/AEI/10.13039/501100011033/.

\appendix

\section{Complement to the proof of Theorem \ref{thm:C-normality-WLSE-1}}\label{appA}

The following lemma establishes \emph{(a-iii)} in the proof of Theorem \ref{thm:C-normality-WLSE-1}.

\begin{lem}\label{aditional_lemma} Under the conditions of Theorem \ref{thm:C-normality-WLSE-1}, we have for $\hat{w}_n(z)= \hat{w}^{(1)}_n(z)$ and $\hat{w}_n(z)= \hat{w}^{(2)}_n(z)$,
\begin{align*}
\lim_{k\to\infty}\bigg(\limsup_{n\to\infty}\mbP_i\bigg[\left|\sum_{z=k+1}^\infty\frac{\partial m^\uparrow(z,\T_0)}{\partial \theta_i}  \hat{w}_n(z) \sqrt{n}(m^\uparrow(z,\T_0)-\hat{m}_n(z))\right|>\epsilon|Z_n>0\bigg]\bigg)=0.
\end{align*}
\end{lem}

\noindent\textbf{Proof of Lemma \ref{aditional_lemma}.}
We establish the result when $\hat w_n(z)=\hat w^{(2)}_n(z)$; the reasoning when $\hat w_n(z)=\hat w^{(1)}_n(z)$ follows from similar (albeit simpler) arguments. To lighten the notation we drop the dependence on $\T_0$ in the rest of the proof.

Fix $j=1,\ldots,d$ and denote
$$
V_{nz}=f(z)  \hat w_n(z) \sqrt{n}(m^\uparrow(z)-\hat{m}_n(z)),\quad n,z\in\N,
$$
where $f(z):=\frac{\partial m^\uparrow(z)}{\partial \theta_j}$.
We prove that for each $\epsilon>0$,
\begin{align}\label{eq:a3-prefinal}
\lim_{k\to\infty}\bigg(\limsup_{n\to\infty}\mbP_i\bigg[\bigg|\sum_{z=k+1}^\infty V_{nz}\bigg|>\epsilon|Z_n>0\bigg]\bigg)=0.
\end{align}

We divide the proof of \eqref{eq:a3-prefinal} in two steps. Recall the definitions of $\{ Z_\ell^{(n)} \}$ and $\{Z^{\uparrow}_{\ell}\}$ from Section \ref{sec:Q-process}.
For $n,z\in\N$, let
\begin{align*}
V_{nz}^\uparrow &=f(z)q_n^\uparrow(z) \sqrt{n}(m^\uparrow(z)-\hat{m}_n^\uparrow(z)),\\
q_n^\uparrow(z) &=\frac{z\sum_{i=0}^{n-1}\ind{Z_i^\uparrow=z}}{\sum_{i=0}^{n-1}Z_i^\uparrow},\quad
\hat{m}_n^\uparrow(z)=\frac{\sum_{i=0}^{n-1}Z_{i+1}^\uparrow\ind{Z_i^\uparrow=z}}{z\sum_{i=0}^{n-1}\ind{Z_i^\uparrow=z}},
\end{align*}
and
\begin{align*}
V_{nz}^{(n)} &=f(z)q_n^{(n)}(z) \sqrt{n}(m^\uparrow(z)-\hat{m}_n^{(n)}(z)),\\
q_n^{(n)}(z) &=\frac{z\sum_{i=0}^{n-1}\ind{Z_i^{(n)}=z}}{\sum_{i=0}^{n-1}Z_i^{(n)}},\quad \hat{m}_n^{(n)}(z)=\frac{\sum_{i=0}^{n-1}Z_{i+1}^{(n)}\ind{Z_i^{(n)}=z}}{z\sum_{i=0}^{n-1}\ind{Z_i^{(n)}=z}}.
\end{align*}
\normalsize

\noindent\textbf{Step 1.} We show that the counterpart of \eqref{eq:a3-prefinal} for the $Q$-process holds, that is,
\begin{align}\label{eq:a3-final-Q-process}
\forall\epsilon>0,\quad\lim_{k\to\infty}\bigg(\limsup_{n\to\infty}\mbP_i\bigg[\bigg|\sum_{z=k+1}^\infty V_{nz}^\uparrow\bigg|>\epsilon\bigg]\bigg)=0.
\end{align}

\noindent\textbf{Step 2.} We prove that there exists a sequence of probability spaces (indexed by $n$) on which both $\{Z^{(n)}_\ell\}$ and $\{Z^{\uparrow}_\ell\}$ are defined such that
\begin{align}\label{eq:a3-final-Q-original-process}
\forall\epsilon>0,\quad\lim_{k\to\infty}\bigg(\limsup_{n\to\infty}\mbP_i\bigg[\bigg|\sum_{z=k+1}^\infty (V_{nz}^{(n)}-V_{nz}^\uparrow)\bigg|>\epsilon\bigg]\bigg)=0.
\end{align}

We start with the proof of \eqref{eq:a3-final-Q-process}. First, observe that
\begin{align*}
\sum_{z=k+1}^\infty V_{nz}^\uparrow &= \sum_{z=k+1}^\infty  f(z) q_n^\uparrow(z) \sqrt{n}(m^\uparrow(z)-\hat{m}_n^\uparrow(z))\\
&=\sum_{z=k+1}^\infty f(z) \frac{z\sum_{i=0}^{n-1}\ind{Z_i^\uparrow=z}}{\sum_{i=0}^{n-1}Z_i^\uparrow}\cdot\sqrt{n} \frac{\sum_{i=0}^{n-1}(z m^\uparrow(z)- Z_{i+1}^\uparrow)\ind{Z_i^\uparrow=z}}{z\sum_{i=0}^{n-1}\ind{Z_i^\uparrow=z}},\\
&=\frac{1}{\sum_{i=0}^{n-1}Z_i^\uparrow}\sum_{z=k+1}^\infty f(z) \sqrt{n} \sum_{i=0}^{n-1}(z m^\uparrow(z)- Z_{i+1}^\uparrow)\ind{Z_i^\uparrow=z},
\end{align*}
then, using Markov's inequality and the fact that $\sum_{l=0}^{n-1}Z_l^\uparrow\geq n$ a.s., we have
\begin{align}
\nonumber \mbP_i&\bigg[\bigg|\sum_{z=k+1}^\infty V_{nz}^\uparrow\bigg|>\epsilon\bigg]=\\\nonumber
&=\mbP_i\bigg[\frac{1}{\sum_{l=0}^{n-1}Z_l^\uparrow}\bigg|\sum_{z=k+1}^\infty f(z)\sqrt{n} \sum_{l=0}^{n-1}(z m^\uparrow(z)- Z_{l+1}^\uparrow)\ind{Z_l^\uparrow=z}\bigg|>\epsilon\bigg]\\\nonumber
&\leq \frac{1}{\epsilon^2}\cdot \mbE_i\left[\left(\frac{\sum_{z=k+1}^\infty f(z) \sqrt{n} \sum_{l=0}^{n-1}(z m^\uparrow(z)- Z_{l+1}^\uparrow)\ind{Z_l^\uparrow=z}}{\sum_{l=0}^{n-1}Z_l^\uparrow}\right)^2\right]\\\nonumber
&= \frac{1}{\epsilon^2}\cdot \mbE_i\left[\left(\frac{\sum_{z=k+1}^\infty \frac{f(z)}{\sqrt{n}} \sum_{l=0}^{n-1}(z m^\uparrow(z)- Z_{l+1}^\uparrow)\ind{Z_l^\uparrow=z}}{\frac{1}{n}\sum_{l=0}^{n-1}Z_l^\uparrow}\right)^2\right]\\\nonumber
&\leq \frac{1}{\epsilon^2}\cdot \mbE_i\left[\left(\sum_{z=k+1}^\infty \frac{f(z)}{\sqrt{n}} \sum_{l=0}^{n-1}(z m^\uparrow(z)- Z_{l+1}^\uparrow)\ind{Z_l^\uparrow=z}\right)^2\right]\\\label{last}
&= \frac{1}{n\epsilon^2}\cdot \mbE_i\left[\left( \sum_{l=0}^{n-1}\sum_{z=k+1}^\infty f(z)(z m^\uparrow(z)- Z_{l+1}^\uparrow)\ind{Z_l^\uparrow=z}\right)^2\right].
\end{align}
We now show that the expectation in \eqref{last} is finite, that is, letting $a_l:=\sum_{z=k+1}^\infty f(z)(z m^\uparrow(z)- Z_{l+1}^\uparrow)\ind{Z_l^\uparrow=z}$ we show that $\mbE_i\left[\left( \sum_{l=0}^{n-1} a_l\right)^2\right]<\infty$.
We have
\begin{equation}\label{ea}\mbE_i\left[\left( \sum_{l=0}^{n-1} a_l\right)^2\right]= \sum_{l=0}^{n-1} \mbE_i[a_l^2]+2\sum_{j=0}^{n-1}\sum_{l=0}^{j-1} \mbE_i[a_l\, a_j].\end{equation}
We start by showing that the second sum in \eqref{ea} vanishes: for any fixed $0\leq j\leq n-1$ and $0\leq l\leq j-1$,
\begin{eqnarray*} \lefteqn{\mbE_i[a_l\, a_j]}\\
&=&\sum_{z_1=k+1}^\infty \sum_{z_2=k+1}^\infty \mbE[a_l\, a_j\,|\,Z_l^\uparrow=z_1, Z_j^\uparrow=z_2] \mbP_i[Z_l^\uparrow=z_1,Z_j^\uparrow=z_2]
\\
&=& \sum_{z_1=k+1}^\infty \sum_{z_2=k+1}^\infty f(z_1)f(z_2) \mbE[(z_1 m^\uparrow(z_1)- Z_{l+1}^\uparrow)\, (z_2 m^\uparrow(z_2)- Z_{j+1}^\uparrow)\,|\,Z_l^\uparrow=z_1, Z_j^\uparrow=z_2]\\&&\phantom{\sum_{z_1=k+1}^\infty \sum_{z_2=k+1}^\infty}\cdot \mbP_i[Z_l^\uparrow=z_1,Z_j^\uparrow=z_2]\\
&=&0,
\end{eqnarray*}
where the last equality follows from the fact that $\mbE[(z m^\uparrow(z)- Z_{l+1}^\uparrow)\,|\,Z_l^\uparrow=z]=0$ for any $z>0$, and the conditional independence of $Z_{l+1}^\uparrow$ and $Z_{j+1}^\uparrow$ given $Z_l^\uparrow=z_1$ and $Z_j^\uparrow=z_2$.

Now we treat the first sum in \eqref{ea}:
\begin{eqnarray*}\mbE_i[a_l^2]&=&\sum_{x=k+1}^\infty \mbE[a_l^2\,|\,Z_l^\uparrow=x] \mbP_i[Z_l^\uparrow=x]\\&=&\sum_{x=k+1}^\infty f(x)^2 \mbE[(x m^\uparrow(x)- Z_{l+1}^\uparrow)^2\,|\,Z_l^\uparrow=x]\mbP_i[Z_l^\uparrow=x] \\&\leq & M_1^{*2}\sum_{x=k+1}^\infty \V[Z_{l+1}^\uparrow\,|\,Z_l^\uparrow=x] \mbP_i[Z_l^\uparrow=x], \end{eqnarray*}
where in the last inequality we used Assumption \ref{cond:bound-m-arrow}\ref{cond:m-arrow-bound-deriv}  (and the fact that $\mbE[(x m^\uparrow(x)- Z_{l+1}^\uparrow)\,|\,Z_l^\uparrow=x]=0$).
Note that according to \cite[Lemma 13]{braunsteins2022parameter}, there exists an operator $\textbf{S}$ such that $
\textbf{Q}^n = \rho^n \vc{v} \vc{u}^\top + \textbf{S}^n.$
We then have from \eqref{last} and \eqref{ea}
\begin{eqnarray}\nonumber \mbP_i\bigg[\bigg|\sum_{z=k+1}^\infty V_{nz}^\uparrow\bigg|>\epsilon\bigg] &\leq& \frac{1}{n \epsilon^2}
\sum_{l=0}^{n-1} \mbE_i[a_l^2]\\\nonumber &\leq& \frac{M_1^{*2}}{n \epsilon^2}\sum_{z=k+1}^\infty \sum_{l=0}^{n-1} \V[Z_{l+1}^\uparrow\,|\,Z_l^\uparrow=z] \mbP_i[Z_l^\uparrow=z]\\\nonumber &=&
\frac{M_1^{*2}}{n \epsilon^2}\sum_{z=k+1}^\infty \sum_{l=0}^{n-1}  z^2 \sigma^{\uparrow 2}(z)(\textbf{Q}^{\uparrow l})_{iz}\\\nonumber
&=&\frac{M_1^{*2}}{n\epsilon^2 v_i}\sum_{z=k+1}^\infty z^2 \sigma^{\uparrow 2}(z) v_z \sum_{l=0}^{n-1} \frac{Q_{iz}}{\rho^l}\\\nonumber
&=&\frac{M_1^{*2}}{n\epsilon^2 v_i}\sum_{z=k+1}^\infty z^2 \sigma^{\uparrow 2}(z) v_z \sum_{l=0}^{n-1} \frac{\boldsymbol{e}_i^\top  \textbf{Q}^l \boldsymbol{e}_z}{\rho^l}\\\nonumber
&=&\frac{M_1^{*2}}{n\epsilon^2 v_i}\sum_{z=k+1}^\infty z^2 \sigma^{\uparrow 2}(z) v_z \sum_{l=0}^{n-1} \frac{\boldsymbol{e}_i^\top  (\rho^l \vc v \vc u + \textbf{S}^l) \boldsymbol{e}_z}{\rho^l}\\\nonumber
&=&\frac{M_1^{*2}}{n\epsilon^2 v_i}\sum_{z=k+1}^\infty z^2 \sigma^{\uparrow 2}(z) v_z \sum_{l=0}^{n-1} \left(v_i  u_z+\frac{\boldsymbol{e}_i^\top \textbf{S}^l \boldsymbol{e}_z}{\rho^l}\right)\\\nonumber
&=&\frac{M_1^{*2}}{n\epsilon^2 v_i}\sum_{z=k+1}^\infty z^2 \sigma^{\uparrow 2}(z) v_z n v_i u_z\\\nonumber
&\phantom{=}&+\frac{M_1^{*2}}{n\epsilon^2 v_i}\sum_{z=k+1}^\infty z^2 \sigma^{\uparrow 2}(z) v_z \sum_{l=0}^{n-1} \frac{\boldsymbol{e}_i^\top \textbf{S}^l \boldsymbol{e}_z}{\rho^l}\\
\label{ft}
&=&\frac{M_1^{*2}}{\epsilon^2}\sum_{z=k+1}^\infty z^2 \sigma^{\uparrow 2}(z) v_z u_z\\\label{st}
&\phantom{=}&+\frac{M_1^{*2}}{n\epsilon^2 v_i}\sum_{z=k+1}^\infty z^2 \sigma^{\uparrow 2}(z) v_z \sum_{l=0}^{n-1} \frac{\boldsymbol{e}_i^\top \textbf{S}^l \boldsymbol{e}_z}{\rho^l}.
\end{eqnarray}

Now, we first show that $\sum_{z=1}^\infty z^2 \sigma^{\uparrow 2}(z) v_z u_z<\infty$ so that the term \eqref{ft} vanishes as $k\to\infty$. In what follows, we use the function $t(\cdot)$ and the norms $\|\cdot\|_{\infty,t}$ and $\|\cdot\|_{1,t}$ defined in \cite[Section 5.3.1]{braunsteins2022parameter}. We have
\begin{eqnarray*}
\sum_{z=1}^\infty z^2 u_z v_z \sigma^{\uparrow 2}(z)&\leq&
 \sum_{z=1}^\infty z^2 u_z v_z \frac{\sum_{k=1}^\infty k^2 Q_{zk}^\uparrow}{z^2}\\
&=&  \sum_{z=1}^\infty u_z v_z \frac{\sum_{k=1}^\infty k^2 v_k Q_{zk}}{\rho v_z}\\
&= &\frac{1}{\rho} \sum_{z=1}^\infty u_z \sum_{k=1}^\infty \frac{ k^{\iota+3} v_k Q_{zk}}{k^{\iota+1}}\\
&\leq& \frac{\|\vc v\|_{\infty,t}}{\rho} \sum_{z=1}^\infty u_z \sum_{k=1}^\infty k^{\iota+3} Q_{zk}\\
&=& \frac{\|\vc v\|_{\infty,t}}{\rho} \sum_{z=1}^\infty u_z \mbE[Z_1^{\iota+3}|Z_0=z]\\
&\leq& \frac{\|\vc v\|_{\infty,t}}{\rho} \sum_{z=1}^\infty z^{\iota+3}u_z \mbE[\xi_{01}(z)^{\iota+3}]\\
&\leq& \frac{ C \|\vc v\|_{\infty,t}\|\vc u\|_{1,t^*}}{\rho}<\infty ,
\end{eqnarray*}
for some contant $C>0$, where we applied \cite[Lemma 13]{braunsteins2022parameter} with the functions $t(x)=x^{\iota+1}$ and $t^*(x)=x^{\iota+3}$, Minkowski's inequality, and Assumption \ref{cond:bounded-moments-xi}.

Finally, to obtain \eqref{eq:a3-final-Q-process} it remains to treat \eqref{st}, i.e., to
prove that
\begin{align}\label{eq:a3-final-Q-process-2term}
\lim_{k\to\infty}\left(\limsup_{n\to\infty}\left(\frac{M_1^{*2}}{n\epsilon^2 v_i}\sum_{z=k+1}^\infty z^2 \sigma^{\uparrow 2}(z) v_z \sum_{l=0}^{n-1} \frac{\boldsymbol{e}_i^\top \textbf{S}^l \boldsymbol{e}_z}{\rho^l}\right)\right)=0.
\end{align}

 To show this, we make use of \cite[Inequality (51)]{braunsteins2022parameter}, that is, the fact that there exist $K<\infty$ and $\delta>0$ such that $|\boldsymbol{e}_i^\top \textbf{S}^l \boldsymbol{e}_z|\leq \frac{t(i)}{t(z)} K (1-\delta)^l \rho^l$. We have
 \begin{align}\nonumber
 \frac{1}{n\epsilon^2 v_i}&\sum_{z=k+1}^\infty z^2 \sigma^{\uparrow 2}(z) v_z \sum_{l=0}^{n-1} \frac{\boldsymbol{e}_i^\top \textbf{S}^l \boldsymbol{e}_z}{\rho^l}\\ \nonumber
 &\leq \frac{1}{n\epsilon^2 v_i}\sum_{z=k+1}^\infty z^2 \sigma^{\uparrow 2}(z) v_z \sum_{l=0}^{n-1} \frac{|\boldsymbol{e}_i^\top \textbf{S}^l \boldsymbol{e}_z|}{\rho^l}\\\nonumber
 &\leq \frac{t(i)K}{n\epsilon^2 v_i}\sum_{z=k+1}^\infty z^2 \sigma^{\uparrow 2}(z) v_z \sum_{l=0}^{n-1} \frac{(1-\delta)^l\rho^l}{\rho^l t(z)}\\\nonumber
 &\leq \frac{t(i)K}{n\epsilon^2 v_i}\sum_{z=k+1}^\infty z^2 \sigma^{\uparrow 2}(z) \frac{v_z}{t(z)} \sum_{l=0}^{n-1} (1-\delta)^l\\\nonumber
 & = \frac{t(i)K}{n\epsilon^2 v_i} \frac{1-(1-\delta)^n}{\delta}\sum_{z=k+1}^\infty z^2 \sigma^{\uparrow 2}(z) \frac{v_z}{t(z)}\\\label{last_step1}& \leq \frac{t(i)K}{n\epsilon^2 v_i} \frac{1-(1-\delta)^n}{\delta}\sum_{z=k+1}^\infty \sum_{l=1}^\infty l^2 Q^{\uparrow}_{zl} \frac{v_z}{t(z)},
 \end{align}
 where in the last step we used the definition of $\sigma^{\uparrow 2}(z)$ and the fact that $\V [X]\leq \mbE[X^2]$. Next we use \cite[Lemma 13]{braunsteins2022parameter} which implies that $v_l<C\, l^{\nu^*}$ for all $l\geq 1$ for some positive constants $C$ and $\nu^*$ and which justifies the choice of $t(z)=z^{\nu^*+m'}$ for some positive constant $m'$:
  \begin{align}\nonumber
 \eqref{last_step1} &= \frac{t(i)K\,[1-(1-\delta)^n]}{n\epsilon^2 v_i\delta\rho} \sum_{z=k+1}^\infty \sum_{l=1}^\infty l^2 Q_{zl} \frac{v_l}{v_z} \frac{v_z}{t(z)}\\\nonumber
 &= \frac{t(i)K\,[1-(1-\delta)^n]}{n\epsilon^2 v_i\delta\rho} \sum_{z=k+1}^\infty \sum_{l=1}^\infty l^2 Q_{zl} v_l\frac{1}{z^{\nu^*+m'}}\\\label{last2}&<
 \frac{t(i)K\,[1-(1-\delta)^n]\,C}{n\epsilon^2 v_i\delta\rho} \sum_{z=k+1}^\infty \sum_{l=1}^\infty l^{2+\nu^*} Q_{zl} \frac{1}{z^{\nu^*+m'}}.
 \end{align}
 Now we use Minkowski's inequality and Assumption \ref{cond:bounded-moments-xi} to write
 $$\sum_{l=1}^\infty l^{2+\nu^*} Q_{zl} =\mbE[Z_1^{2+\nu^*}\,|\,Z_0=z]\leq z^{2+\nu^*} \mbE[\xi^{2+\nu^*}]\leq z^{2+\nu^*} C'$$ for some constant $C'$. By choosing $m'=4$ we obtain
 $$\eqref{last2}\leq \frac{t(i)K }{n\epsilon^2 v_i} \frac{1-(1-\delta)^n}{\delta}\frac{C C'}{\rho}\sum_{z=k+1}^\infty \frac{z^{2+\nu^*}}{z^{\nu^*+4}}<\infty,$$where the last sum is finite since it is bounded by $\sum_{z=1}^\infty \frac{1}{z^{2}}<\infty$.

\medskip

We now move on to Step 2.
According to \cite[Theorem 3.1(i)]{braunsteins2022parameter}, there exists a sequence of probability spaces (indexed by $n$) such that for any $K$,
\begin{equation}\label{sps}\liminf_{n \to \infty} \mathbb{P}_i (Z_j^{(n)}=Z^{\uparrow}_j, \, \forall \, j=1, \dots, n-K) = 1 - \frac{\rho^{-K}}{2} \sum_{j=1}^\infty u_j | \boldsymbol{e}_j^{\top}\textbf{S}^{K} \mathbf{1} |.\end{equation} Next, using \cite[Inequality (52)]{braunsteins2022parameter} which states that
there exists $L<\infty$ and $\varepsilon > 0$ such that, for any function $t$ identified in \cite[Lemma 13]{braunsteins2022parameter}, $| \boldsymbol{e}_j^{\top}\textbf{S}^{K} \mathbf{1} |\leq t(j) L (1-\varepsilon)^K \rho^K$, we get
$$\eqref{sps} \geq 1- \frac{1}{2} L (1-\varepsilon)^{K} \sum_{j=1}^\infty u_j t(j),$$where $||\vc u||_{1,t}=\sum_{j=1}^\infty u_j t(j)<\infty$  (see \cite[Section 5.3.1]{braunsteins2022parameter}).   Therefore, for any $\eta >0$ there exists $K(\eta) \in \mathbb{N}$ such that
$$\liminf_{n \to \infty} \mathbb{P}_i (Z_j^{(n)}=Z^{\uparrow}_j, \, \forall \, j=1, \dots, n-K) \geq 1- \eta/3.$$

In addition, for any $K(\eta) \in \mathbb{N}$, both $\sum_{j=n-K(\eta)}^n Z^{(n)}_j$ and $\sum_{j=n-K(\eta)}^n Z^\uparrow_j$ converge to non-defective random variables as $n \to \infty$. Indeed, for $\sum_{j=n-K(\eta)}^n Z^{(n)}_j$ this follows from [Gosselin, Theorem 3.1(b)], and for $\sum_{j=n-K(\eta)}^n Z^\uparrow_j$ this follows from the ergodicity of $\{ Z^\uparrow_n \}$.
Consequently for any $\eta>0$ and $K(\eta) \in \mathbb{N}$ there exists $M(K,\eta)$ such that
\[
\limsup_{n \to \infty} \mathbb{P}\left(\sum_{j=n-K(\eta)}^n Z^{(n)}_j > M(K,\eta) \right) \leq \eta/3 \; \text{ and } \;   \limsup_{n \to \infty} \mathbb{P}\left(\sum_{j=n-K(\eta)}^n Z^{\uparrow}_j > M(K,\eta) \right) \leq \eta/3.
\]
We define the event
\begin{eqnarray*}
E^*_n&:=&\bigg\{Z_j^{(n)} = Z^\uparrow_j, \, \forall \, j=1,\dots, n-K(\eta) \bigg\}\\
&& \cap\; \bigg\{ \sum_{j=n-K(\eta)}^n Z^{(n)}_j \leq M(K,\eta) \bigg\} \cap  \bigg\{ \sum_{j=n-K(\eta)}^n Z^\uparrow_j \leq M(K,\eta) \bigg\}.
\end{eqnarray*}
For any $\eta >0$ there then exists $K(\eta) \in \mathbb{N}$ and  $M(K,\eta)$ such that, on the sequence of probability spaces described in \cite[Theorem 3.1]{braunsteins2022parameter}, we have
\begin{align*}
\liminf_{n \to \infty}\, &\mathbb{P}( E^*_n ) \\
&\geq \liminf_{n \to \infty} \bigg[ 1 - \mathbb{P} \bigg(\exists j \in \{1 ,\dots, n-K(\eta)\} : Z_j^{(n)} \neq Z^\uparrow_j \bigg) - \mathbb{P}\bigg(\sum_{j=n-K(\eta)}^n Z^{(n)}_j > M(K,\eta) \bigg)\\
&\qquad -  \mathbb{P}\bigg(\sum_{j=n-K(\eta)}^n Z^{\uparrow}_j > M(K,\eta) \bigg) \bigg] \\
&\geq 1- \eta.
\end{align*}
It therefore suffices to show that, on the event $E^*_n$,
we have
\begin{equation}\label{eq:WUP}
\left| \sum^\infty_{z=k} (V^{(n)}_{nz} - V^\uparrow_{nz} )  \right| \to 0, \quad \text{as } n \to \infty
\end{equation}
a.s. for any $k, K(\eta), M(K,\eta) \in \mathbb{N}$, because in this case we have
\[
\limsup_{n \to \infty}  \mathbb{P}_i \left[ \left| \sum_{z=k+1}^\infty (V^{(n)}_{nz} - V^\uparrow_{nz}) \right|>\varepsilon \right]  \leq 1-\liminf_{n \to \infty} \mathbb{P}(E^{*c}_n) \leq \eta,
\]
where, since $\eta$ was chosen arbitrarily, we can let $\eta \downarrow 0$.

We now establish \eqref{eq:WUP}.
We have
\begin{align*}
\sum_{z=k}^\infty (V_{nz}^{(n)}-V_{nz}^\uparrow) &= \sqrt{n} \sum_{z=k}^\infty f(z) \left[q_n^{(n)}(z)(m^\uparrow(z)-\hat{m}_n^{(n)}(z))-q_n^\uparrow(z)(m^\uparrow(z)-\hat{m}_n^\uparrow(z)) \right] \\
&=  \sqrt{n} \sum_{z=k}^\infty f(z) \left[ \left\{ (q_n^{(n)}(z)- q_n^\uparrow(z))m^\uparrow(z) \right\} + \left\{ \hat{m}_n^\uparrow(z)q_n^\uparrow(z) - \hat{m}_n^{(n)}(z) q_n^{(n)} \right\} \right].
\end{align*}
We bound the sums resulting from these two terms separately.
We start with the second term. Using the assumption that we are on the event $E^*_n$ in steps 4 to 6, we obtain
\begin{align*}
&\left| \sum_{z=k}^\infty  f(z)\left[ \hat{m}_n^{(n)}(z) q_n^{(n)} - \hat{m}_n^\uparrow(z)q_n^\uparrow(z) \right] \right| \\
&= \left| \sum_{z=k}^\infty f(z) \left[ \frac{\sum_{i=0}^{n-1}Z_{i+1}^{(n)}\ind{Z_i^{(n)}=z}}{\sum_{i=0}^{n-1}Z_i^{(n)}}-\frac{\sum_{i=0}^{n-1}Z_{i+1}^\uparrow\ind{Z_i^\uparrow=z}}{\sum_{i=0}^{n-1}Z_i^\uparrow} \right] \right| \\
&=\left| \sum^\infty_{z=k} f(z)\frac{\bigg( \sum_{i=1}^{n-1} Z^{(n)}_{i+1}\ind{Z^{(n)}_i=z}\bigg)\bigg( \sum^{n-1}_{i=0} Z^\uparrow_i \bigg)- \bigg( \sum_{i=0}^{n-1} Z^{\uparrow}_{i+1} \ind{Z^{^{\uparrow}}_i = z} \bigg)\bigg( \sum_{i=0}^{n-1}Z_i^{(n)}\bigg)}{\bigg(\sum_{i=0}^{n-1} Z^{(n)}_i \bigg) \bigg(\sum_{i=0}^{n-1} Z^\uparrow_i \bigg)} \right|\\
&=\frac{1}{\left(\sum_{i=0}^{n-1}Z_i^{(n)}\right)\left(\sum_{i=0}^{n-1}Z_i^\uparrow \right)}\Bigg| \sum_{z=k}^{\infty} f(z) \Bigg[ \bigg( \sum_{i=0}^{n-K(\eta)-1}Z^{(n)}_{i+1} \ind{Z_i^{(n)} =z} \bigg)\bigg( \sum_{i=0}^{n-1}(Z^\uparrow_i - Z^{(n)}_i) \bigg) \\
&\qquad+\bigg( \sum^{n-1}_{i=n-K(\eta)} Z^{(n)}_{i+1} \ind{Z_i^{(n)}=z} \bigg)\bigg( \sum_{i=0}^{n-1} Z^\uparrow_i \bigg) - \bigg( \sum_{i=n-K(\eta)}^{n-1} Z^{\uparrow}_{i+1} \ind{Z^\uparrow_i=z} \bigg) \bigg( \sum^{n-1}_{i=0} Z_i^{(n)} \bigg) \Bigg] \Bigg| \\
&\leq M^* \left\{ \left| \frac{\bigg(\sum_{i=0}^{n-1}(Z^\uparrow_i - Z^{(n)}_i) \bigg)\bigg( \sum_{z=k}^\infty \sum_{i=0}^{n-K(\eta)-1}Z_{i+1}^{(n)} \ind{Z_i^{(n)}=z} \bigg) }{\bigg(\sum_{i=0}^{n-1} Z^{(n)}_i \bigg) \bigg(\sum_{i=0}^{n-1} Z^\uparrow_i \bigg)}\right| \right. \\
&\qquad + \left. \left|\frac{\sum_{z=k}^\infty \sum_{i=n-K(\eta)}^{n-1} Z_{i+1}^{(n)}\ind{Z_i^{(n)}=z}}{ \sum_{i=0}^{n-1} Z_i^{(n)}}\right| +\left| \frac{\sum_{z=k}^\infty \sum_{i=n-K(\eta)}^{n-1} Z_{i+1}^{\uparrow}\ind{Z_i^{\uparrow}=z}}{ \sum_{i=0}^{n-1} Z_i^{\uparrow} } \right| \right\} \\
&\leq 3 M^* M(K,\eta) / n,
\end{align*}
where in the penultimate step we have applied the triangular inequality,  and Assumption~\ref{cond:bound-m-arrow}\ref{cond:m-arrow-bound-deriv}, and in the final step, we have applied the inequalities $\sum_{i=0}^{n-1} Z^{(n)}_i \geq n$ and $\sum_{i=0}^{n-1} Z^{\uparrow}_i \geq n$ and the fact that we are on $E_n^*$.

Applying similar arguments (which are therefore omitted) we obtain the same bound for the first term.
We have thus established the result.\qed

\blue{

\section{Additional simulation results}
\label{appB}

\blue{
\subsection{Offspring means and mean function in the $Q$-process}\label{sec:color_map}

Before complementing the numerical results of Section~\ref{empirical}, we first illustrate (see Figure~\ref{fig:means-examples}) the offspring mean function $m(z)$ together with the equivalent function $m^\uparrow(z)$ in the $Q$-process, for the three Beverton-Holt binary splitting models considered, that is with $(K_0, v_0) =(100,0.6)$, $(K_0, v_0) =(25,0.7)$, and $(K_0, v_0) =(25, 0.53)$. We observe that in the first two cases, the functions $m(z)$ and $m^\uparrow(z)$ only differ at small population sizes. In addition, in these cases, the process does not visit those small population sizes very often (see Figures~\ref{oooooh} and \ref{fig:ex3:paths}). When $(K_0, v_0) =(100,0.6)$ or $(25,0.7)$, we therefore expect the respective conditional limits $\T_0$ and $\tilde{\T}$ of the estimators $\hat{\T}_n$ and $\tilde{\T}_n$ to be very similar. On the other hand, when $(K_0, v_0) =(25, 0.53)$, we see that the difference between $m(z)$ and $m^\uparrow(z)$ is greater, and noticeable even for population sizes which are greater than the carrying capacity. In addition, the process corresponding to this model tends to spend more time in small population sizes (see Figure~\ref{oooooh}).  Consequently, we expect that the limits  $\T_0$ and $\tilde{\T}$ are significantly different.

\begin{figure}
\centering\includegraphics[width=0.32\textwidth]{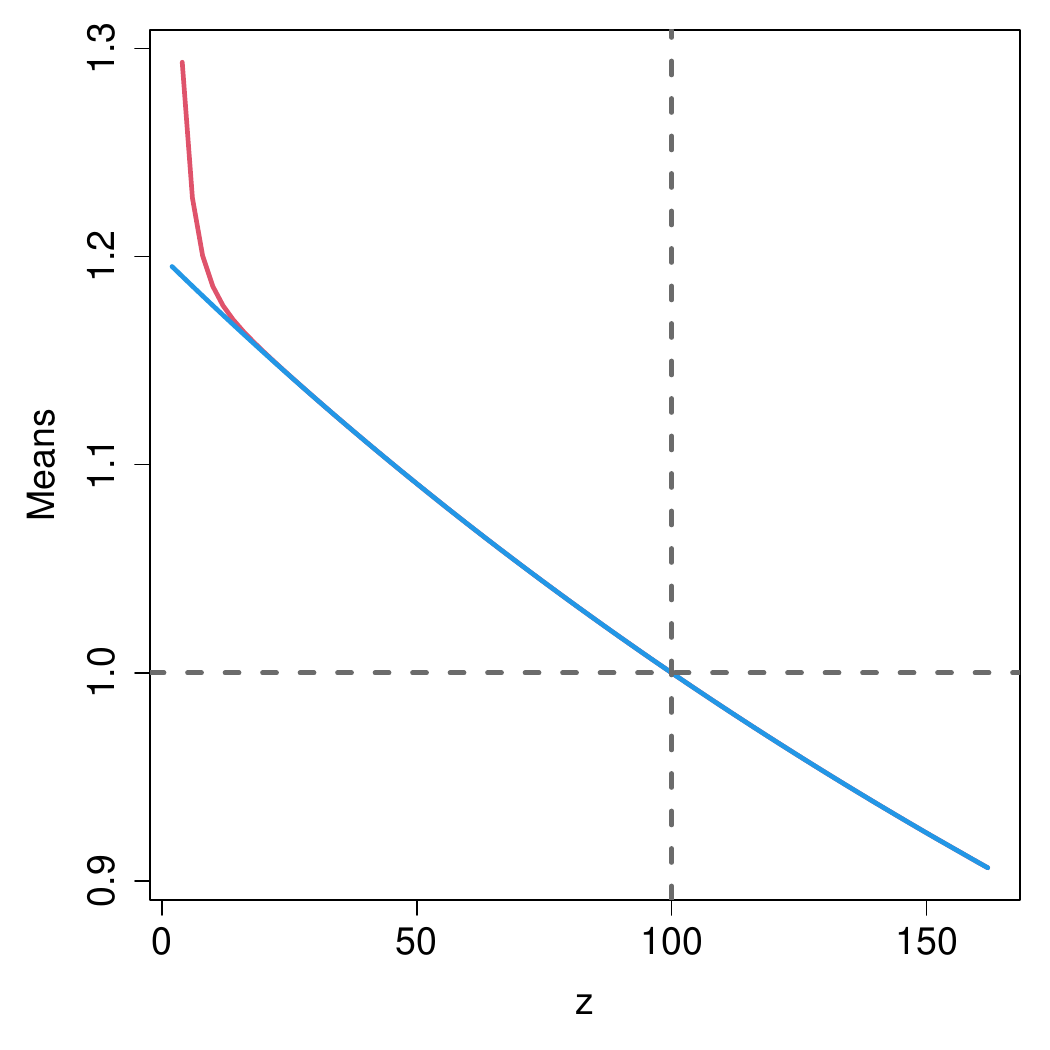}
\includegraphics[width=0.32\textwidth]{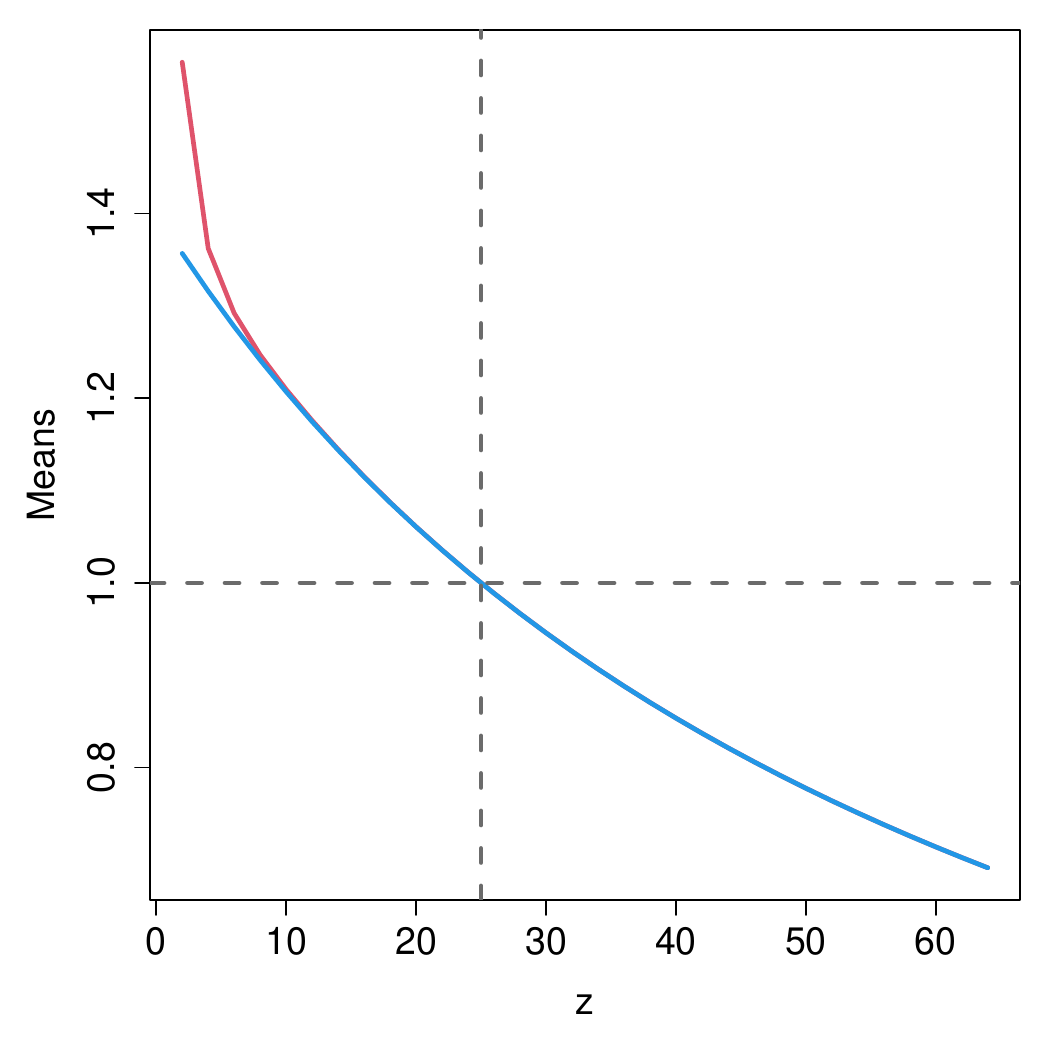}
\includegraphics[width=0.32\textwidth]{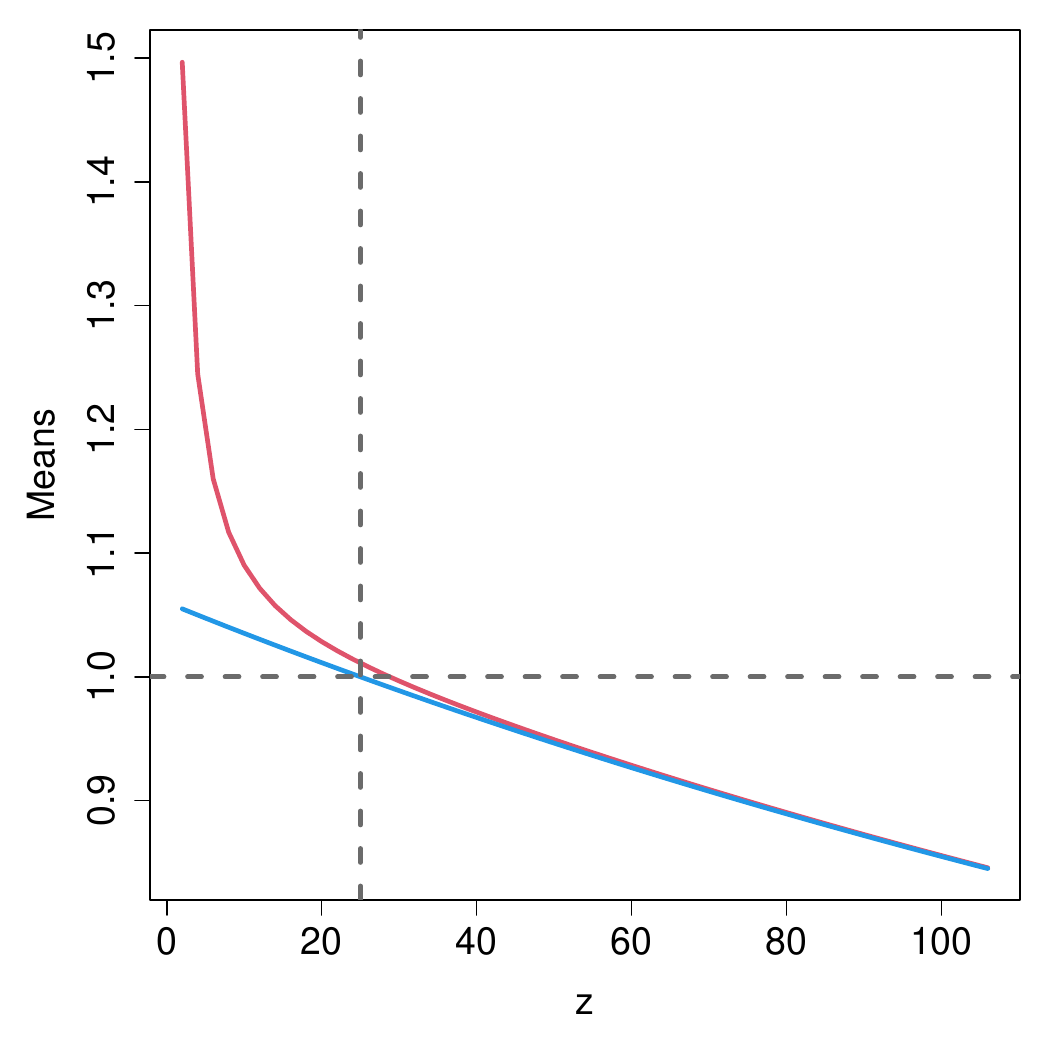}
\caption{\label{fig:means-examples}
\blue{\textsf{Beverton-Holt binary splitting model}. Left: $(K_0, v_0) =(100,0.6)$. Centre: $(K_0, v_0) =(25,0.7)$. Middle: $(K_0, v_0) =(25, 0.53)$. Offspring mean function $m(z)$ (blue) together with its counterpart $m^\uparrow(z)$ in the $Q$-process (red).}}
\end{figure}
\begin{figure}
\centering\includegraphics[width=0.51\textwidth]{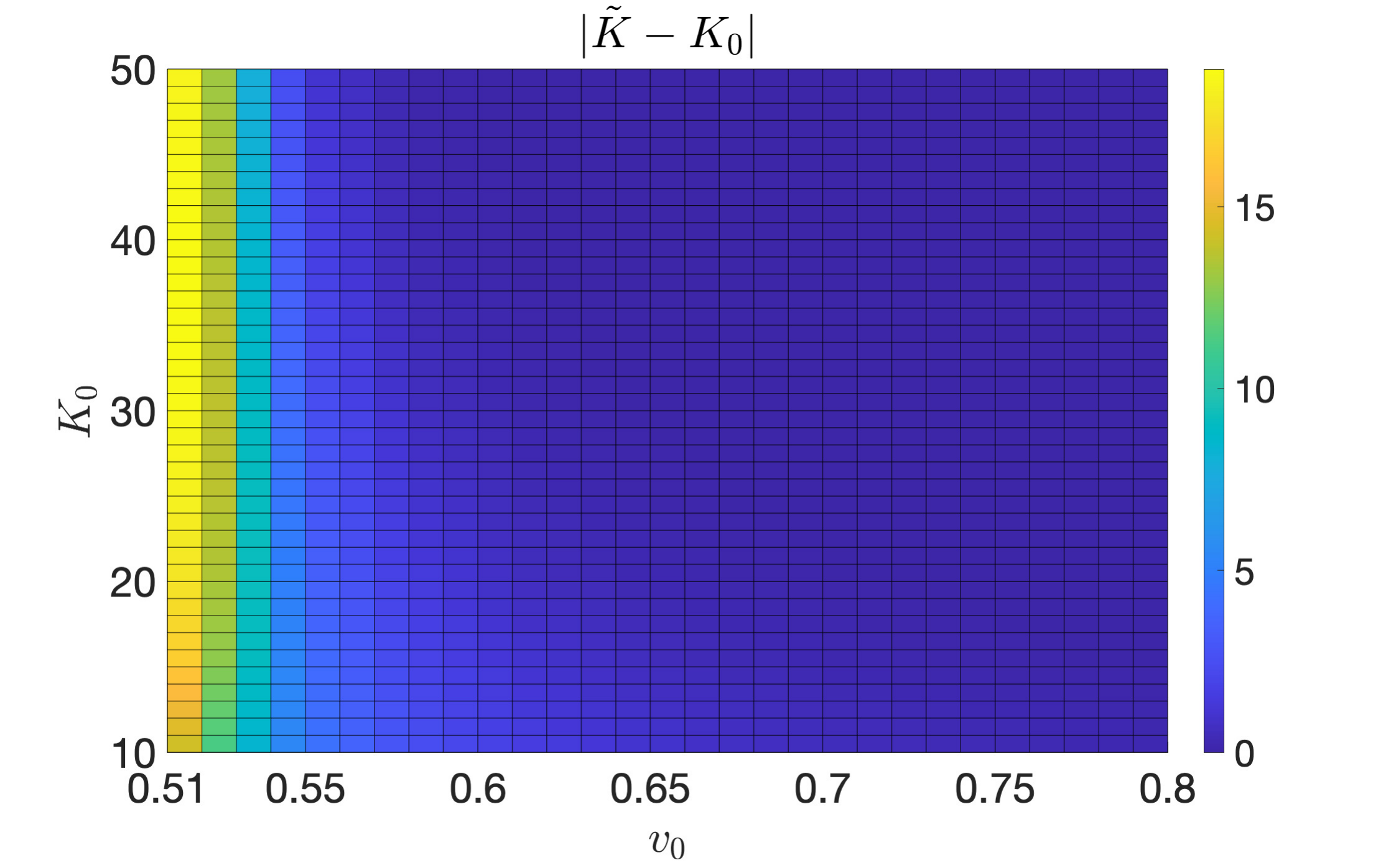}
\hspace{-0.7cm}\includegraphics[width=0.51\textwidth]{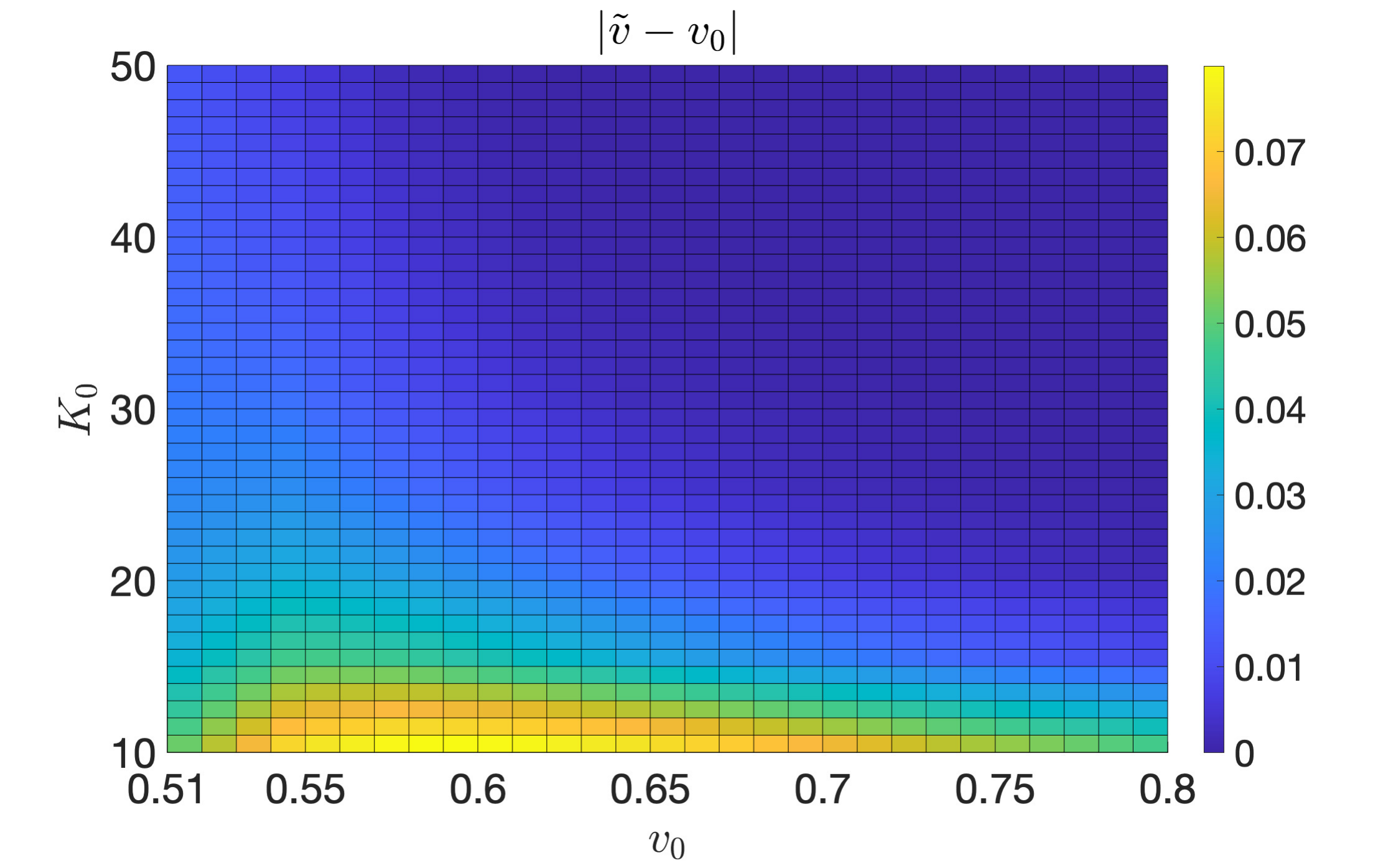}
\caption{\label{fig:diff_theta_thetatilde}
\blue{\textsf{Beverton-Holt binary splitting model}. Absolute differences between $\T_0$ and $\tilde{\T}$ for different combinations of the parameters $(K_0, v_0)$.}}
\end{figure}

This intuition can be confirmed using Proposition \ref{thm:Q-consistency-WLSE-1}, which effectively describes how differences in $m(z)$ and $m^\uparrow(z)$ translates into differences in $\hat{\T}_n$ and $\tilde{\T}_n$ for large $n$ (i.e. to differences in their respective conditional limits $\T_0$ and $\tilde{\T}$). In addition, we can use Proposition \ref{thm:Q-consistency-WLSE-1} to display differences in $\T_0$ and $\tilde{\T}$ for a wide range of combinations of the parameters $(K_0, v_0)$. In the left and right panels of Figure \ref{fig:diff_theta_thetatilde} we plot $|\tilde{K}-K_0|$ and $|\tilde{v}-v_0|$, respectively, for $(K_0,v_0)\in [10,50]\times [0.51,0.8]$, using weights $\{w_z^{(2)}\}$. We generally observe that if $K_0$ and $v_0$ are smaller, then $|\tilde{K}-K_0|$ and $|\tilde{v}-v_0|$ are larger (although not always). This agrees with our intuition because when $K_0$ and $v_0$ are small, the process spends longer at small population sizes, where the difference between $m(z)$ and $m^\uparrow(z)$ is greater. More surprisingly, Figure \ref{fig:diff_theta_thetatilde} also shows that $|\tilde{K}-K_0|$ is more sensitive with respect to $v_0$, whereas $|\tilde{v}-v_0|$ is more sensitive with respect to $K_0$.
}

\subsection{Unconditional trajectories}\label{subsec:conterfactual}

Here, we examine the quality of the estimators $\hat{\T}_n$ and $\tilde{\T}_n$ applied to trajectories simulated up to a given generation $n$, \emph{without imposing the condition $Z_n>0$}. 
Note that $\hat{\T}_n$ is designed to account for the bias introduced under the condition $Z_n>0$ (i.e. when extinct trajectories are excluded). Without the condition $Z_n>0$, we expect the classical estimator $\tilde{\T}_n$ to be less biased than the new estimator $\hat{\T}_n$ ---the reverse of what we generally observe in the numerical results in Section~\ref{empirical}. Our goal is to further illustrate the difference between the two estimators.

We simulate $1000$ trajectories in the growing phase of the Beverton-Holt model considered in Section \ref{growing_pop} (with $K=100$ and $v=0.6$) up to times $n=5,10, 15, 20, 25$; among these trajectories, respectively $365, 433, 442, 447, 490$ became extinct. Typically, as PSDBPs initially grow like supercritical branching processes, if extinction occurs before reaching the carrying capacity, it happens quite rapidly (for example, the mean and standard deviation of the extinction time for the trajectories generated up to $n=25$ are $4.6327$ and $3.5607$, respectively). In Figures \ref{cont_K} and \ref{cont_v}, we present boxplots of the estimates for $K$ and $v$ using $\hat{\T}_n$ and $\tilde{\T}_n$ with weighting function $\{w_z^{(2)}\}$, considering (i) all 1000 trajectories combined in the left panels, (ii) only the surviving trajectories (with $Z_n>0$) in the middle panels, and (iii) only the extinct trajectories (with $Z_n=0$) in the right panels.

As expected, the estimates obtained from the surviving trajectories exhibit similar qualitative properties as in Section \ref{growing_pop}, showing a smaller bias for the new estimator $\hat{\T}_n$. In contrast, the estimates derived from extinct trajectories are poor, due to the limited data available in these trajectories. This is particularly evident in the case of the estimates $\hat{K}_n$
and for the estimates $\tilde{v}_n$. Consequently, when combining all trajectories (including both surviving and extinct ones), the estimators have a greater bias
than when considering only the non-extinct trajectories, with this being more pronounced for  $\hat{K}_n$.

In summary,
if the data consist of an extinct trajectory
there is no need to consider the new estimator $\hat{\T}_n$. However, in the case of a non-extinct trajectory (i.e. a trajectory generated under the condition $Z_n>0$), the results suggest that $\hat{\T}_n$ has a lower bias.

\begin{figure}
\centering\includegraphics[width=\textwidth]{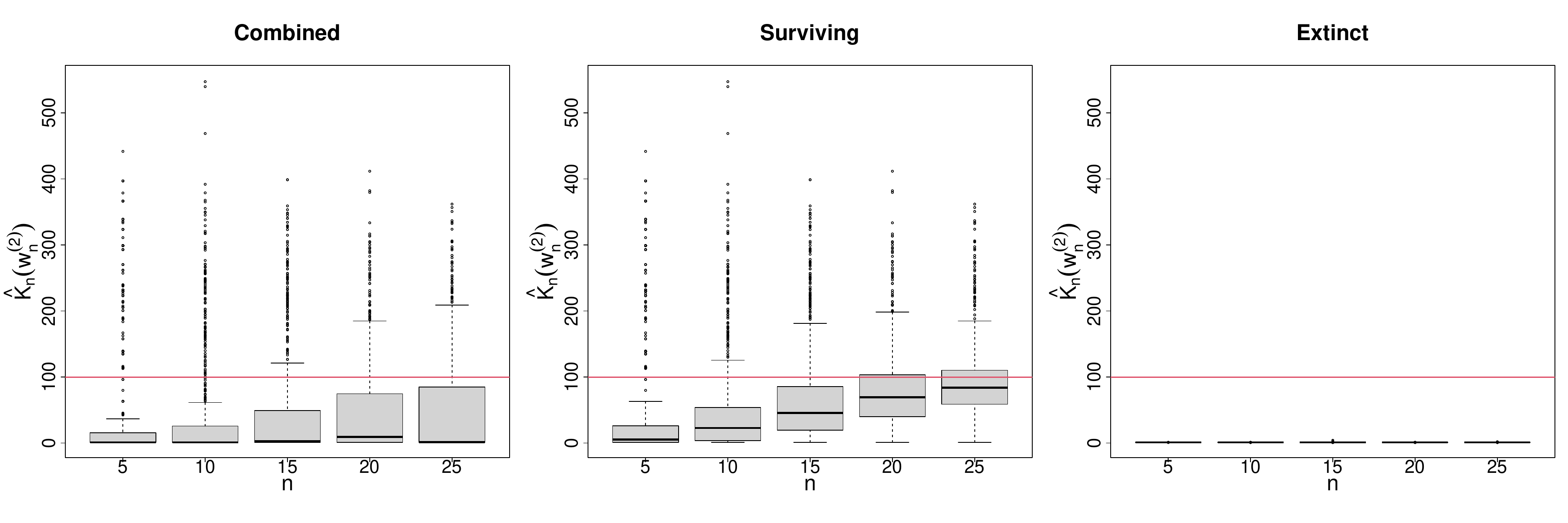}
\centering\includegraphics[width=\textwidth]{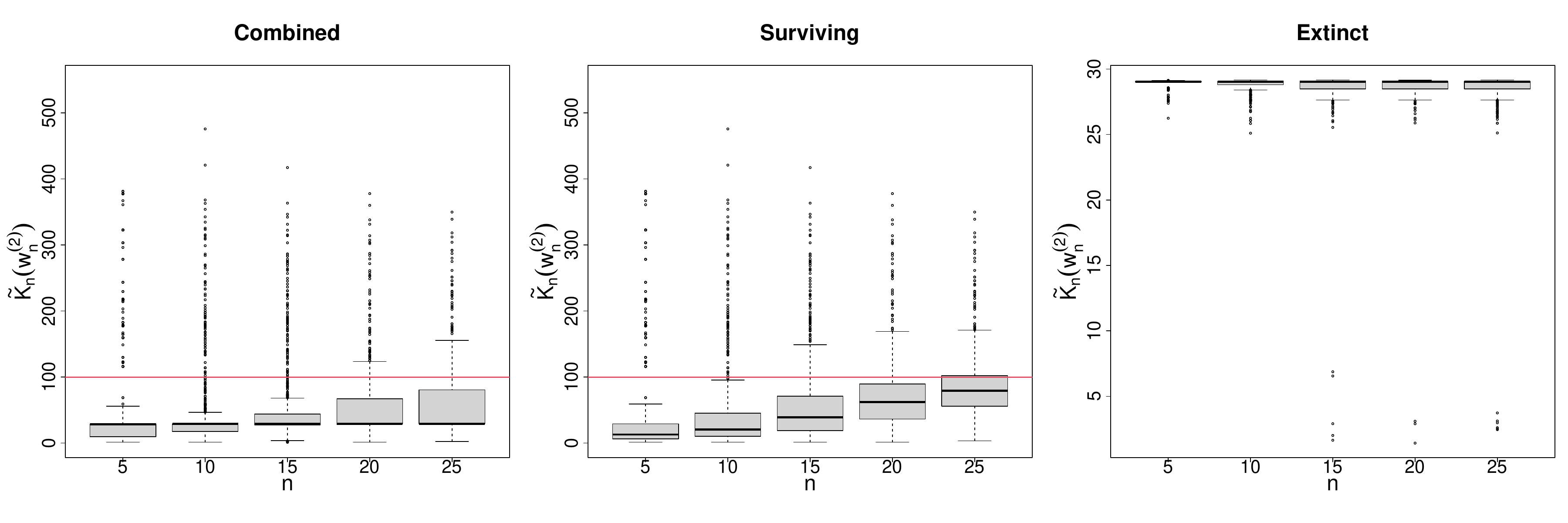}
\caption{\label{cont_K}
\blue{\textsf{Beverton-Holt model with $(K_0, v_0) =(100,0.6)$}. Box plots of $\widehat{K}_n$ and $\widetilde{K}_n$ with the weighting function $\boldsymbol{w}_{n}^{(2)}$. Red lines represent the true value of each parameter. Left panels: combined trajectories; Middle panels: surviving trajectories; Right panels: extinct trajectories.}}
\end{figure}

\begin{figure}
\centering\includegraphics[width=\textwidth]{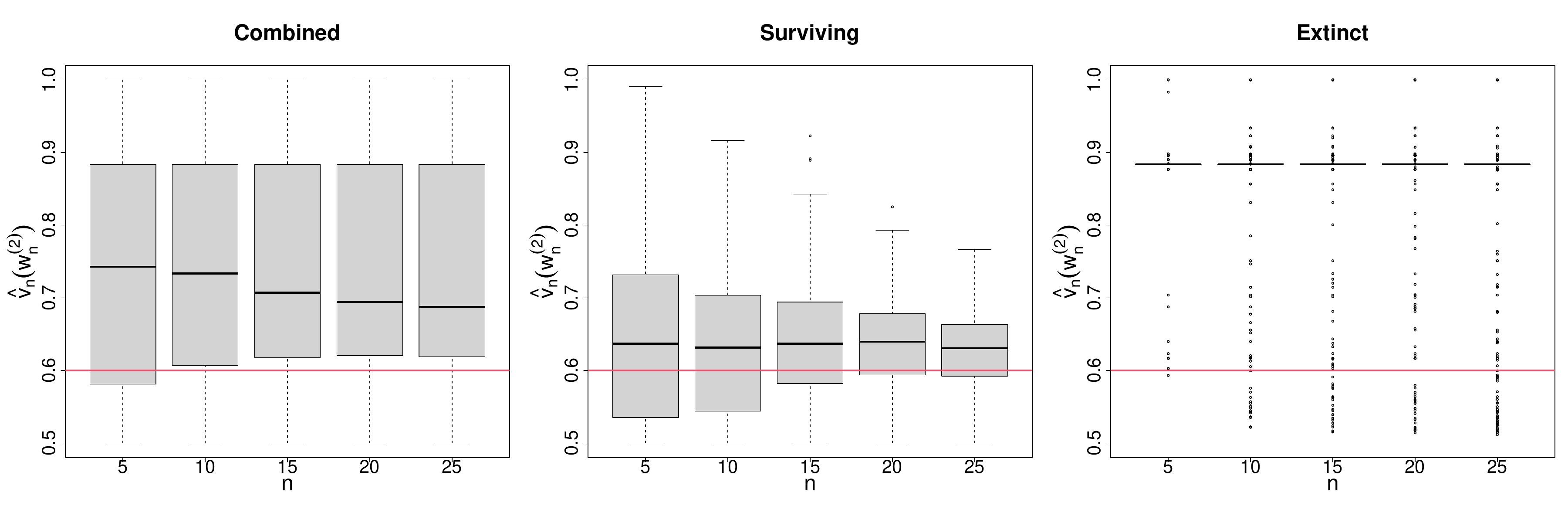}
\centering\includegraphics[width=\textwidth]{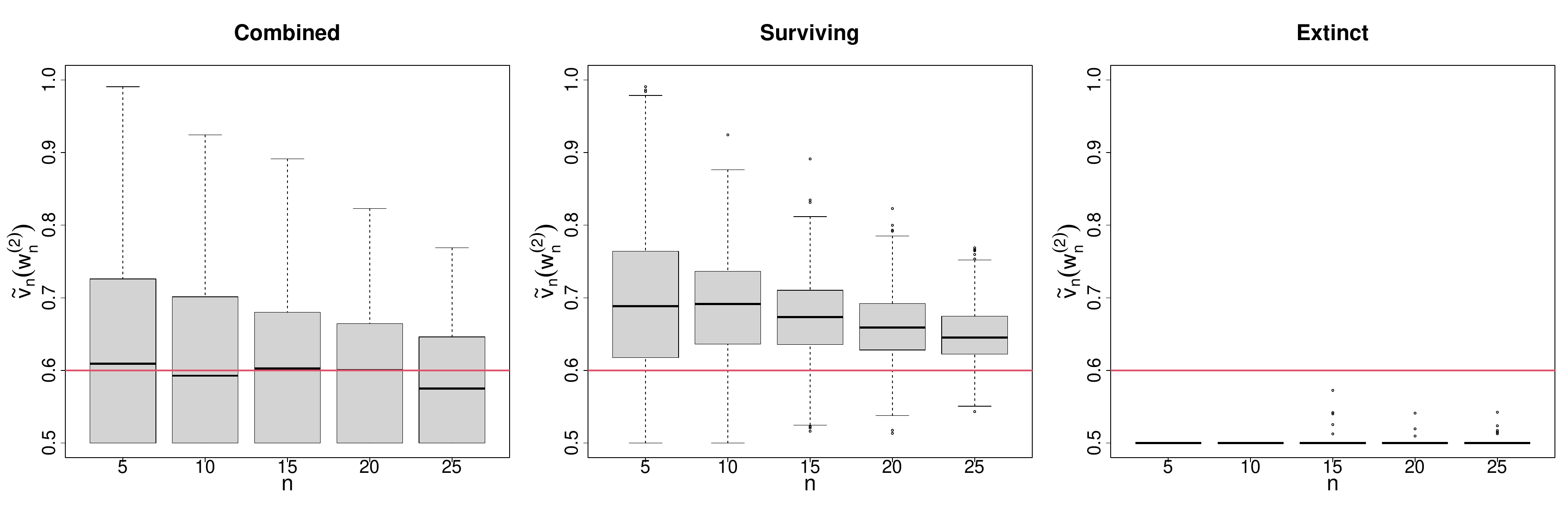}
\caption{\label{cont_v}\blue{\textsf{Beverton-Holt model with $(K_0, v_0) =(100,0.6)$}. Box plots of $\widehat{v}_n$ and $\widetilde{v}_n$ with the weighting function $\boldsymbol{w}_{n}^{(2)}$. Red lines represent the true value of each parameter. Left panels: combined trajectories; Middle panels: surviving trajectories; Right panels: extinct trajectories.}
}
\end{figure}

\subsection{Bias correction}\label{subsec:bias-correction}

While we have proved that  $\widehat{\T}_n$ is asymptotically unbiased, this does not necessarily mean that it is unbiased for small $n$. Indeed, 
the numerical examples in Section \ref{empirical} show that the estimators $\widehat{\T}_n$ and $\widetilde{\T}_n$ are biased for fixed values of $n$; this is especially apparent for small values of $n$ in the growing phase. We therefore investigate the effect of bias correction on $\widehat{\T}_n$ and $\widetilde{\T}_n$ by applying the bootstrap bias correction method (see for example \cite[Section 10.6]{tibshirani1993introduction}) with 1000 trajectories (leading to 1000 point estimates for each estimator), and 100 bootstrap trajectories per point estimate (used to estimated the bias). More specifically,  the bias-corrected estimator corresponding to $\widehat{\T}_n$ is defined as
$$\bar{\T}_n=\widehat{\T}_n-\left(\frac{1}{B}\sum_{i=1}^B \T^{(B)}_{n,i} - \widehat{\T}_n\right),$$ where $B$ is the number of bootstrap samples ($B=100$ here), and $\T^{(B)}_{n,i}$ is the point estimate resulting from the $i$th bootstrap sample from the model with parameter $\widehat{\T}_n$; a similar definition applies to $\widetilde{\T}_n$.

The results are presented in Figure \ref{bias_corr} for the Beverton-Holt model considered in Section~\ref{growing_pop}, for times $n=15, 25, 35$. The graphs indicate that the bias improvements, while more noticeable for $v$ than for $K$, are not substantial. One possible explanation for the bias correction method's relatively poorer performance for $K$ could be the over-estimation of the bias due to the skewed distribution of the estimators $\widehat{\T}_n$ and $\widetilde{\T}_n$.
We also note that the bias corrected estimators generally have a larger variance. In addition, since $B$ additional point estimates are required for each population trajectory, the bias-corrected estimator is more computationally expensive. Another limitation is that the bias-corrected estimators may produce values outside the parameter space, as is the case for $v$ (whose values should be in $(0.5,1]$).
Finally, as opposed to the new estimators $\widehat{\T}_n$, the bias corrected estimators are not easily interpretable and their consistency is not guaranteed.

\begin{figure}
\centering\includegraphics[width=0.4\textwidth]{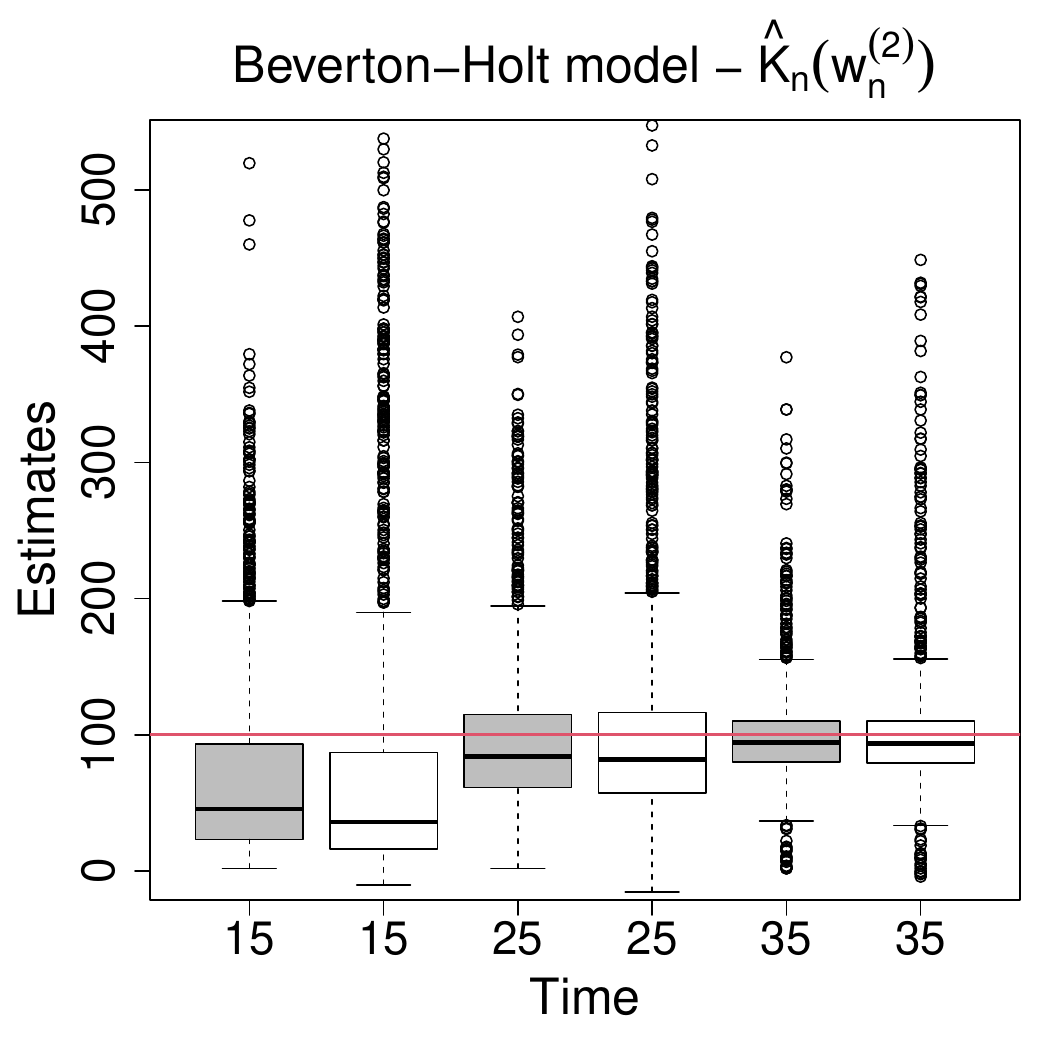}
\centering\includegraphics[width=0.4\textwidth]{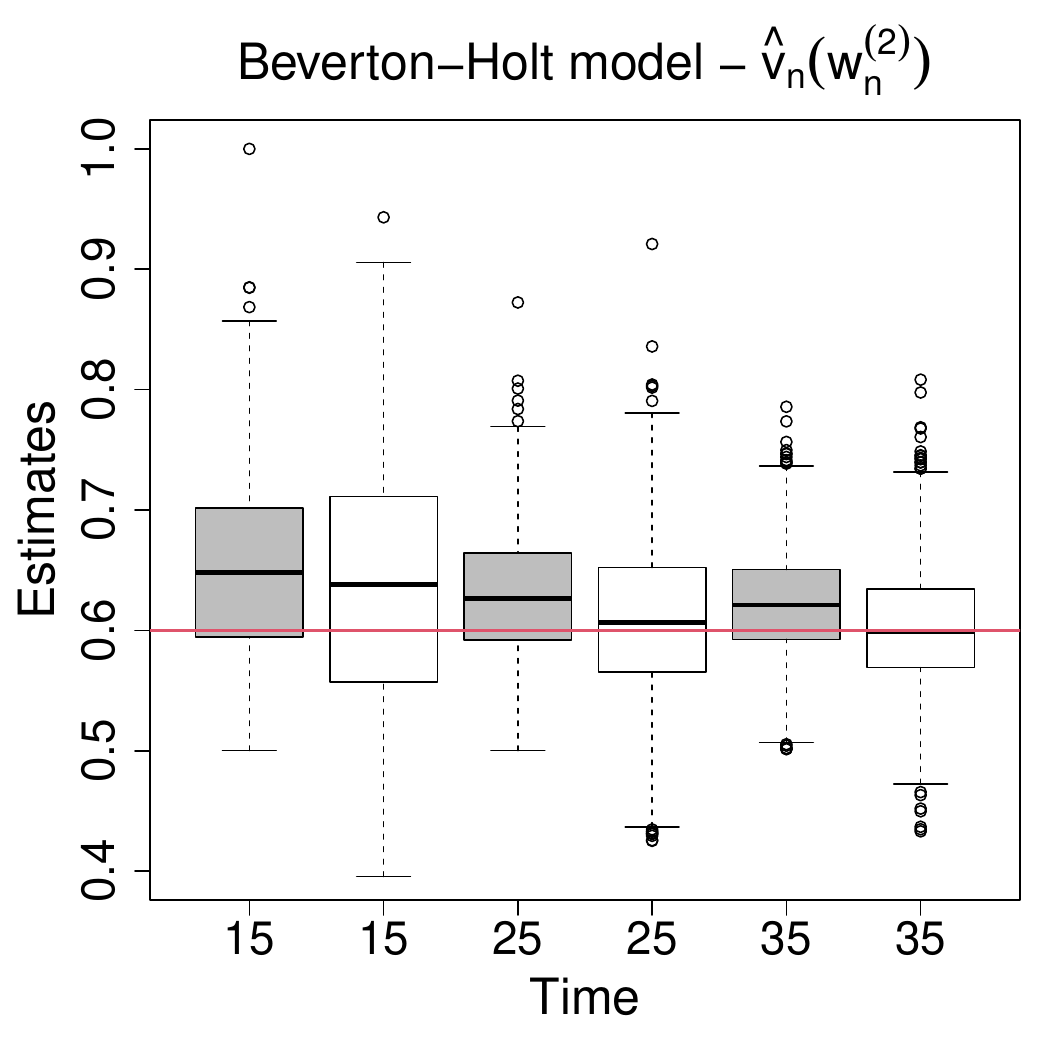}\\
\centering\includegraphics[width=0.4\textwidth]{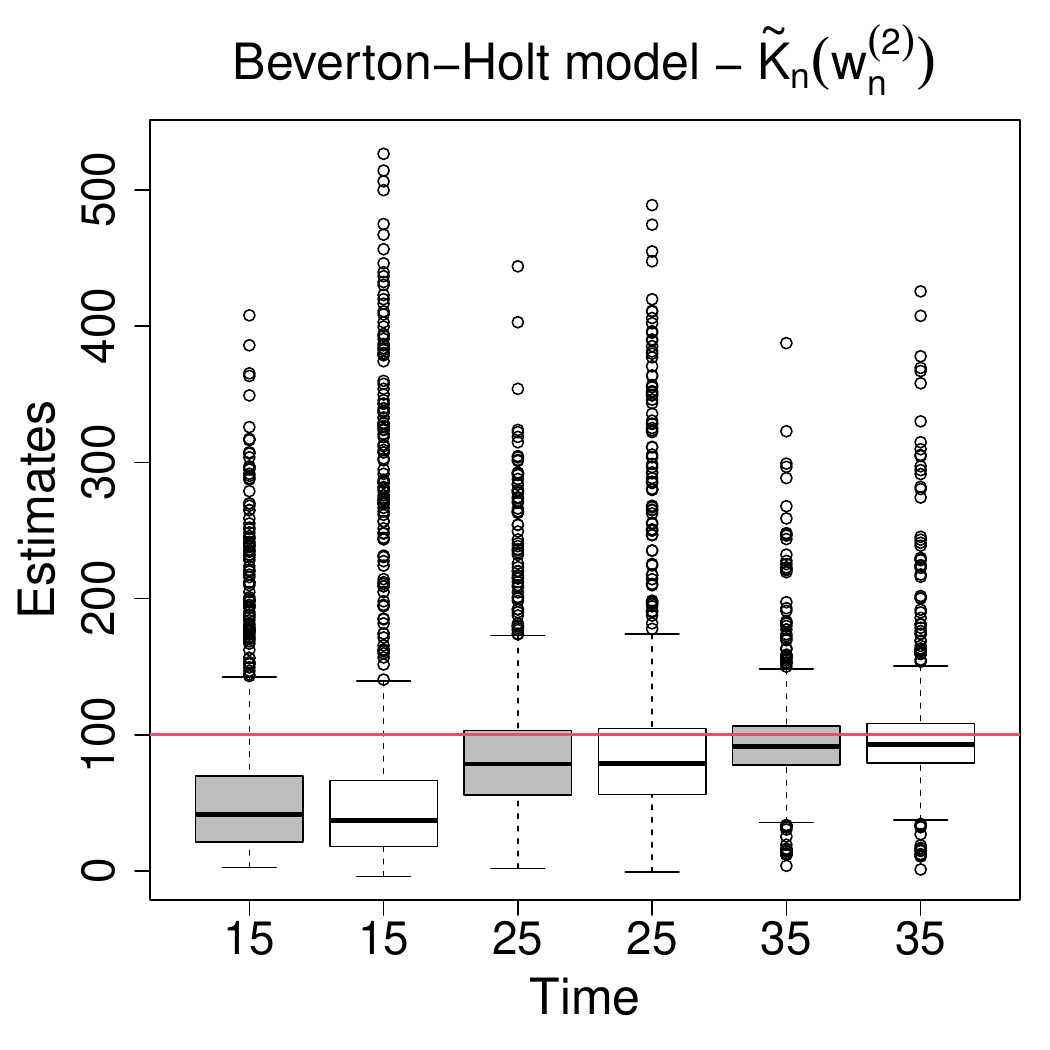}
\centering\includegraphics[width=0.4\textwidth]{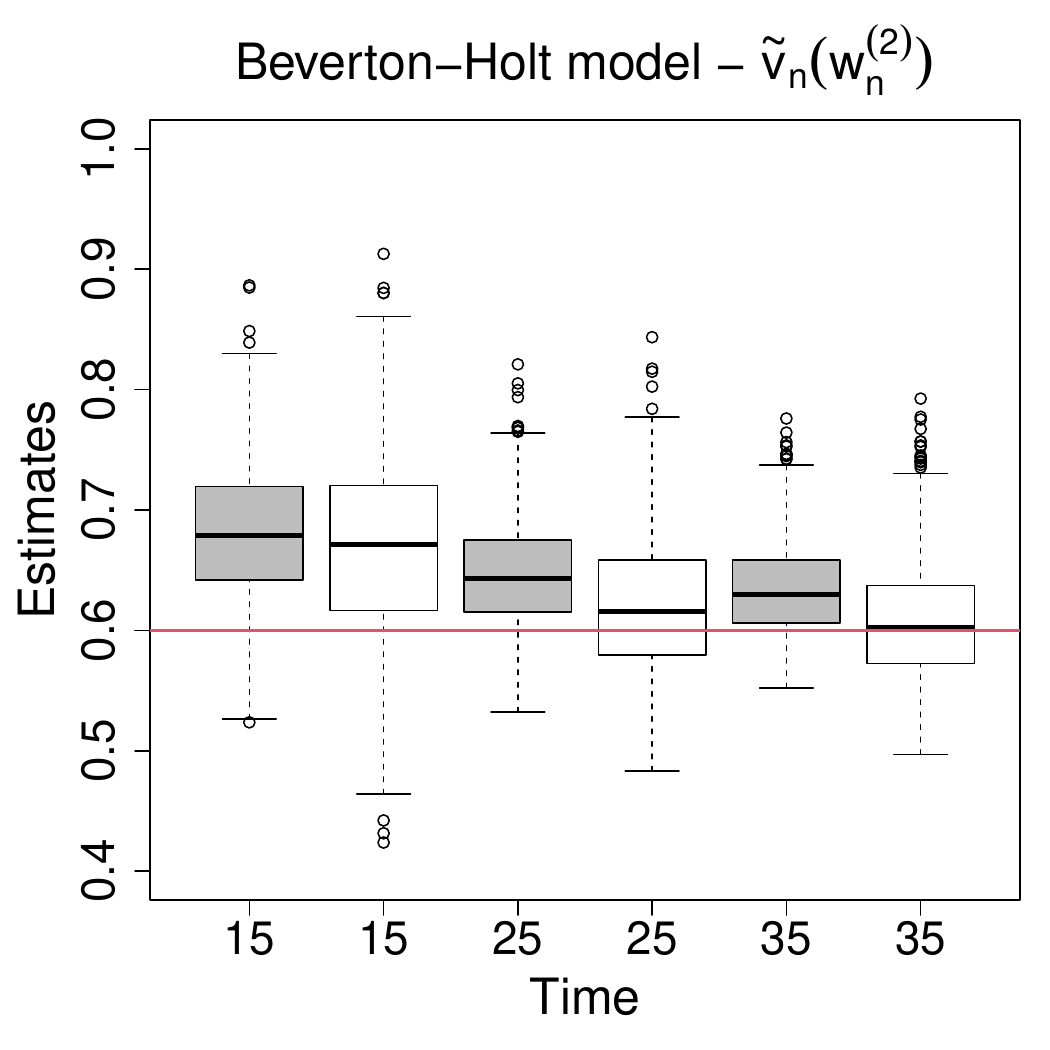}
\caption{\label{bias_corr}\blue{\textsf{Beverton-Holt model with $(K_0, v_0) =(100,0.6)$}. Box plots of $\widehat{\T}_n$ and $\widetilde{\T}_n$ (grey) and their bias corrected version (white) with the weighting function $\boldsymbol{w}_{n}^{(2)}$. Red lines represent the true value of each parameter.
}}
\end{figure}

}

\blue{

\subsection{Coverage properties and efficiency}\label{sec:cov}

In Section~\ref{qsp} we investigated the distribution of the new estimator $\widehat{\T}_n$ for large $n$, and compared it to the asymptotic distribution
 implied by Theorem~\ref{thm:C-normality-WLSE-1}. We now continue this analysis by computing
 empirical confidence regions for the parameter $\T_0$
 and comparing them to the theoretical confidence regions implied by Theorem~\ref{thm:C-normality-WLSE-1}.
 To compute the theoretical confidence regions,  
 we use the Mahalanobis distance with the covariance matrix $\vc \beta(\hat{\T}_n^*)/n$, where $\hat{\T}_n^*$ is the empirical mean of the estimates $\hat{\T}_n$. To compute the empirical confidence regions we use the function \texttt{ci2d()} of the R package \texttt{gplots}.

The left panels of Figures~\ref{fig:conf-regions-non-bannana} and \ref{fig:conf-regions-bannana} display the 90\% and 95\% empirical and theoretical confidence regions based on the estimator $\widehat{\T}_n$ for trajectories of populations alive after $n=5000$ time units when $(K_0,v_0)=(25,0.7)$ and $(25,0.53)$. In these plots, we observe that the empirical and theoretical confidence regions are in close correspondence, although this is clearer when $(K_0,v_0)=(25,0.7)$ than when $(K_0,v_0)=(25,0.53)$. Next we investigate the coverage properties of these confidence regions. We do so by simulating independent non-extinct trajectories of the process, computing $\widehat{\T}_n$ and the corresponding confidence region implied by Theorem~\ref{thm:C-normality-WLSE-1} for each trajectory, and finally computing the proportion of these confidence regions that contain the true parameter $\T_0$.
The results are gathered in Table~\ref{tab:coverage}, where we observe that for large values of $n$ the corresponding proportions when $\T_0=(25,0.7)$ are around 95\%, as expected. The results for $\T_0 =(25,0.53)$ seem to be less satisfactory, however, this is due to the fact that the value $v_0=0.53$ is closer to the boundary of the parameter space $\Theta$ (recall that $v$ takes its values in $(0.5,1]$), and therefore the empirical distribution is not as close to the theoretical distribution given by Theorem~\ref{thm:C-normality-WLSE-1}.

The right panels of Figures 6 and 7 display the $90\%$ and $95\%$ empirical confidence regions based on the estimator $\widetilde{\T}_n$. Note that there are no corresponding theoretical confidence regions in this case.

\begin{figure}
\centering\includegraphics[width=0.3\textwidth]{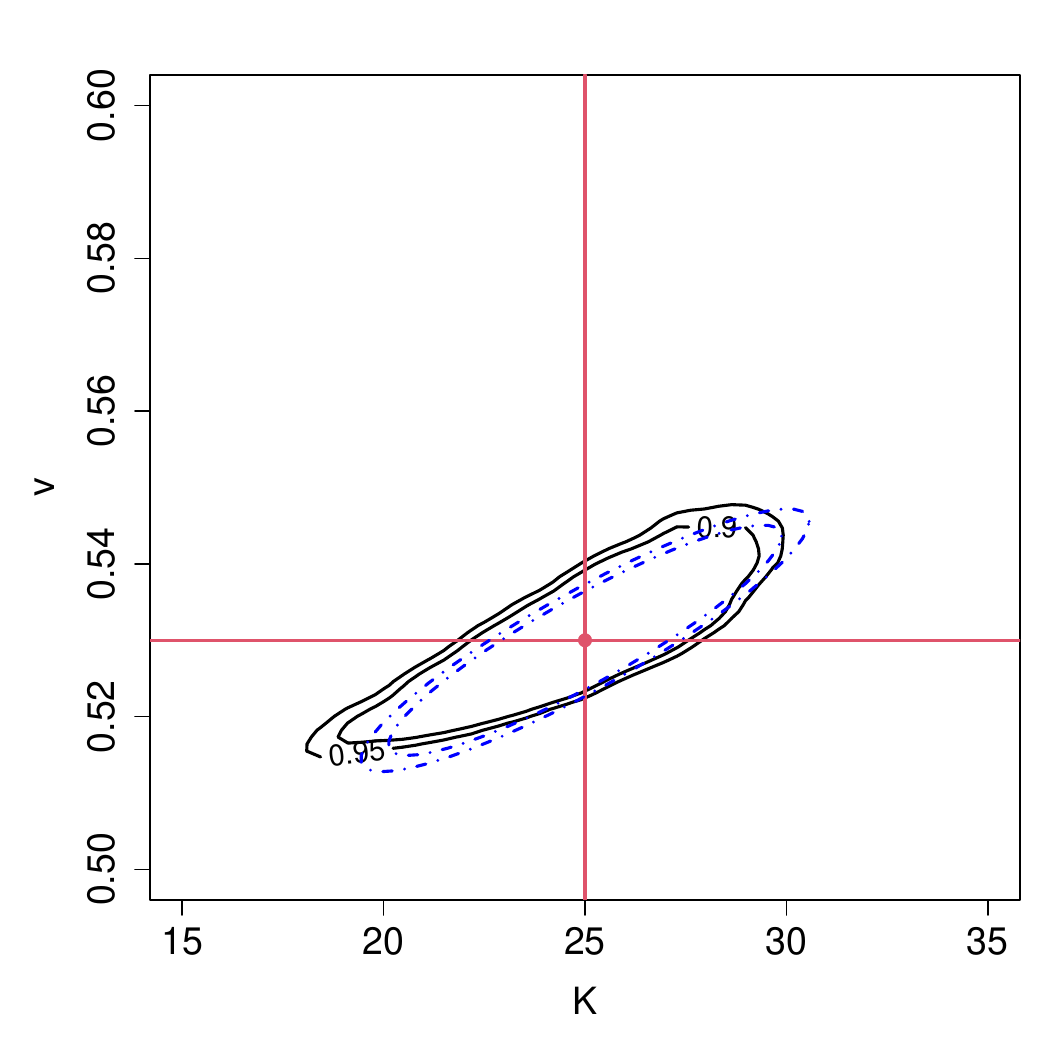}\hspace{1em}
\includegraphics[width=0.3\textwidth]{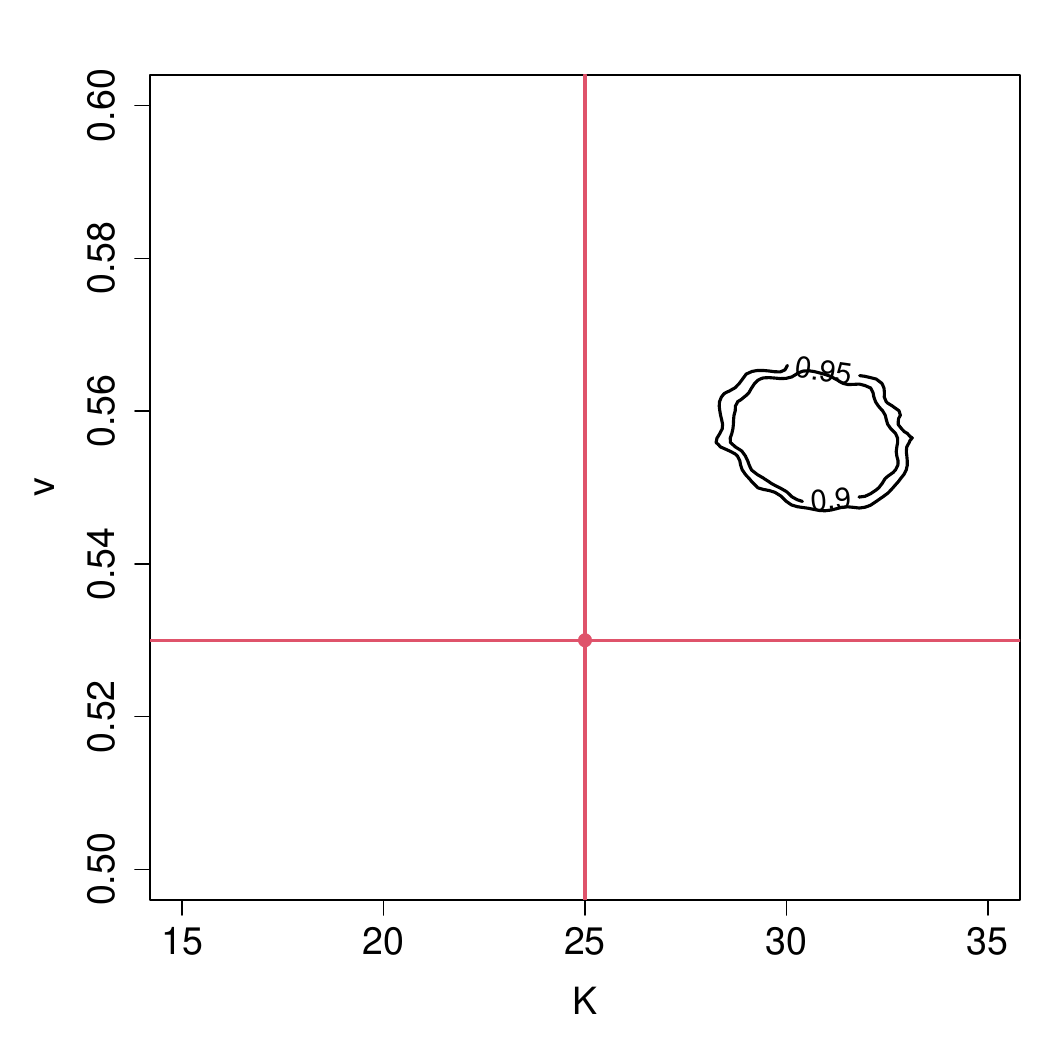}
\caption{\blue{\textsf{Case $(K_0,v_0)=(25,0.7)$}. Empirical confidence regions (black solid line) and theoretical confidence regions (dashed-dotted blue lines) for levels $90\%,95\%$.  Left: based on the estimator $\hat{\T}_n(w_n^{(2)})$. Right: based on the estimator $\tilde{\T}_n(w_n^{(2)})$.}\label{fig:conf-regions-non-bannana}}
\end{figure}

\begin{figure}
\centering\includegraphics[width=0.3\textwidth]{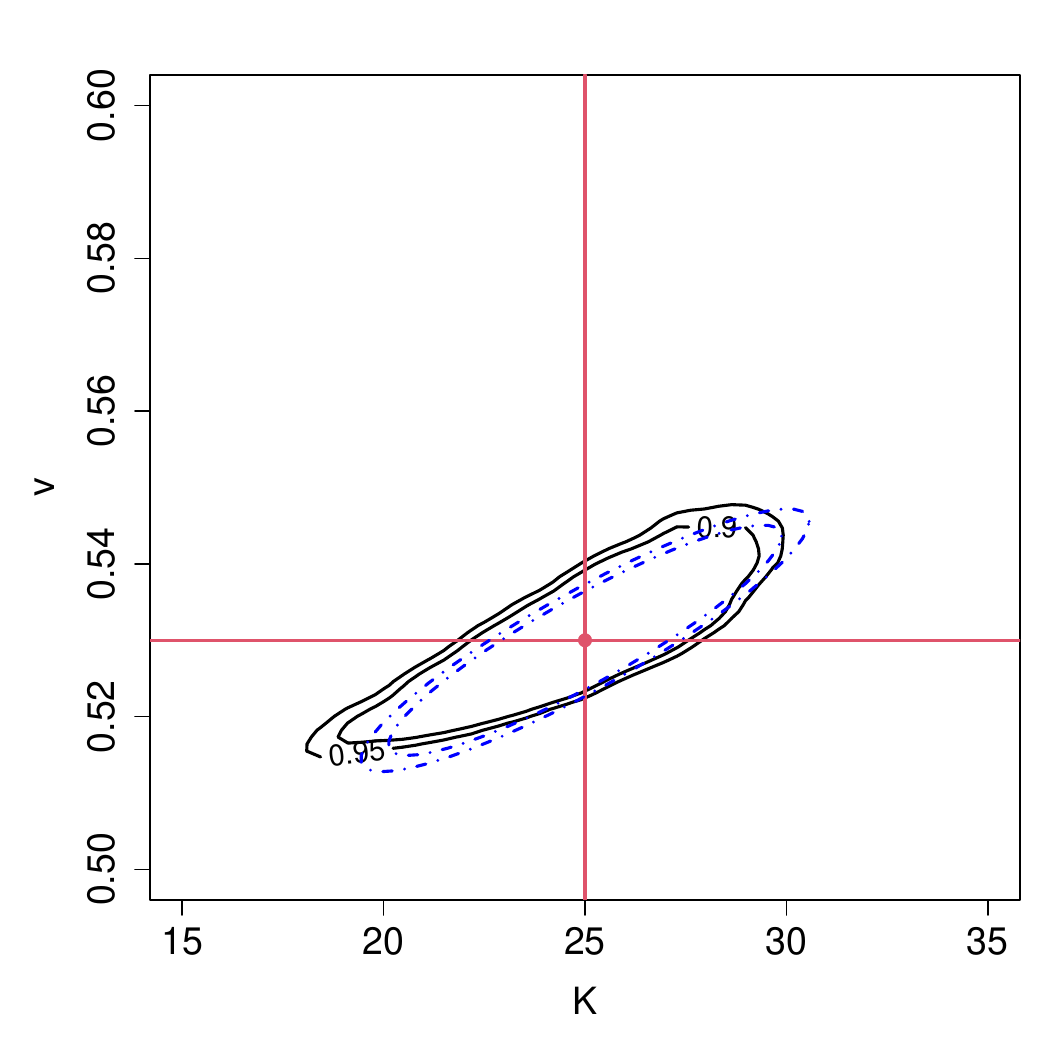}\hspace{1em}
\includegraphics[width=0.3\textwidth]{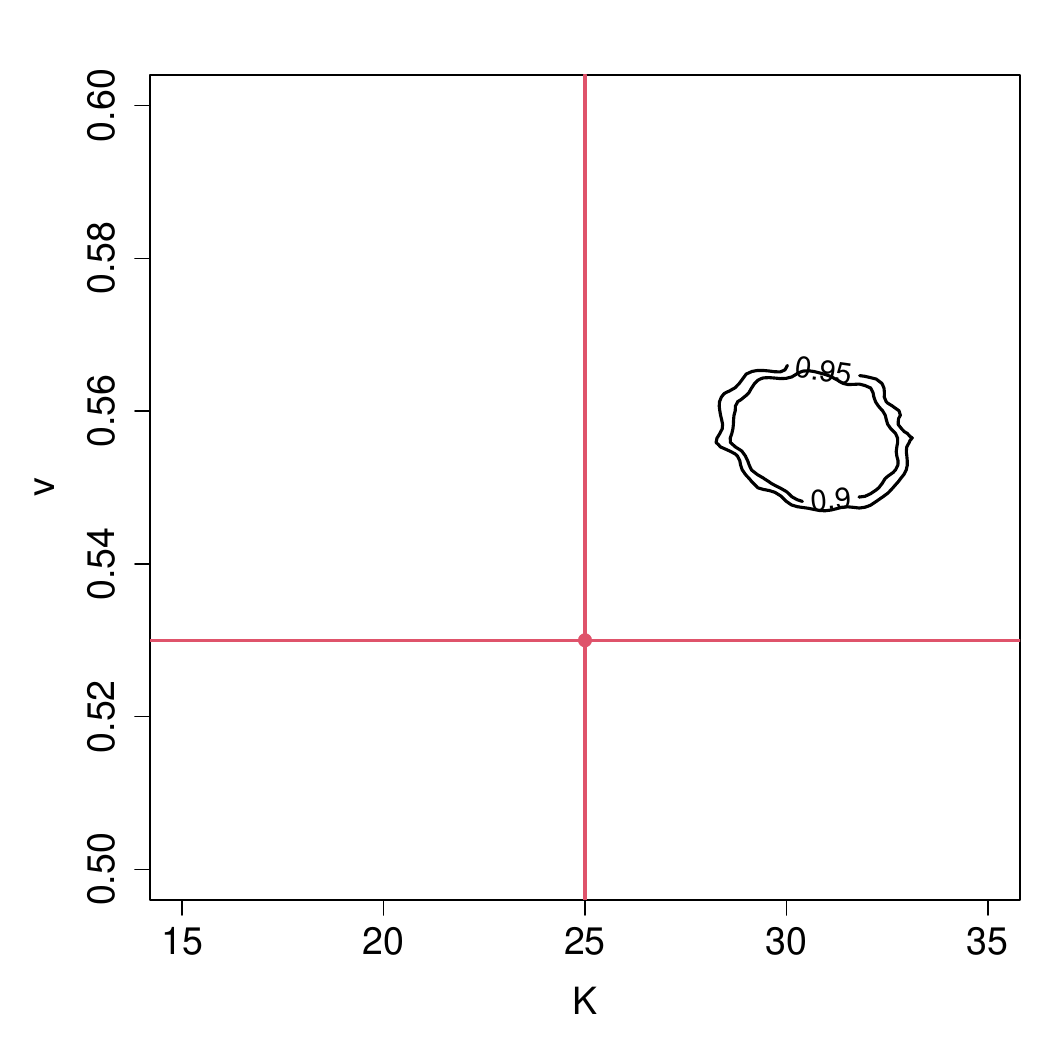}
\caption{\blue{\textsf{Case $(K_0,v_0)=(25,0.53)$}. Empirical confidence regions (black solid line) and theoretical confidence regions (dashed-dotted blue lines) for levels $90\%,95\%$.  Left: based on the estimator $\hat{\T}_n(w_n^{(2)})$. Right: based on the estimator $\tilde{\T}_n(w_n^{(2)})$.}\label{fig:conf-regions-bannana}}
\end{figure}

\begin{table} \centering
\scalebox{0.75}{\blue{\begin{tabular}{c|cc|cc|}
\cline{2-5}
&\multicolumn{2}{c|}{$\T_0 =(25,0.7)$}&\multicolumn{2}{c|}{$\T_0 =(25,0.53)$}\\[0.5ex]
\cline{2-5}
\multicolumn{1}{c|}{} & $\hat{\T}_n(\boldsymbol{w}_{n}^{(1)})$ & $\hat{\T}_n(\boldsymbol{w}_{n}^{(2)})$ & $\hat{\T}_n(\boldsymbol{w}_{n}^{(1)})$ & $\hat{\T}_n(\boldsymbol{w}_{n}^{(2)})$ \\[0.5ex]
\hline
\multicolumn{1}{|c|}{$n=1000$} & $0.952$ & $0.948$ & $0.788$ & $0.855$ \\
\multicolumn{1}{|c|}{$n=2000$} & $0.948$ & $0.949$ & $0.837$ & $0.898$ \\ 
\multicolumn{1}{|c|}{$n=3000$} & $0.950$ & $0.954$ & $0.860$ & $0.912$ \\
\multicolumn{1}{|c|}{$n=4000$} & $0.947$ & $0.955$ & $0.866$ & $0.915$ \\
\multicolumn{1}{|c|}{$n=5000$} & $0.952$ & $0.957$ & $0.892$ & $0.924$ \\
\hline
\end{tabular}}}
  \caption{\blue{\textsf{Beverton-Holt model with \blue{$\T_0 =(25,0.7)$} and \blue{$\T_0 =(25,0.53)$}.} Proportion of the 95\% confidence regions for each trajectory containing the true value $\T_0$.}} \label{tab:coverage}
\end{table}

\bigskip

The efficiency of the new estimator $\widehat{\T}_n$ can be investigated by computing the empirical covariance matrix and comparing it with the Cram\'er-Rao lower bound. This analysis is quite involved and is left as a topic for future research.

\subsection{Bootstrap confidence intervals}\label{subsec:boot_CI}

Theorem \ref{thm:C-normality-WLSE-1} provides the asymptotic distribution of $\hat{\vc\theta}_n$, which can be used to derive confidence intervals for the model parameters. However, in non-asymptotic settings such as for the growing populations considered in Section \ref{growing_pop} and the black robin application in Section \ref{sec:br}, the confidence intervals obtained using Theorem \ref{thm:C-normality-WLSE-1} may be misleading. In these cases, parametric bootstrap approaches can be used to obtain empirical distributions of estimators and confidence intervals.

	In our setting (and assuming a single parameter w.l.o.g.), the general idea is to first estimate the model parameter using the (single) observed population trajectory, yielding the estimate $\hat{\theta}$. Then, we generate $M$ bootstrap samples by simulating new population trajectories from the fitted model (i.e., the model with parameter $\hat{\theta}$). For each bootstrap sample, we re-estimate the parameter, resulting in $M$ (ordered) bootstrap estimates $\hat{\theta}^*_{(1)}, \ldots, \hat{\theta}^*_{(M)}$. 
	Confidence intervals can be constructed by taking percentiles from the empirical distribution of the bootstrap estimates, such as the 2.5th and 97.5th percentiles for a $95\%$ \textit{percentile} confidence interval:
	$$L=\hat{\theta}^*_{(0.025)},\qquad U= \hat{\theta}^*_{(0.975)};$$ see for instance  (\cite[Section	13.3]{tibshirani1993introduction}).
	In order to account for the difference between the true parameter $\theta$ and the estimate $\hat{\theta}$, the \textit{pivot-based} confidence interval (a.k.a. \textit{bootstrap-t} interval  \cite[Section	13.4]{tibshirani1993introduction}) assumes that the behaviour of $\theta-\hat{\theta}$ is approximately the same as that of $\hat{\theta}-\hat{\theta}^*$. Therefore,
	\begin{eqnarray*} 
	0.95 &=& P[\hat{\theta}^*_{(0.025)}\leq \hat{\theta}^*\leq \hat{\theta}^*_{(0.975)}]
	\\
	&\approx &P[\hat{\theta}-\hat{\theta}^*_{(0.025)}\geq \theta-\hat{\theta}\geq \hat{\theta}-\hat{\theta}^*_{(0.975)}]\\
	&=& P[2\hat{\theta}-\hat{\theta}^*_{(0.025)}\geq \theta\geq 2\hat{\theta}-\hat{\theta}^*_{(0.975)}], 
	\end{eqnarray*}
	yielding the pivot confidence interval: $$L_p=2\hat{\theta}-\hat{\theta}^*_{(0.975)},\qquad U_p=2\hat{\theta}-\hat{\theta}^*_{(0.025)}.$$

	In order to obtain narrow confidence intervals with accurate coverage probabilities, we need the $M$ simulated population trajectories to be a ``good representation'' of the original data. However, for growing populations such as the black robins, it is not immediately clear what the best method is for simulating trajectories that are representative of
	the original data. 
	We explore three approaches to generate bootstrap samples: $(i)$ simulating (non-extinct) trajectories of the same length as the original data, $(ii)$ simulating trajectories with the same total cumulative number of individuals, and $(iii)$ simulating trajectories with the same final population size. 
	
	We compare these approaches by performing a bootstrap confidence interval analysis using the classical estimator $\tilde{\T}_{n}(\boldsymbol{w}_{n}^{(2)})$ for a Beverton-Holt binary splitting model with $(K_0,v_0)=(110,0.7)$, assuming the data consists of non-extinct trajectories of length $n=30$ (growing phase). The empirical results, based on 1000 model trajectories bootstrapped $M=1000$ times, are shown in Tables \ref{test_boot_Q_1}, \ref{test_boot_Q_2}, \ref{test_boot_Q_3}. 

	These results indicate that the coverage probabilities of the confidence interval using methods $(i)$ and $(iii)$ are better than those using $(ii)$, while bias (median of $(L+U)/2$) and precision (median of $U-L$) are much better for methods $(i)$ and $(ii)$ than for $(iii)$ (especially for $K$). This preliminary investigation suggests that simulating (non-extinct) bootstrap trajectories of the same length as the original data offers the best trade-off between coverage probability and bias/precision.

Finally we compute percentile confidence intervals and their pivot versions for the parameters $(K,v)$ in the Beverton-Holt model for the black robin population described in Section~\ref{sec:br}.
	The results are shown in Table  \ref{br_boot_Q_1} for the classical estimator $\tilde{\T}_{n}(\boldsymbol{w}_{n}^{(2)})$, and in Table  \ref{br_boot_C_1} for the $C$-consistent estimator $\hat{\T}_{n}(\boldsymbol{w}_{n}^{(2)})$.  The results are more consistent across the three simulation methods for $v$ than for $K$. This is because the black robin dataset is a relatively short trajectory, starting from a single individual. 
	Due to the random nature of the simulated trajectories, different constraints imposed by the three different resampling approaches can lead to different patterns in the trajectories, which generally impacts the estimates of $K$ more than those of $v$. For instance, in the third approach, we may generate more trajectories exhibiting exponential-like growth, resulting in very large estimates for $K$. This is particularly apparent for the classical estimator (Table~\ref{br_boot_Q_1}), as reflected in the high values for $U$ and the corresponding negative values for $L_p$, which we set to 0.

	
	\begin{table} \centering
\blue{\begin{tabular}{|cc|c|c|c|c|}
\cline{3-6}
\multicolumn{2}{c|}{ }&\multicolumn{2}{c|}{$K$} & \multicolumn{2}{c|}{$v$}\\
\hline
& & $K_0\in[L,U]$ & $K_0\in[L_p,U_p]$  & $v_0\in [L,U]$  & $v_0\in [L_p,U_p]$\\
\hline
\multirow{3}{*}{$(i)$} &
$85\%$    & $0.7970$ &$0.8070 $  & $0.7040$ & $0.8380$ \\ &
$90\%$    & $ 0.7720$ & $0.8650$ & $0.7554$ & $0.8830$ \\ &
$95\%$    & $0.8450$ & $0.9200$ & $0.7614$ & $0.9330$ 
\\
\hline
\multirow{3}{*}{$(ii)$} &
$85\%$    & $0.7960$ & $0.8010$ & $0.7030$ & $0.8340  $ \\ &
$90\%$    & $0.7750 $ & $0.8470$ & $0.7559$ & $0.8840$ \\ &
$95\%$    & $0.8360$ & $0.8990$ & $0.7624$ & $ 0.9280$ 
\\
\hline
\multirow{3}{*}{$(iii)$} &
$85\%$    & $0.9090$ & $0.8570$ & $0.8360$ & $0.8700$ \\ &
$90\%$    & $0.8850$ & $0.8860$ & $0.7452$ & $0.9080$ \\ &
$95\%$    & $0.9390$ & $ 0.9110$ & $0.7511$ & $0.9480$ 
\\
\hline
\end{tabular}}
  \caption{\blue{\textsf{Coverage probabilities of bootstrap CIs}. Empirical probabilities computed from 1000 trajectories of length $n=30$ of the Beverton-Holt binary splitting model with $(K_0,v_0)=(110,0.7)$, for each of which we generated $M=1000$ bootstrap samples, and used the classical estimator $\tilde{\T}_{n}(\boldsymbol{w}_{n}^{(2)})$.
   Three methods were used to generate bootstrap samples: $(i)$ simulating trajectories of the same length as the original data ($n=30$), $(ii)$ simulating trajectories with the same total cumulative number of individuals, and $(iii)$ simulating trajectories with the same final population size.
  \label{test_boot_Q_1}}}

\end{table}

	\begin{table} \centering
\blue{\begin{tabular}{|cc|c|c|c|c|}
\cline{3-6}
\multicolumn{2}{c|}{ }&\multicolumn{2}{c|}{$K$} & \multicolumn{2}{c|}{$v$}\\
\hline
& & Median of $\frac{L+U}{2}$ & Median of $\frac{L_p+U_p}{2}$ & Median of $\frac{L+U}{2}$ & Median of $\frac{L_p+U_p}{2}$\\
\hline
\multirow{3}{*}{$(i)$} &
$85\%$    & $109.0482$ &$109.6672 $  & $0.7519$ & $0.6959$ \\ &
$90\%$    & $109.4512$ & $109.0265$ & $0.7554$ & $0.6932$ \\ &
$95\%$    & $110.0038$ & $107.1984$ & $0.7614$ & $0.6874$ 
\\
\hline
\multirow{3}{*}{$(ii)$} &
$85\%$    & $109.6766$ & $109.0267$ & $0.7529$ & $0.6955 $ \\ &
$90\%$    & $109.9137$ & $108.4575$ & $0.7559$ & $0.6925$ \\ &
$95\%$    & $110.6507$ & $107.3913$ & $0.7624$ & $ 0.6859$ 
\\
\hline
\multirow{3}{*}{$(iii)$} &
$85\%$    & $137.0492$ & $82.0341$ & $0.7411$ & $0.7071$ \\ &
$90\%$    & $148.0171$ & $ 71.6185$ & $0.7452$ & $0.7036$ \\ &
$95\%$    & $178.2165$ & $41.7926$ & $0.7511$ & $0.6979$ 
\\
\hline
\end{tabular}}
  \caption{\blue{\textsf{Bias of bootstrap CIs}. Median of the CI's midpoints (to compare with the true parameter value $(K_0,v_0)=(110,0.7)$), computed from 1000 replicates of $M=1000$ bootstrap samples.
   Three methods were used to generate bootstrap samples: $(i)$ simulating trajectories of the same length as the original data ($n=30$), $(ii)$ simulating trajectories with the same total cumulative number of individuals, and $(iii)$ simulating trajectories with the same final population size.
  \label{test_boot_Q_2}}}

\end{table}

	\begin{table} \centering
\blue{\begin{tabular}{|cc|c|c|}
\cline{3-4}
\multicolumn{2}{c|}{ }
&  Median of $U-L$ for $K$ & Median of $U-L$ for $v$ \\
\hline
\multirow{3}{*}{$(i)$} &
$85\%$    & $25.8973$  & $0.1539 $  \\ &
$90\%$    & $29.9663$  & $0.1764$  \\ &
$95\%$    & $36.4705$ & $0.2117$ 
\\
\hline
\multirow{3}{*}{$(ii)$} &
$85\%$    & $25.2143$  & $0.1542$ \\ &
$90\%$    & $28.9566$  & $0.1769$  \\ &
$95\%$    & $34.7326 $  & $0.2120$ \\ 
\hline
\multirow{3}{*}{$(iii)$} &
$85\%$    & $78.6534$  & $0.1694$  \\ &
$90\%$    & $102.5739$  & $0.1942$  \\ &
$95\%$    & $166.6130$ & $ 0.2322$ 
\\
\hline
\end{tabular}}
  \caption{\blue{\textsf{Precision of bootstrap CIs}. Median of the CI's widths, computed from 1000 replicates of $M=1000$ bootstrap samples.
   Three methods were used to generate bootstrap samples: $(i)$ simulating trajectories of the same length as the original data ($n=30$), $(ii)$ simulating trajectories with the same total cumulative number of individuals, and $(iii)$ simulating trajectories with the same final population size.
  \label{test_boot_Q_3}}}

\end{table}

	
	\begin{table} \centering
\blue{\begin{tabular}{|cc|c|c|c|c|}
\cline{3-6}
\multicolumn{2}{c|}{ }&\multicolumn{2}{c|}{$K$} & \multicolumn{2}{c|}{$v$}\\
\hline
& & $[L,U]$ & $[L_p,U_p]$ & $[L,U]$ & $[L_p,U_p]$\\
\hline
\multirow{3}{*}{$(i)$} &
$85\%$    & $[78.18,  139.98]$ & $[74.71,  136.51]$ & $[0.60 ,   1]$ & $[0.43, 0.83]$ \\ &
$90\%$    & $[74.98,  149.88]$ & $[64.81,  139.71]$ & $[0.59 ,   1]$ & $[0.43, 0.84]$ \\ &
$95\%$    & $[69.30,  173.90]$ & $[40.79,  145.39]$ & $[0.56,1]$ & $[0.43,    0.87]$ 
\\
\hline
\multirow{3}{*}{$(ii)$} &
$85\%$    & $[74.83,  156.45]$ & $[58.24,  139.86]$ & $[0.60 ,   1]$ & $[0.43, 0.83]$ \\ &
$90\%$    & $[70.04,  173.44]$ & $[41.25,  144.65]$ & $[0.57 ,   1]$ & $[0.43, 0.85]$ \\ &
$95\%$    & $[63.77,  204.84]$ & $[9.85,  150.92]$ & $[0.55,1]$ & $[0.43,    0.88]$ 
\\
\hline
\multirow{3}{*}{$(iii)$} &
$85\%$    & $[82.25,  236.30]$ & $[58.24,  132.44]$ & $[0.57 ,   1]$ & $[0.43, 0.86]$ \\ &
$90\%$    & $[80.19,  297.50]$ & $[0^*,  134.51]$ & $[0.55 ,   1]$ & $[0.43, 0.87]$ \\ &
$95\%$    & $[76.43,  721.20]$ & $[0^*,  138.27]$ & $[0.52,1]$ & $[0.43,    0.90]$ 
\\
\hline
\end{tabular}}
  \caption{\blue{\textsf{Black robin population}.
  Bootstrap confidence intervals for the carrying capacity $K$ and the efficiency parameter $v$ using the classical estimator (rounded to the second decimal). The intervals were constructed from $M=1000$ bootstrap samples generated from the fitted Beverton-Holt model with parameters $\tilde{\T}_{n}(\boldsymbol{w}_{n}^{(2)}) = (107.3461, 0.7141)$. Three methods were used to generate bootstrap samples: $(i)$ simulating trajectories of the same length as the original data ($n=26$), $(ii)$ simulating trajectories with the same total cumulative number of individuals ($708$), and $(iii)$ simulating trajectories with the same final population size ($86$).
  \label{br_boot_Q_1}}}

\end{table}


	\begin{table} \centering
\blue{\begin{tabular}{|cc|c|c|c|c|}
\cline{3-6}
\multicolumn{2}{c|}{ }&\multicolumn{2}{c|}{$K$} & \multicolumn{2}{c|}{$v$}\\
\hline
& & $[L,U]$ & $[L_p,U_p]$ & $[L,U]$ & $[L_p,U_p]$\\
\hline
\multirow{3}{*}{$(i)$} &
$85\%$    & $[80.03,  150.45]$ & $[68.07,  138.49]$ & $[0.59 ,   1]$ & $[0.40, 0.81]$ \\ &
$90\%$    & $[76.22,  162.56]$ & $[55.95,  142.29]$ & $[0.57 ,   1]$ & $[0.40, 0.82]$ \\ &
$95\%$    & $[70.83,  183.47]$ & $[35.05,  147.68]$ & $[0.54,1]$ & $[0.40,    0.86]$ 
\\
\hline
\multirow{3}{*}{$(ii)$} &
$85\%$    & $[75.31,  173.89]$ & $[44.62,  143.20]$ & $[0.58 ,   1]$ & $[0.40, 0.82]$ \\ &
$90\%$    & $[71.40,  187.53]$ & $[30.99,  147.11]$ & $[0.57 ,   1]$ & $[0.40, 0.83]$ \\ &
$95\%$    & $[65.14,  205.18]$ & $[13.34,  153.37]$ & $[0.54,1]$ & $[0.40,    0.85]$ 
\\
\hline
\multirow{3}{*}{$(iii)$} &
$85\%$    & $[84.36,  187.13]$ & $[31.38,  134.16]$ & $[0.57 ,   1]$ & $[0.40, 0.83]$ \\ &
$90\%$    & $[82.87,  193.86]$ & $[24.66,  135.65]$ & $[0.55 ,   1]$ & $[0.40, 0.85]$ \\ &
$95\%$    & $[78.81,  204.49]$ & $[14.03,  139.70]$ & $[0.53,1]$ & $[0.40,    0.86]$ 
\\
\hline
\end{tabular}}
  \caption{\blue{\textsf{Black robin population}.
  Bootstrap confidence intervals for the carrying capacity $K$ and the efficiency parameter $v$ using the $C$-consistent estimator (rounded to the second decimal). The intervals were constructed from $M=1000$ bootstrap samples generated from the fitted Beverton-Holt model with parameters $\hat{\T}_{n}(\boldsymbol{w}_{n}^{(2)}) = (109.2578, 0.6989)$. Three methods were used to generate bootstrap samples: $(i)$ simulating trajectories of the same length as the original data ($n=26$), $(ii)$ simulating trajectories with the same total cumulative number of individuals ($708$), and $(iii)$ simulating trajectories with the same final population size ($86$).
  \label{br_boot_C_1}}}

\end{table}

}

\newpage

\addcontentsline{toc}{section}{References}


\end{document}